\documentclass[11pt]{amsart}
\usepackage{amsmath,amssymb,amsthm,comment,bm,setspace,mathtools,appendix,textgreek}
\usepackage[utf8]{inputenc}
\usepackage[margin=1in]{geometry}
\usepackage{paralist,tikz-cd}
\usepackage{tikz}
\usetikzlibrary{decorations.pathmorphing,shapes}
\usepackage{euscript}
\usepackage{mathrsfs}

\usepackage[pagebackref]{hyperref}
\hypersetup{colorlinks=true,linkcolor=green!35!black,anchorcolor=green,citecolor=blue,filecolor=blue,menucolor=black,urlcolor=brown}

\usepackage{BOONDOX-uprscr}
\usepackage{mathrsfs}
\usepackage[scr=rsfs,cal=boondox]{mathalfa}

\theoremstyle{plain}

\newtheorem{theorem}{Theorem}[section]
\newtheorem{lemma}[theorem]{Lemma}

\newtheorem{proposition}[theorem]{Proposition}
\newtheorem{corollary}[theorem]{Corollary}

\theoremstyle{remark}
\newtheorem{defn}[theorem]{Definition}
\newtheorem{defnprop}[theorem]{Definition-Proposition}

\newtheorem{remark}[theorem]{Remark}
\newtheorem{example}[theorem]{Example}

\newtheorem{conjecture}[theorem]{Conjecture}
\newtheorem{property}[theorem]{Property}
\newtheorem{assumption}[theorem]{Assumption}

\newcommand{\Hom}{{\mathrm{Hom}}}

\newcommand{\gr}{{\mathrm{gr}}}

\newcommand{\cOmega}{ {\Omega}}

\title{Irregular Hodge numbers of stacky Clarke mirror pairs}
\author{Andrew Harder and Sukjoo Lee}

\address{Andrew Harder 
   \textnormal{Department of Mathematics, Lehigh University, Chandler--Ullmann Hall, 17 Memorial Dr. E., Bethlehem, PA 18015, USA.}    \textnormal{\texttt{anh318@lehigh.edu}}}

   \address{Sukjoo Lee
   \textnormal{Center for Geometry and Physics, Institute for Basic Science (IBS), Pohang 37673, South Korea.}    
   \textnormal{\texttt{sukjoo216@ibs.re.kr}}}

\date{}

\begin{document}

\begin{abstract}
We prove a duality between the graded pieces of the irregular Hodge filtration on the twisted cohomology for a large class of Clarke mirror pairs of stacky Landau--Ginzburg models. We use this to recover results of Batyrev--Borisov, generalize results of Ebeling--Gusein-Zade--Takahashi and Krawitz, and prove results similar to those of Gross--Katzarkov--Ruddat. We apply our results to prove a generalized version of a conjecture of Katzarkov--Kontsevich--Pantev for orbifold toric complete intersections with nef anticanonical divisors and orbifold Fano stacks, and we prove the Hodge number duality result for orbifold log Calabi--Yau complete intersections. Along the way, we study the behaviour of twisted cohomology under degeneration and prove that for certain degenerations of toric Landau--Ginzburg models, irregular Hodge numbers admit a tropical realization. 
\end{abstract}

\maketitle

\section{Introduction}

\subsection{Motivation} One of the first notable results in the mathematical mirror symmetry literature is the monumental proof of Batyrev and Borisov \cite{batyrev1996mirror} that if $X$ and $\check{X}$ are $d$-dimensional Calabi--Yau varieties coming from dual nef partitions, then the following relation holds
\begin{equation}\label{e:bb}
h^{p,q}_\mathrm{st}(X) = h^{d-p,q}_\mathrm{st}(\check{X}).
\end{equation}
Here, the stringy Hodge numbers of $X$ and $\check{X}$ can be defined by motivic integration or by more direct combinatorial means. In the past few decades, there have been a number of results in the same vein. Borisov and Mavlyutov \cite{borisov2003string} give a simplified version of the proof of Batyrev and Borisov; Gross \cite{gross2005toric} combines general results of Gross and Siebert \cite{gross2010mirror} to recover Batyrev and Borisov's result in the smooth case, and, very recently, Matessi and Renaudineau \cite{matessi2024mirror} recover Batyrev and Borisov's result for hypersurfaces in smooth toric varieties using tropical homology.  

Over the past decades, since Batyrev and Borisov's work, it has become clear that mirror symmetry produces relationships between many pairs of varieties which are not Calabi--Yau. For instance, both Fano varieties and varieties of general type are expected to be mirror to objects called \emph{Landau--Ginzburg models}. In plain terms, a Landau--Ginzburg model is a pair consisting of a quasiprojective variety or Deligne--Mumford stack $Y$ along with a regular function $w$ called a \emph{Landau--Ginzburg superpotential}. Furthermore, simple constructions of mirror pairs $X$ and $(Y,w)$ exist in many cases, in particular, when $X$ is a toric complete intersection (see Section \ref{s:smooth} for more details). 

For smooth, ample, general-type hypersurfaces in toric varieties, Gross, Katzarkov, and Ruddat explained how to construct a Landau--Ginzburg mirror $(Y,w)$ and, in a major breakthrough, they proved a result analogous to that of Batyrev and Borisov \cite{gross2017towards}. Precisely, for any such hypersurface $X$, there is a particular singular fibre $Y_0$ of $w:Y\rightarrow \mathbb{A}^1$ so that
\[
    h^{p,q}(X) = \dim \gr_F^{p}\mathbb{H}^{p+d+2-q}(Y_0, \phi_w\underline{\mathbb{Q}}_{\mathcal{Y}}). 
\]
This relationship was predicted by Katzarkov, Kapustin, Orlov, and Yotov \cite{kapustin2009homological}. The proofs of Batyrev--Borisov and Gross--Katzarkov--Ruddat are based heavily on the algorithm of Danilov and Khovanskiǐ along with various combinatorial generating function identities. Both results use the compactness of $X$ in an important way.

If $X$ is a Fano variety, the mirror object is expected to be a Landau--Ginzburg model, $(Y,w)$ consisting of a sufficiently smooth quasiprojective variety $Y$ and a regular function $w$. Katzarkov, Kontsevich, and Pantev construct a cohomology theory for $(Y,w)$ (see also \cite{yu2014irregular,esnault20171}), which we denote $H^*(Y ,w)$. This cohomology theory admits a mixed Hodge structure \cite{esnault20171,shamoto2018hodge} and we denote the $(p,q)$-Hodge number $f^{p,q}(Y,w) = \dim \gr_F^pH^{p+q}(Y,w)$. Katzarkov, Kontsevich, and Pantev conjectured that if $X$ is a $d$-dimensional Fano variety and $(Y,w)$ is a homological mirror dual Landau--Ginzburg model, then 
\begin{equation}\label{eq:kkp}
h^{p,q}(X) = f^{d-p,q}(Y,w).
\end{equation}
To date, only a little progress has been made proving \eqref{eq:kkp}, even in small classes of examples. It is known for weak Fano toric varieties \cite{harder2016geometry,batyrev1993variations}, del Pezzo surfaces \cite{lunts2018landau}, smooth Fano 3-folds \cite{cheltsov2018katzarkov}, and it is known that $h^{d-1,1}(X) = f^{1,1}(Y,w)$ for $d$-dimensional smooth complete intersections in weighted projective spaces \cite{przyjalkowski2015hodge}. The Batyrev--Borisov mirror construction can also be extended to produce mirrors of complete intersections in toric varieties with nef anticanonical bundle (see Section \ref{s:mirror-constr}). 

In \cite{clarke2016dual}, Clarke reinterprets many toric mirror constructions, including those listed above, as a pair of Landau--Ginzburg models coming from a simple duality of fans: For rank $d$-lattice $M$ and its dual $N$, two fans $\Sigma \subseteq M$ and $\check{\Sigma}\subseteq N$ are called \emph{Clarke dual} if $\langle m, n \rangle \geq 0$ for all $m \in \mathrm{Supp}(\Sigma)$ and $n \in \mathrm{Supp}(\check{\Sigma})$. From this data, Clarke \cite{clarke2016dual} constructs a pair of Landau--Ginzburg models which we call {\em Clarke dual pairs}. In this article, we extend this definition to stacky fans to produce more general Clarke dual pairs of stacky Landau--Ginzburg models. 

One expects that that arbitrary Clarke dual pairs are mirror to one another. However, it is difficult to even make a precise conjecture about mirror symmetry between pairs of stacky Landau--Ginzburg models. For instance, little is known about the definition of the Fukaya--Seidel category for stacky Landa--Ginzburg models (see e.g., \cite{cho2022fukayaorb}), and while \cite{katzarkov2008hodge} provides a definition of the Gromov--Witten theory of a Landau--Ginzburg model, there does not seem to be much research in this direction. So at the moment neither Homological nor Hodge-theoretic mirror symmetry seems to be within reach for completely general Clarke dual pairs.

The main goal of this paper is to prove that Clarke dual pairs of Landau–Ginzburg models satisfy a duality of irregular Hodge numbers resembling \eqref{e:bb} and which will be described in detail below. This establishes fundamental evidence that Clarke 
 dual pairs of Landau--Ginzburg models are in fact mirror pairs. In doing so, we are able to completely recover general results of Batyrev--Borisov, Krawitz \cite{krawitz2010fjrw}, partially recover results of Gross--Katzarkov--Ruddat, and prove that \eqref{eq:kkp} holds for all toric complete intersections. We will elaborate on all of this in the remainder of the introduction.

%Our approach, much like that of \cite{gross2017towards}, starts by reinterpreting both $h^{p,q}(X)$ and $f^{p,q}(Y,w)$ as sums of Deligne--Hodge numbers of Landau--Ginzburg models of higher dimension whose total spaces are toric varieties. The question then reduces to understanding relative cohomology of hypersurfaces in toric varieties. As in \cite{matessi2024mirror}, we use tropical techniques to reinterpret these Deligne--Hodge numbers in terms of cohomology groups of sheaves on polyhedral complexes and thus reinterpret the duality in \eqref{eq:kkp} as a duality of polyhedral complexes and tropical sheaves, which can be proved without much difficulty. Despite the use of tropical homology in both cases, our calculations are quite different from those in \cite{matessi2024mirror} and our results seem to be stronger.

%A byproduct of this approach is that we are able to prove a Hodge duality statement for a class of Clarke mirror pairs of stacks -- pairs of toric Landau--Ginzburg models satisfying certain combinatorial conditions. It is well known that Clarke mirror pairs encompass many known combinatorial mirror symmetry constructions. For example, the construction of Batyrev and Borisov, the construction of Berglund, Hübsch, and Krawitz, Givental's construction of dual pairs of Landau--Ginzburg models for toric complete intersections, and our construction of Landau--Ginzburg duals for Fano complete intersection alluded to above. The remainder of this introduction gives more detail about the structure of the results of this paper. 

\subsection{Irregular Hodge numbers of a Landau--Ginzburg model}\label{s:intro-irreg}

As mentioned above, the main focus of this article is a Landau--Ginzburg model $(Y, w: Y \to \mathbb{A}^1)$, which, for us, consists of a smooth Deligne--Mumford stack $Y$ and a regular function $w$. Specifically, we are interested in the cohomology of $(Y, w)$ and the irregular Hodge structure it carries. These objects have been extensively studied over the past decade: In the tame case, when all components of the pole divisor have degree 1, see \cite{katzarkov2017bogomolov, shamoto2018hodge, harder2016geometry}; for extensions to the general case, see \cite{yu2014irregular, esnault20171}.
We briefly recall this background and provide more details in Section \ref{s:tw} to make the article self-contained. For the moment, suppose that $Y$ is a smooth, quasiprojective variety, and $w$ is proper. We adopt the de Rham theoretic description of the cohomology of $(Y, w)$, denoted simply by $H^*(Y, w)$. Fix a normal crossings compactification $\overline{Y}$ of $Y$ and let $P = \overline{Y}\setminus Y$. The {\em twisted de Rham cohomology} is defined as the hypercohomology of the twisted holomorphic de Rham complex 
    \[
    (\cOmega^\bullet_{\overline{Y}}(*P),  d + d{w}\wedge(-)).
    \]
Furthermore, in \cite{yu2014irregular}, Yu introduces a rationally graded decreasing filtration $F^\bullet_\mathrm{irr}$ on the twisted cohomology, called the \emph{irregular Hodge filtration}. The dimensions of the associated graded pieces of this filtration are called \emph{irregular Hodge numbers}, denoted by
\[
f^{\lambda,\mu}(Y, w)=\dim_\mathbb{C}\gr_{F_\mathrm{irr}}^\lambda H^{\lambda+\mu}(Y,w)
\]
for $\lambda,\mu \in \mathbb{Q}$ and $\lambda+\mu \in \mathbb{Z}$.\footnote{To emphasize that the gradings are rational, we use $\lambda$ and $\mu$ instead of $p$ and $q$, which are commonly used to describe the usual Hodge numbers.} In particular, in the tame case, this filtration becomes integrally graded and also coincides with the one introduced in \cite{katzarkov2017bogomolov} (see \cite[Proposition 1.5.3]{esnault20171}). The irregular Hodge numbers are then the same as the Hodge-graded dimensions of the relative cohomology $H^*(Y,w^{-1}(t))$ for a generic fibre $w^{-1}(t)$ (see e.g., \cite{harder2016geometry,shamoto2018hodge}).

In this article, we relax the assumption that $Y$ is smooth and instead work in the orbifold setting, specifically where $Y$ is a toric Deligne--Mumford stack. We consider the so-called \emph{orbifold cohomology} introduced in \cite{ChenRuan2004orb}, which is defined as the usual cohomology of the inertia stack of $Y$ with a particular twist applied to cohomological degree and the irregular Hodge filtration. Following the procedure outlined in Section \ref{s:tw}, we can naturally define the orbifold twisted cohomology and the orbifold irregular Hodge filtration. We denote the orbifold twisted cohomology by $H^*_{\mathrm{orb}}(Y, w)$ and the corresponding irregular Hodge numbers by $f^{\lambda,\mu}_{\mathrm{orb}}(Y, w)$.

% A toric Deligne--Mumford stack is determined by combinatorial data, called a stacky fan, which consists of a pair ${\bf \Sigma} = (\Sigma, \beta)$. Here, $\Sigma$ is a simplicial fan, and $\beta : \Sigma[1] \to \mathbb{Z}_{>0}$ assigns a positive integer to each ray generator of $\Sigma$. The associated toric Deligne--Mumford stack is defined as the quotient stack $T({\bf \Sigma}) = [T(\Sigma) / G_{\beta}]$, where $G_{\beta}$ is the kernel of the induced action of $\beta$ on the torus $(\mathbb{C}^\ast)^{|\Sigma[1]|}$. The underlying toric variety $T(\Sigma)$ is a coarse moduli space of $T({\bf \Sigma})$ so there is a canonical morphism $\pi_{\bf \Sigma}:T({\bf \Sigma}) \to T(\Sigma)$.

%Now, suppose we have a regular function $w$ on a toric variety $T(\Sigma)$, and, for simplicity, we also denote its composition with the canonical morphism by $w$. When the Newton polytope of $w$ is a standard simplex and $T(\Sigma) = \mathbb{C}^{*d}$, the irregular Hodge numbers can be computed in a purely combinatorial manner \cite{as, yu2014irregular}. This computation will serve as a key building block for our main result (see Theorem \ref{t:as})

\subsection{Irregular Hodge number duality for Clarke mirror pairs}

As above, let $N$ be a lattice of rank $d$, and let $M$ be its dual. Let $\langle -, - \rangle: N \times M \to \mathbb{Z}$ denote the canonical nondegenerate pairing. Consider rational fans $\Sigma \subset M$ and $\check{\Sigma} \subset N$, which we assume to be quasiprojective and simplicial. In \cite{clarke2016dual}, Clarke studies a pair of such fans $(\Sigma, \check{\Sigma})$ satisfying the condition $\langle m, n \rangle \geq 0$ for any $m \in \mathrm{Supp}(\Sigma)$ and $n \in \mathrm{Supp}(\check{\Sigma})$. This condition implies that the rational functions
 \[
 w(\check{\Sigma}):=1+\sum_{n \in \check{\Sigma}[1]}a_n\underline{\check{x}}^n , \qquad w({\Sigma}):=1+\sum_{m \in \check{\Sigma}[1]}b_m \underline{x}^m ,
\]
with sufficiently general coefficients are indeed regular on $T(\Sigma)$ and $T(\check{\Sigma})$, respectively. We generalize this definition to the stacky setting. 

For us, a simplicial stacky fan consists of a pair ${\bf{\Sigma}} = (\Sigma, \beta)$ where $\Sigma$ is a simplicial fan and $\beta : \Sigma[1] \to \mathbb{Z}_{>0}$ assigns a positive integer to each ray generator of $\Sigma$. Associated to this data is a smooth toric Deligne--Mumford stack which we denote $T({\bf \Sigma})$. The underlying toric variety $T(\Sigma)$ is a coarse moduli space of $T({\bf \Sigma})$ so there is a canonical morphism $\pi_{\bf \Sigma}:T({\bf \Sigma}) \to T(\Sigma)$. See Section \ref{s:tstack} for a review. 

We define ${\bf{\Sigma}}[1] = \{\beta(\rho) \cdot \rho \mid \rho \in \Sigma[1]\}$, and for each cone $c \in {\bf{\Sigma}}$, let $\Delta_c$ be the convex hull of $c[1] \cup {0}$. We call a pair of stacky fans $({\bf \Sigma}, \check{\bf \Sigma})$ a (stacky) Clarke dual pair if it satisfies the following two conditions:

\begin{enumerate}
    \item (Regularity) $\langle m, n \rangle \geq 0$ for any $m \in \mathrm{Supp}(\Sigma)$ and $n \in \mathrm{Supp}(\check{\Sigma})$, 
    \item (Convexity) The polyhedral complexes $\Delta_{\bf\Sigma}=\cup_{c \in {\bf\Sigma}}\Delta_c $ and $\Delta_{\bf\check{\Sigma}}=\cup_{\check{c} \in {\bf \check{\Sigma}}}\Delta_{\check{c}}$ have convex support.
\end{enumerate}
A Clarke dual pair induces a pair of Landau--Ginzburg models 
\[
(T({\bf \Sigma}), w({\bf \check{\Sigma})}), \qquad (T({\bf \check{\Sigma}}), w({\bf {\Sigma}}))
\]
where $w({\bf \check{\Sigma}})=\pi_{\bf \check{\Sigma}} \circ w({\bf\check{\Sigma}})$ and $w({\bf {\Sigma}})=\pi_{\bf {\Sigma}} \circ w({\bf{\Sigma}})$. We call this pair of Landau--Ginzburg models a \emph{Clarke mirror pair}. Our main theorem establishes the \emph{irregular} Hodge number duality for such a pair, thereby justifying the use of the term "mirror." See Section \ref{s:pfThm} for the outline of the proof. 

\begin{theorem}\label{thm:intro-clarkedual}\label{t:clintro}
	Let $({\bf\Sigma}, {\bf \check{\Sigma}})$ be a Clarke dual pair. For $\lambda,\mu \in \mathbb{Q}$, we have the identification of the orbifold irregular Hodge numbers 
 \[
 f_\mathrm{orb}^{\lambda,\mu}(T({\bf \Sigma}), w({\bf \check{\Sigma}}))=f_\mathrm{orb}^{d-\lambda,\mu}(T({\bf \check{\Sigma}}), w(\bf {\Sigma})).
 \]
\end{theorem}

\subsection{Katzarkov--Kontsevich--Pantev conjecture for toric complete intersections}

The first major application of Theorem \ref{t:clintro} is Hodge number duality for quasi-smooth, log Calabi-Yau, Gorenstein complete intersections of nef divisors in toric varieties. For this application, we will always have $\beta = (1,\dots, 1)$, so we simply write $\Sigma$ instead of ${\bf \Sigma}$. Let us briefly outline the construction of mirror pairs in this context and provide an overview of our proof of Hodge number duality. For more details, we refer the reader to Section \ref{s:smooth}.

Let $\Delta \subset M$ be a reflexive polytope and $T_\Delta$ be the an orbifold crepant resolution of the toric Fano variety attached to $\Delta$. Consider a nef partition $A_1, \dots, A_{k+1}$ of $\Delta$ (Definition \ref{d:nefpar}). There is a union of toric strata $E_i$ attached to each $A_i$ so that $\cup_i E_i$ is the entire toric boundary of $T_\Delta$. Let $\Delta_i=\mathrm{Conv}(A_i, 0)$ and  $\mathcal{O}_{T_\Delta}(E_i)$ be the corresponding line bundle. Choose a generic section $s_i$ of $\mathcal{O}_{T_\Delta}(E_i)$ and take the complete intersection $X_A:=V(s_1, \dots, s_k)$. We also let $D_A$ denote the restriction of the toric divisor $E_{k+1}$ of $T_\Delta$ into $X_A$. Then by adjunction $(X_A, D_A)$ is a log Calabi--Yau pair and we set $U_A=X_A \setminus D_A$. 

Given a nef partition, Borisov \cite{borisov1993towards} explains how to construct a dual nef partition, $\check{\Delta}_1,\dots, \check{\Delta}_{k+1}$ of the polar dual polytope $\check{\Delta}:=\mathrm{Conv}(\check{\Delta}_1, \dots, \check{\Delta}_{k+1})$. As before, we may choose generic sections $\check{s}_{1}, \dots, \check{s}_{k+1}$, of the line bundles $\mathcal{O}_{T_{\check{\Delta}}}(\check{E}_{1}),\dots, \mathcal{O}_{T_{\check{\Delta}}}({E}_{{k+1}})$ and apply the same construction outlined above to obtain a mirror Landau--Ginzburg model $(U_{\check{A}}, w_{\check{A}})$ where $U_{\check{A}}=X_{\check{A}} \setminus D_{\check{A}}$ and $w_{\check{A}}$ is the regular function on $U_{\check{A}}$ induced by $\check{s}_{k+1}$.

\begin{example}
	If we set $A_{k+1}=\{0\}$, then $U_A$ and $U_{\check{A}}$ are the compact Calabi--Yau complete intersections studied in \cite{batyrev1996mirror}. If $k=0$, that is, $\Delta_1 = \Delta$, then $X_A = T_\Delta$, $U_{\check{A}} = (\mathbb{C}^*)^d$ and $w_{\check{A}}$ is the usual Hori--Vafa superpotential.
\end{example}

To apply Theorem \ref{thm:intro-clarkedual}, must incorporate the mirror pair $(X_A, D_A), (U_{\check{A}}, w_{\check{A}})$ into the framework of Clarke duality. This is achieved by using the so-called Cayley trick: Each section $s_i$ induces a regular function $g_i$ on the total space of the dual line bundle $V_i:=\mathrm{Tot}(\mathcal{O}_{T_\Delta}(-E_i))$. Then we show that there exists an $F$-filtered isomorphism between the orbifold cohomology of the Landau--Ginzburg model $(V_A:=V_1\oplus\dots \oplus V_k, g:=g_1\oplus \dots \oplus g_k)$ and $X$ up to Tate twist (Proposition \ref{p:tj2} and Corollary \ref{c:tj2}). Hence, we can replace $X_A$ with a higher-dimensional Landau--Ginzburg model $(V_A, g)$ without losing any Hodge-theoretic information. Similarly, $U_A$ can be replaced with the restriction of $(V_A, g)$ over the open part $T^\circ=T_\Delta \setminus E_{k+1}$, denoted by $(V_A^\circ, g)$. It turns out that both $V_A$ and $V_A^\circ$ are simplicial Gorenstein toric varieties, and we denote the associated fans by $\Sigma_A$ and $\Sigma_{\check{A}}^\circ$, respectively. By applying the same construction on the mirror side, we obtain two Clarke dual pairs

\begin{enumerate}
	\item $(\Sigma_A, \Sigma^\circ_{\check{A}})$ corresponds to the mirror pair $(X_A, (U_{\check{A}}, w_{\check{A}}))$; 
	\item $(\Sigma_A^\circ, \Sigma^\circ_{\check{A}})$ corresponds to the mirror pair  $(U_A, U_{\check{A}})$.
\end{enumerate}
The following result is a direct corollary of Theorem \ref{t:clintro}.

\begin{theorem}[Theorem \ref{t:kkp}]\label{t:functor-intro}
    Let $(X_A, D_A)$ and $(U_{\check{A}}, w_{\check{A}})$ be as above. There are identifications of the orbifold Hodge numbers: For $p,q \in \mathbb{Z}$, 
    \[
    h_{\mathrm{orb}}^{p,q}(X_A)=f_{\mathrm{orb}}^{d-p, q}(U_{\check{A}}, w_{\check{A}}), \quad  f_{\mathrm{orb}}^{p,q}(U_A)=f_{\mathrm{orb}}^{d-p, q}(U_{\check{A}}).
    \] 
\end{theorem}

\subsection{Non-convex Clarke dual pairs}
One can check that Theorem \ref{t:clintro} does not hold when either $\Delta_{\bf \Sigma}$ or $\Delta_{\check{\bf \Sigma}}$ is not convex. However, in mirror symmetry, there are situations where non-convex Clarke dual pairs may appear. In Sections \ref{s:weakFano} and \ref{s:gentype} we introduce what seems to be a novel framework for dealing with such examples. Suppose we have pair of fans $(\Sigma, \check{\Sigma})$ which satisfy the regularity condition but fail to satisfy the convexity condition. We address this issue by introducing appropriate additional data $\beta$ and $\beta'$. In other words, we construct a new {\em stacky} Clarke dual pair $({\bf \Sigma}, {\bf \check{\Sigma}})$ where Theorem \ref{t:clintro} can be applied. In some specific cases, we find that the Hodge numbers for $(T(\Sigma), w(\check{\Sigma}))$ and $(T(\check{\Sigma}), w(\Sigma))$ are characterized as the integer-graded part of the Hodge numbers for $(T({\bf\Sigma}), w({\bf\check{\Sigma}}))$ and $(T({\bf\check{\Sigma}}), w({\bf\Sigma}))$, respectively.

For instance, let $Z$ be a smooth hypersurfaces in a smooth, compact, toric variety $T$. As described in the previous section, there is an induced Landau--Ginzburg model $g_Z : \mathrm{Tot}(\mathcal{O}_{T}(-Z)) \rightarrow \mathbb{A}^1$ which admits a natural choice of Clarke mirror Landau--Ginzburg model which we denote $(\check{T}, \check{w})$ in this introduction (see Section \ref{s:gentype} or \cite{gross2010mirror} for details). If $Z$ does not have nef anticanonical bundle, this pair of Landau--Ginzburg models does not satisfy the convexity condition. However, we may replace $\mathrm{Tot}(\mathcal{O}_T(-Z))$ with an appropriate root stack $\sqrt[n]{\mathrm{Tot}(\mathcal{O}_T(-Z)/0_T}$ for a sufficiently large $n$, where $0_T$ indicates the zero section of the bundle $\mathrm{Tot}(\mathcal{O}_T(-Z))\rightarrow T$. Composing $g_Z$ with the canonical morphism $\sqrt[n]{\mathrm{Tot}(\mathcal{O}_T(-Z))/0_T} \rightarrow \mathrm{Tot}(\mathcal{O}_T(-Z))$ we obtain a new stacky Landau--Ginzburg model, $(\sqrt[n]{\mathrm{Tot}(\mathcal{O}_T(-Z))/0_T}, {g}_Z)$ whose Clarke mirror has the same underlying space, $\check{T}$, as the Clarke dual of $(\mathrm{Tot}(\mathcal{O}_T(-Z))/0_T,g_Z)$, but the modified potential function, which we will denote by $\check{\bm{w}}$ in this introduction. Then the associated stacky Clarke dual pair becomes convex, and induces a Clarke mirror pair of Landau--Ginzburg models 
\[
(\sqrt[n]{\mathrm{Tot}(\mathcal{O}_T(-Z))/0_T},{g}_Z),\qquad (\check{T}, \check{\bm w})
\]
We also prove that 
\[
f^{p,q}_\mathrm{orb}(\sqrt[n]{\mathrm{Tot}(\mathcal{O}_T(-Z))/0_T},{g}_Z) = h^{p-1,q-1}(Z),\qquad p,q \in \mathbb{Z}.
\]
Thus we obtain the following result by applying Theorem \ref{t:clintro}.
\begin{theorem}[Theorem \ref{t:gkr}]
    Let $Z$ be a smooth toric hypersurface in a $d$-dimensional toric variety $T$. Then 
    \[
    h^{p-1,q-1}(Z) = f_\mathrm{orb}^{d-p,q}(\check{T},\check{\bm w})
    \]
    for all $p,q \in \mathbb{Z}$.
\end{theorem}
Note that under the conditions listed above, $h^{p-1,q-1}(Z) = 0$ if $p,q \notin \mathbb{Z}$ but $f^{d-p,q}(\check{T},\check{\bm  w})$ can be nonzero if $p,q$ are not integers. Therefore, Theorem \ref{t:gkr} only finds a source for the {\em integral} irregular Hodge numbers of $(\check{T},\check{\bm w})$.
\begin{remark}
    It is interesting to compare this to the results of Gross, Katzarkov, and Ruddat in \cite{gross2010mirror}, who prove that if $Z$ is of general type and $\check{T}$ is smooth, then 
    \[
    h^{p-1,q-1}(Z) = f^{d-p,q}(\check{T}, \phi_{\check{w}}\mathbb{C}) := \dim \gr^F_{d-p} \mathbb{H}^{d-p+q}(\check{T},\phi_{\check{w}}\mathbb{C}).
    \]
    We expect, but have not checked, that $f^{p,q}(\check{T}, \phi_{\check{w}}\mathbb{C}) = f^{p,q}(\check{T}, \phi_{\check{\bm w}}\mathbb{C})$, which would imply that for all integers, 
    \begin{equation}\label{e:irvan}
    f^{p,q}(\check{T}, \phi_{\check{\bm w}}\mathbb{C}) = f^{p,q}(\check{T},\check{\bm w}).
    \end{equation}
    It has been observed (see e.g. \cite[Remark 34]{lunts2018landau}) that, in general, Hodge numbers of vanishing cycle sheaves of Landau--Ginzburg models are not related to the irregular Hodge numbers, so \eqref{e:irvan} indicates that the Landau--Ginzburg models discussed in this section might have special properties that warrant further study.
\end{remark}

\subsection{Outline of the proof of Theorem \ref{thm:intro-clarkedual}}\label{s:pfThm}

The proof of Theorem \ref{thm:intro-clarkedual} may be roughly divided into two parts. The first part requires proving a comparison theorem between the irregular Hodge numbers of Landau--Ginzburg models  and what we call \emph{tropical irregular Hodge numbers}. These are combinatorial invariants attached to the fans $({\bf \Sigma},\check{\bf \Sigma})$. Once this result is proved, Theorem \ref{thm:intro-clarkedual} becomes a purely combinatorial question which we are able to address without appealing algebraic geometry.

\subsubsection{Comparison results} Suppose we have a Clarke dual pair of fans $({\bf \Sigma}, \check{\bf \Sigma})$. Since we assume that ${\bf \Sigma}$ and $\check{\bf \Sigma}$ are quasiprojective, there is a ${\bf \Sigma}$-linear function $\varphi: {\bf \Sigma} \rightarrow \mathbb{Z}$, which allows us to extend the Landau--Ginzburg model $(T(\check{\Sigma}),w(\bf {\Sigma}))$ to a family of Landau--Ginzburg models over a punctured disc $\mathbb{D}^*$, $w({\bf \Sigma})_\varphi: T(\check{\bf \Sigma}) \times \mathbb{D}^*\rightarrow \mathbb{A}^1$ using the Laurent polynomial 
\[
w({\bf \Sigma})_\varphi = 1 + \sum_{m \in {\bf \Sigma}[1]} b_m t^{\varphi(m)}\underline{x}^m.
\]
In Sections \ref{s:combdeg} and \ref{s:polysub} we show that there is a partial compactification  of $T(\check{\bf \Sigma})\times \mathbb{D}^*$ to which $w({\bf \Sigma})$ extends to a rational function, and so that the projection $T(\check{\bf \Sigma}) \times \mathbb{D}^* \rightarrow \mathbb{D}^*$ extends to a proper morphism to $\mathbb{D}$. In fact this becomes what we call a {\em quasi-stable degeneration of Landau--Ginzburg models} (Definition \ref{d:deglg}). In Section \ref{s:degnby} we study the cohomology of quasi-stable degenerations of Landau--Ginzburg models. Following a modification of the classical approach of Steenbrink \cite{steenbrink1974mixed} we define a log relative twisted de Rham complex for quasi-stable degenerations of Landau--Ginzburg models (Section \ref{s:nbf}) and show that the log relative twisted de Rham complex computes the nearby cohomology of the twisted de Rham complex of a degenerating family of Landau--Ginzburg models (Theorem \ref{t:constdim}). We equip the nearby cohomology with an irregular Hodge filtration and prove that for toric degenerations of Landau--Ginzburg models, the irregular Hodge numbers are constant under degenerations (Theorem \ref{t:nearbyfibcoh}(1), Proposition \ref{p:hnconst}) 

Given a quasi-stable degeneration of $(T(\check{\bf \Sigma}), w({\bf\Sigma}))$ we construct a tropical polyhedral complex $\mathsf{T}(\check{\bf \Sigma}, w({\bf \Sigma})_\varphi)_0$ which indexes pairs of strata in the central fibre. Motivated by work of Itenberg, Katzarkov, Mikhalkin, and Zharkov \cite{itenberg2019tropical}, we build a graded sheaf $({\bf J}_\mathrm{orb}^\bullet,\mathsf{F}^\bullet)$ of filtered vector spaces on $\mathsf{T}(\check{\bf \Sigma}, w({\bf \Sigma})_\varphi)_0$ in Section \ref{s:tsheaf}, which we call the \emph{orbifold Jacobian sheaf}. We prove the following result.

\begin{theorem}[Theorem \ref{t:nearbyfibcoh}, Proposition \ref{p:hodgeorb}]\label{t:degintro}
    Let $(T(\check{\bf \Sigma}), w({\bf \Sigma}))$ be a toric Landau--Ginzburg model coming from a Clarke dual pair of fans $({\bf \Sigma}, \check{\bf \Sigma})$. Then for $\lambda, \mu \in \mathbb{Q}$,
    \[
    f^{\lambda,\mu}_\mathrm{orb}(T({\bf\check{\Sigma}}),w({\bf \Sigma})) = \sum_{k} \dim \gr_{\mathsf{F}}^\lambda H^{\lambda+\mu-k}(\mathsf{T}(\check{\bf \Sigma}, w({\bf \Sigma})_\varphi)_0, {\bf J}_\mathrm{orb}^k).
    \]
\end{theorem}
In \cite{itenberg2019tropical}, the authors use results in classical Hodge theory to prove a similar result. Notably, they use the $E_2$ degeneration of the Steenbrink spectral sequence in a crucial way. Since we do not have access to these tools, we are required to construct our tropical complexes by hand, and prove a particular degeneration result (Theorem \ref{t:nearbyfibcoh}(2)) directly. This also suggests an alternate approach to results of \cite{itenberg2019tropical} which we will explore in forthcoming work \cite{descent1}.

\subsubsection{Combinatorial results}
The constructible sheaf ${\bf J}_\mathrm{orb}^\bullet$ on $\mathsf{T}(\check{\bf \Sigma}, w({\bf \Sigma})_\varphi)_0$ can be viewed as a sheaf on the underlying poset equipped with its Alexandrov topology (Section \ref{s:tropsheaves}). In Proposition \ref{p:cover}, we show that there is a basis of open subsets of the Alexandrov topology of $\mathsf{T}(\check{\bf \Sigma}, w({\bf \Sigma})_\varphi)_0$ which is in bijection with the fan
\[
({\bf \Sigma} \oplus \check{\bf \Sigma})_0 = \{(c, \check{c}) \in {\bf \Sigma} \times \check{\bf \Sigma} \mid \langle c,\check{c}\rangle=0\}
\]
considered as a poset. By symmetry the same fact is true for $\mathsf{T}({\bf \Sigma}, w(\check{\bf \Sigma})_\varphi)_0$. Let $\check{\bf J}_\mathrm{orb}^\bullet$ denote the Jacobian sheaf on $\mathsf{T}({\bf \Sigma}, w(\check{\bf \Sigma})_\varphi)_0$. Not only do we have that $\mathsf{T}({\bf \Sigma}, w(\check{\bf \Sigma})_\varphi)_0$ and $\mathsf{T}(\check{\bf \Sigma}, w({\bf \Sigma})_\varphi)_0$ are topologically dual to one another, but we have that the stalks ${\bf J}_\mathrm{orb}^\bullet$ and $\check{\bf J}^\bullet_\mathrm{orb}$ are the same after rearranging various gradings. Taking cohomology, we see that, a version of Theorem \ref{t:clintro} holds for tropical orbifold cohomology groups (Theorem \ref{t:tropidual}). Combining this with Theorem \ref{t:degintro}, we obtain Theorem \ref{t:clintro}.

\subsection{Structure of the article}
In Section \ref{s:2}, we provide some background on toric varieties and orbifold (Chen--Ruan) cohomology. Section \ref{s:tw} collects background on twisted cohomology and the irregular Hodge filtration. The main result of Section \ref{s:tw} is Proposition \ref{p:adolphson-sperber}, which refines a result of Adolphson and Sperber \cite{as}. Section \ref{s:thyper} discusses tropical twisted cohomology. This section is heavily inspired by \cite{itenberg2019tropical} however, our description of the tropical Jacobian sheaf for tropical twisted cohomology in Sections \ref{t:thomsheaf} and \ref{s:torbsheaf} seem to be new. Section \ref{s:clamp} introduces Clarke dual pairs of fans and proves three related numerical dualities for them; the first is purely on the level of fans, and has similar flavour to work of Borisov and Mavlyutov \cite{borisov2014stringy}. The second is on the level of tropical twisted cohomology, and the final result is geometric. These results lead to three major applications:
\begin{itemize}
    \item (Section \ref{s:bhk}) Hodge number duality for possibly degenerate Berglund--H\"ubsch--Krawitz mirrors;
    \item (Section \ref{s:smooth}) Hodge number duality for toric nef complete intersections (known as the Katzarkov--Kontsevich--Pantev conjecture \cite{cheltsov2018katzarkov}) and log Calabi--Yau complete intersections; 
    \item (Section \ref{s:non-convex}) Hodge number duality for arbitrary smooth toric hypersurfaces. 
\end{itemize}

The last two sections contain the most technical parts of this article. Section \ref{s:degnby} shows that twisted cohomology behaves well under semistable degeneration. Precisely, we construct a nearby twisted de Rham complex for semistable degenerations and show that the rank of the nearby twisted cohomology agrees with the rank of the twisted cohomology on a general fibre of the degenerating family. We also construct an irregular Hodge filtration on the nearby twisted de Rham complex. The results in this section can be readily extended to the orbifold setting and also may be of independent interest. Section \ref{s:nearbyfib} constructs semistable degenerations of toric Landau--Ginzburg models and shows that under certain combinatorial assumptions, the irregular Hodge filtration on the nearby twisted de Rham complex degenerates at the initial term, and that the nearby cohomology is filtered isomorphic to the cohomology of the Jacobian sheaf.

\subsection{Acknowledgements}
The authors would like to thank L. Borisov, K. Chan, A. Corti, C. Doran, D. Favero, L. Katzarkov, T. Kelly, M. Kerr, D. Kozevnikov, N.C. Leung, T. Pantev, and N. Sheridan for helpful discussions at various stages of this work. AH was supported by a Simons Foundation Travel Grant. SL was supported by the Leverhulme Prize award from the Leverhulme Trust and the Enhancement Award from the Royal Society University Research Fellowship, both awarded to Nick Sheridan. He was also supported by the Institute for Basic Science (IBS-R003-D1).

\tableofcontents

\section{Background}\label{s:2}
In this section, we collect various background material on toric geometry and orbifolds, with the intent of setting notation and reviewing necessary facts and definitions.

\subsection{Toric varieties}
We recollect some background about toric varieties to set up notation. We refer for more details to \cite{Cox2011toricbook}. Let $N$ and $M$ be dual lattices of rank $d$ with the natural bilinear pairing $\langle -,- \rangle:N \times M \to \mathbb{Z}$. We write $N_\mathbb{K}:=N \otimes_\mathbb{Z} \mathbb{K}$ and $M_\mathbb{K}=M \otimes_\mathbb{Z} \mathbb{K}$ for $\mathbb{K}=\mathbb{Q}, \mathbb{R}, \mathbb{C}$. A rational convex polyhedral cone (simply called \textit{cone}) $c$ in $N_\mathbb{R}$ is a convex cone generated by finitely many vectors in $N$. A cone $c$ is {\em strongly convex} if $c \cap (-c) = 0$. Associated to a strongly convex cone $c$, one can construct an affine toric variety $T(c):=\mathrm{Spec}(\mathbb{C}[c^\vee \cap M])$ where $c^\vee \subset M_\mathbb{R}$ is a dual cone of $c$ defined by 
\begin{equation*}
c^\vee:=\{m \in M_\mathbb{R}\mid \langle m,n\rangle \geq 0 \text{ for all } n \in c\}.
\end{equation*}
Such affine toric varieties can be glued to produce more general toric varieties. This gluing data is combinatorially encoded in a fan $\Sigma \subset N_\mathbb{R}$ which is a collection of strongly convex polyhedral cones such that 
\begin{enumerate}
	\item each face of a cone in $\Sigma$ is also a cone in $\Sigma$,
	\item the intersection of two cones in $\Sigma$ is a face of each cone.
\end{enumerate}

% If we assign the data of a positive integer $\beta_\rho$ to each primitive integral ray generator $\rho$, we obtain the data of a {\em stacky fan}\footnote{this description is a simplified one.}\textcolor{red}{Cite-stacky fan, extended stacky fan}. We denote this information ${\bf \Sigma}$.  
Given a fan $\Sigma$, we define the toric variety $T(\Sigma)$ by gluing the affine toric varieties $T(c)$: Two affine toric varieties $T(c)$ and $T(c')$ are glued over their intersection $T({c \cap c'}):=\mathrm{Spec}(\mathbb{C}[(c \cap c')^\vee \cap M])$. Globally, the toric variety $T(\Sigma)$ is also constructed as a quotient space. Let $\Sigma[1]=\{\rho_i\mid i=1, \dots, n\}$ be the set of primitive generators of $\Sigma$. Take the monomial ideal of $\mathbb{C}[x_1, \cdots, x_n]$, $J_\Sigma:=\langle \prod_{\rho_i \subsetneq \sigma} x_i \mid \sigma \in \Sigma \rangle$ and the induced quasi-affine variety $\mathbb{C}^n \setminus V(J_\Sigma)$. We have the morphism of lattices $\beta:\mathbb{Z}^n \to N$ that sends the standard basis $e_i$ to $\rho_i$. Take the induced morphism on the duals $\beta^\vee:M=\mathrm{Hom}(N, \mathbb{Z}) \to \mathrm{Hom}(\mathbb{Z}^n, \mathbb{Z})$. By taking the Cartier dual $\mathrm{Hom}(-, \mathbb{C}^*)$ on $\beta^\vee$, we have the morphism of tori $T_{\beta}:(\mathbb{C}^*)^n \to (\mathbb{C}^*)^d$. We let $G_{\beta}$ be its kernel. Then the associated toric variety $T({\Sigma})$ is given by the quotient $\mathbb{C}^n \setminus V(J_\Sigma)/G_{\beta}$ where $G_{\beta}$ acts freely via the action of $(\mathbb{C}^*)^n$.

If the support $|\Sigma|$ of the fan $\Sigma$ is $N_\mathbb{R}$, then it is called \textit{complete}, and the corresponding toric variety $T(\Sigma)$ is compact. Any toric variety is stratified by tori $T_c$ of dimension $(d-\dim c)$ corresponding to the cones of $\Sigma$. We let $T(\Sigma)_c$ denote the closure of $T_c$ in $T(\Sigma)$. In particular, each one dimensional cone $c$, or its ray generator $\rho$ determines the torus-invariant divisor, which we also denote by $E_\rho = T(\Sigma)_c$.

\begin{defnprop} 
 A cone $c$ is called \emph{simplicial} (resp. \emph{unimodular}) if the primitive integral ray generators of $c$ forms a basis of $N_\mathbb{R}$ (resp. $N_\mathbb{Z}$). A simplicial (resp. unimodular) fan is a fan whose maximal cones are simplicial (resp. unimodular). If $\Sigma$ is simplicial (resp. unimodular), then the associated toric variety $T(\Sigma)$ is an orbifold (resp. manifold). 
 
\end{defnprop}
% The data of a simplicial stacky fan ${\bf \Sigma}$ determines a toric Deligne--Mumford stack which we denote $T({\bf \Sigma})$. The underlying toric variety $T(\Sigma)$ is the coarse moduli space for $T({\bf \Sigma})$. Any simplicial fan $\Sigma$ has a canonical stacky fan structure where $\beta_\rho = 1$ for all $\rho$.
\begin{defnprop}
A cone $c$ in $N_\mathbb{R}$ is called \emph{Gorenstein} if there is some $m_c\in M$ so that the integral collection of points $c \cap \{n \mid \langle n, m_c \rangle = 1\}$ generates the cone $c$. A Gorenstein fan is a fan whose cones are all Gorenstein. If $\Sigma$ is Gorenstein, then the associated toric variety $T(\Sigma)$ is Gorenstein. 
\end{defnprop}

\begin{defn}
    A toric variety is \emph{quasiprojective} if there is a convex function on $\mathrm{Supp}(\Sigma)$ which is linear on each cone of $\Sigma$ and takes integral values on $N \cap \mathrm{Supp}(\Sigma)$. 
\end{defn}

% A torus-invariant divisor $E=\sum_{\rho \in \Sigma[1]} a_\rho E_\rho$ is Cartier if and only if there is some piecewise linear function $\varphi$ on $\mathrm{Supp}(\Sigma)$ which islinear on the cones of $\Sigma$ and which takes integral values on $N \cap  \mathrm{Supp}(\Sigma)$. 

A rational convex polytope $\Delta$ in $M_\mathbb{R}$ is the convex hull of a finite number of rational points. We say $\Delta$ is a \textit{lattice polytope} if every vertex of $\Delta$ is in $M$. For example, a lattice polytope is given by
\begin{equation*}
\Delta=\{m \in M_\mathbb{R}\mid\langle m, n_i \rangle \geq -a_i, n_i \in N, a_i \in \mathbb{Z} \text{ for }i=1,\dots, s\}
\end{equation*}
 An \textit{$l$-face} $\sigma$ is the intersection of $\Delta$ with $d-l$ supporting hyperplanes, and we will denote it by $\sigma \prec \Delta$. In particular, a $0$-face, a $1$-face and a $(d-1)$-face are called a vertex, an edge and a facet of $\Delta$, respectively. For each face $\sigma \prec \Delta$, we define a (inward) normal cone $\mathrm{nc}(\sigma)$ to $\sigma$ as
\begin{equation*}
\mathrm{nc}(\sigma):=\{u \in N_\mathbb{R}\mid\langle v, u\rangle \leq \langle v', u \rangle \text{ for all } v \in \sigma \text{ and } v' \in \Delta\}
\end{equation*}
When $\dim(\Delta) = \mathrm{rk}(M)$, the union of such normal cones forms a fan in $N$, called the normal fan of $\Delta$, denoted by $\mathrm{nf}(\Delta)$. The corresponding toric variety is written as $T_\Delta = T(\mathrm{nf}(\Delta))$.

\begin{remark}\label{r:nonpointed}
   If $\dim(\Delta) < \mathrm{rk}(M)$, the normal cone $\mathrm{nc}(\sigma)$ for each $\sigma \prec \Delta$ is not strongly convex, so the induced normal fan is not a fan in the usual sense. Instead, it should be viewed as a non-pointed fan as follows: Let $M_\Delta$ be the minimal sublattice of $M$ containing the vertices of $\Delta$, and denote its dual by $N_\Delta$. Viewing $\Delta$ as a polytope in $M_\mathbb{R}$, we obtain the normal fan, denoted by $\mathrm{nf}(\Delta, N_\Delta)$. The non-pointed normal fan $\mathrm{nf}(\Delta)$ is then given by the product $\mathrm{nf}(\Delta, N_\Delta) \times N^\perp_{\Delta, \mathbb{R}}$.
\end{remark}

% To a lattice polytope $\Delta$, we denote the associated toric variety by $T_\Delta:=T({\mathrm{nf}(\Delta)})$.

% A polytope $\Delta \subset M_\mathbb{R}$ is called \epmh{simplicial}, if there are exactly $d$ edges at each vertex and the primitive vectors at each vertex span $M_\mathbb{R}$ as a vector space. A fan $\Sigma$ in $N_\mathbb{R}$ is \emph{simplicial} if all the maximal cones in $\Sigma$ is simplicial. In particular, if the primitive vectors span the lattice, then it is called \textit{non-singular}.
% \begin{proposition}
% 	If $\Delta$ is simplicial (resp. non-singular), then $T({\mathrm{nf}(\Delta)})$ is an orbifold (resp. manifold).
% \end{proposition}

Suppose that we have a finite collection of integral points $A$ in $M$. We consider a Laurent polynomial $w \in \mathbb{C}[M]$ of the form 
\[
w=\sum_{m \in A}u_m\underline{x}^m .
\]
for some $u_m \in \mathbb{C}$. The Newton polytope of $w$, denoted by $\Delta(w)$, is defined to be the convex hull of $A$ in $M_\mathbb{R}$. For any toric variety $T(\Sigma)$, one can see that $w$ becomes a regular function on $T(\Sigma)$ if $\langle m, \rho \rangle \geq 0$ for every primitive integral ray generator $\rho \in \Sigma[1] $ and $m \in A$.

For any face $\Delta' \subset \Delta(w)$, we write $w_{\Delta'}=\sum_{m \in A \cap \Delta'}u_m\underline{x}^m$. We say that $w$ is \emph{nondegenerate} if for any face $\Delta' \subset \Delta(w)$ not containing $0$, the $d$ partial derivatives $\{\frac{\partial w_{\Delta'}}{\partial x_i}\mid i=1, \cdots, d \}$ have no common solutions. In this case, the hypersurface $X_w=\{w=0\}$ intersects with the affine torus $T_c$ for each cone $c \in \Sigma$ smoothly.

% In this case, the vanishing locus determines the hypersurface $X_f=\{f=0\}$ in $T({\bf \Sigma})$. We say that $X_f$ is \emph{non-degenerate} if the underlying affine subvariety $X_f \cap T_c$ is a smooth hypersurface in the affine torus $T_c$ for any cone $c \in {\bf \Sigma}$. In particular, when $T$ is an orbifold, every non-degenerate hypersurface becomes \emph{quasi-smooth}.

\subsection{Toric Deligne--Mumford stacks}\label{s:tstack}

We recall the definition of a toric Deligne--Mumford stack introduced in \cite{borisov2005orbifold}. Note that our definition is less general than that of {\em op. cit.}. \footnote{The stacky fans discussed in this paper are a proper subset of those introduced in \cite{borisov2005orbifold}, where the finitely generated abelian group $N$, as described in \textit{op. cit.}, has no torsion. For a more general treatment of toric stacks, we refer to \cite{Geraschenko2011ToricSI}.} 

A \emph{stacky fan} is a simplicial fan $\Sigma$ equipped with additional data: a positive integer $\beta_\rho$ assigned to each primitive integral ray generator $\rho \in \Sigma[1]$. We denote this as ${\bf \Sigma} = (\Sigma, \beta)$, where $\Sigma$ is referred to as the underlying fan of ${\bf \Sigma}$. The assignment $\beta$ can be viewed as a lattice morphism $\beta:\mathbb{Z}^n \to N$ that sends the standard basis $e_i$ to $\beta_\rho \cdot \rho$. Similar to the global construction of toric varieties, it induces the morphism of tori $T_{\beta}:(\mathbb{C}^*)^n \to (\mathbb{C}^*)^d$. Since $\beta$ has finite cokernel, $T_{\beta}$ is surjective. Let $G_{\beta}$ be its kernel. 

\begin{defnprop}
    A toric Deligne--Mumford stack $T(\mathbf{\Sigma})$ is the quotient stack 
    \[
    \left[(\mathbb{C}^n \setminus V(J_\Sigma))/G_{\beta}\right]
    \]
    where $G_{\beta}$ acts via the action of $(\mathbb{C}^*)^n$.
\end{defnprop}
The underlying toric variety $T(\Sigma)$ is the coarse moduli space for $T({\bf \Sigma})$ so that we have a canonical morphism $\pi_{\bf\Sigma}:T({\bf \Sigma}) \to T(\Sigma)$. Any simplicial fan $\Sigma$ has a canonical stacky fan structure where $\beta_\rho = 1$ for all $\rho \in \Sigma[1]$. For a stacky fan ${\bf \Sigma}$, let ${\bf \Sigma}[1]$ denote the set $\{\beta_\rho \rho \mid \rho \in \Sigma[1]\}$. We often write $c \in {\bf \Sigma}$ which means that we take ray generators of the cone $c$ to be extended ones $\{\beta_\rho \rho|\rho \in c\}$.   

Some of the properties of a fan naturally extend to a stacky fan as properties of the underlying fan: ${\bf \Sigma}$ is called unimodular, simplicial, or quasiprojective if its underlying fan $\Sigma$ has these properties. An exception is the Gorenstein property, where we say ${\bf \Sigma}$ is Gorenstein if every cone $c$ in ${\bf \Sigma}$ is Gorenstein.

\subsection{Orbifolds and their cohomology}\label{s:orbcoh}

The main geometric objects considered in this paper are orbifolds in particular, we will consider toroidal orbifolds. Such objects admit both an underlying analtyic space and an overlying Deligne--Mumford stack. With only minor caveats, the analytic space of a toroidal orbifold behaves like a smooth manifold. This fact allows us to apply de Rham theoretic methods.

\begin{defn}
    A $d$-dimensional complex toroidal orbifold is an analytic space $X$ along with local analytic charts near each point $p \in X$ which identify, holomorphically, a neighbourhood of $p$ with a neighbourhood of the origin in the quotient of a polydisc $\mathbb{D}^d/G$ for a finite diagonal subgroup $G$ of $\mathrm{GL}_d(\mathbb{C}).$ 
\end{defn}
To each point $p$ in an orbifold there is an {\em isotopy group} $G_p$, which is the subgroup of $G$ in a local chart which fixes $p$. A comdimension $k$ subvariety $V$ of an orbifold $X$ is quasi-smooth if for each point $p \in V$ there is a local orbifold chart $\mathbb{D}^d/G \cong \mathbb{D}^k \times  \mathbb{D}^{d-k}/G$ on $X$ so that $V = 0 \times \mathbb{D}^{d-k}/G$. Therefore $V$ is itself an orbifold and the isotopy group of $p \in X$ is the same as the isotopy  group of $p$ in $V$.

Let $X$ be an orbifold, covered by open charts $(U, G_U)$. To define orbifold cohomology of $X$, we introduce the inertia stack $\mathfrak{I}_X$ of $X$, the set of pairs
\[
(p, g) \in X \times \mathrm{conj}(G_p)
\]
where $G_p$ denotes the stabilizer of the point $p$ in $G_U$ and $\mathrm{conj}(G_p)$ is the set of all conjugacy classes of $G_p$. The inertia stack is itself an orbifold. Locally, let $U^g$ be the fixed-point set of some $g \in G_U$. Then $C(g)$ acts on $U^g$. Let $K_g$ be the kernel of the action of $C(g)$ on $U^g$. Then we define the charts
\begin{align*}
    (U^g, C(g)/K_g) & \longrightarrow \mathfrak{I}_X  \\
    p & \longmapsto (p,g).
\end{align*}
Given a chart $U$ centred at a point $p$, and $q \in U$ there is an injective homomorphism from $G_q$ to $G_p$ and thus we may declare that $g \in G_q$ is equivalent to $g' \in G_p$ if the conjugacy classes of $g$ and $g'$ are the same in $G_p$. In the case of interest to us, all groups are abelian so this is simple equality under inclusion. Let $I_X$ denote the collection of all such equivalence classes. For each $g \in I_X$, we have 
\[
X_{(g)} = \{(p,g) \mid g \in I_X\}.
\]
This can be considered itself as a sub-orbifold of $X$ and an irreducible component of $\mathfrak{I}_X$. Thus we have that 
\[
\mathfrak{I}_X = \coprod_{g \in I_X} X_{(g)}.
\]
The irreducible components $X_{(g)}$ are called the \emph{twisted sectors} of $X$. Given any $g\in I_X$, we can, up to conjugacy, write the action of $g$ as
\[
\mathrm{diag}(\exp(2\pi {\tt i}m_1/m_g),\dots , \exp (2\pi {\tt i}m_d/m_g)  )
\]
for unique integers $0 \leq m_i < m_g$. This expression is an invariant in $I_X$. We define the {\em age} of $g\in I_X$ to be a rational number
\[
\iota{(g)} = \sum_{i=1}^d \left( \dfrac{m_i}{m_g}\right)
\]
In particular, $\iota{(g)} \in \mathbb{N}$ if the group $G_p$ is contained in $\mathrm{SL}_d(\mathbb{C})$, which is equivalent to being Gorenstein at $p$. 

\begin{defn}
    The orbifold cohomology of $X$ is defined to be the cohomology of $\mathfrak{I}_X$ with a particular shift in grading,
\begin{equation*}
H^*_\mathrm{orb}(X) := \bigoplus_{g\in I_X} H^{*-2\iota(g)}(X_{(g)}).
\end{equation*}
\end{defn}
\noindent Note that the orbifold cohomology of $X$ is a rationally graded vector space.

% Similarly, if $Y$ is a subvariety of an orbifold $X$, we define the relative orbifold cohomology as
% \begin{equation}\label{eq:relorbgrad}
%     H^*_\mathrm{orb}(X,Y) = \bigoplus_{g\in \Xi} H^*(X_{(g)},X_{(g)}\cap Y)\otimes \mathbb{Q}(-\iota(g))
%     \end{equation}
% which also admits a canonical mixed Hodge structure.

\begin{remark}\label{r:derhamorb}
    The de Rham cohomology of an orbifold $X$ is defined via its non-singular locus. Let $X^\circ$ denote the complement of the singular locus in $X$. We define ${\Omega}^p_{X}$ as $i_*\Omega^p_{X^\circ}$, where $i : X^\circ \hookrightarrow X$ is the natural inclusion, and $\Omega_{X^\circ}^p$ is the usual de Rham complex of $X^\circ$. The stalk of $\Omega_X^p$ at a point $a$ may be identified in local coordinates $(\mathbb{D}^d,G)$ with $G$-invariant local sections of $\Omega^p_{\mathbb{D}^d}$ at $a$. The induced complex $(\Omega_X^\bullet, d_{\mathrm{dR}})$, equipped with the usual de Rham differential, forms a resolution of $\underline{\mathbb{C}}_X$, thus providing a definition for the de Rham cohomology. 
    
    The mixed Hodge structure on $H^*(X)$ is then defined in the usual way. For instance, according to \cite[Corollary 1.9]{steenbrink1974mixed}, when $X$ is projective, the filtration on ${\Omega}_X^\bullet$ induced by the stupid filtration degenerates at the $E_1$ term, and the resulting filtration coincides with the Hodge filtration. Similarly, a mixed Hodge structure can be defined on the orbifold cohomology $H^*_{\mathrm{orb}}(X)$, where the shift by the age corresponds to the Tate twist.
\end{remark}

One can extend the definition of orbifold cohomology to a separated Deligne--Mumford stack $\mathfrak{X}$ as the cohomology of the inertia stack  $\mathfrak{I}_\mathfrak{X}$ with a particular shift as before. Note that the separatedness implies that the number of irreducible components of the inertia stack $\mathfrak{I}_\mathfrak{X}$ is finite. We mainly deal with smooth toric Deligne--Mumford stacks introduced in Section \ref{s:tstack}.

\subsection{Orbifold cohomology of toric Deligne--Mumford stacks}\label{s:cohdmst} 
Let ${\bf \Sigma}$ be a simplicial stacky fan and $T({\bf \Sigma})$ be the associated toric Deligne--Mumfold stack. For each cone ${c}$ in $\bf{\Sigma}$, we let
\begin{equation}\label{e:box}
    \mathrm{Box}^\circ({c}) := \left\{ \left. \sum_{\rho\in {c}[1]} a_\rho \rho \,\, \right| \,\, a_\rho \in (0,1) \right\}\cap N, \qquad \mathrm{Box}({c}) := \left\{ \left. \sum_{\rho\in {c}[1]} a_\rho \rho \,\, \right| \,\, a_\rho \in [0,1) \right\}\cap N.
\end{equation}
Then the twisted sectors of the inertia stack $\mathfrak{I}_{T({\bf \Sigma})}$ are parametrized by the union of $\mathrm{Box}({c})$ over all $c \in {\bf \Sigma}$ \cite[Theorem 1]{poddar2003orbifold}\cite[Proposition 5.2]{borisov2005orbifold}. For each $g=\sum a_\rho \rho \in \mathrm{Box}^\circ({c})$, the corresponding twisted sector is the closed substack $T({\bf{\Sigma}})_c$, whose coarse moduli space is the closed torus orbit of $T(\Sigma)$ associated with the cone $c$. This twisted sector also has an age $\iota(g) = \sum a_i$. Since the usual cohomology of a toric Deligne--Mumford stack is isomorphic to that of its coarse moduli space, the orbifold cohomology of $T(\mathbf{\Sigma})$ is given by 
\[
H_\mathrm{orb}^*(T(\mathbf{\Sigma}), \mathbb{Q})=\bigoplus_{c \in {\bf\Sigma}} \bigoplus_{g \in \mathrm{Box}^\circ({c})} H^{*-2\iota(g)}(T({\bf\Sigma})_c, \mathbb{Q}). 
\]

\begin{remark}\label{r:orbgr}
    There is a clear distinction between the orbifold cohomology of a toric Deligne–Mumford stack $T(\mathbf{\Sigma})$ and that of its coarse moduli space $T(\Sigma)$. Since the box description \eqref{e:box} depends on the choice of generators for each cone, the age grading differs between the two cases.
\end{remark}

\begin{example}\label{ex:orbsing}
    Let us look at the case of an affine toric orbifold. Let $c$ be a maximal simplicial cone with primitive generators $\rho_1,\dots, \rho_d$. To each cone, we attach positive integers $\beta = (\beta_1,\dots, \beta_d)$, then there is a stacky cone ${\bf c} = (c,\beta)$ constructed from this data. Each toric stratum $T(c')$ of $T(c)$ is a finite quotient of $\mathbb{C}^d$ so $H^i(T(c'),\mathbb{Q})$ is isomorphic to $\mathbb{Q}$ if $i=0$ and 0 otherwise. Therefore, for all $n$,
    \[
    H_\mathrm{orb}^{n}(T({\bf c}), \mathbb{Q})=\bigoplus_{{\bf c}' \subseteq {\bf c}} \bigoplus_{\substack{ g \in \mathrm{Box}^\circ({\bf c}')\\ \iota(g) = n/2}} \mathbb{Q} = \bigoplus_{\substack{ g \in \mathrm{Box}({\bf c})\\ \iota(g)=n/2}} \mathbb{Q}.
    \]

\end{example}

\subsection{A comment on linear algebra}\label{s:linalg}

For later application, we describe a particular piece of linear algebra that we were unable to find in the literature, and plays a large role in the constructions of this paper. Suppose $V$ is a $\mathbb{K}$-vector space, $W\subseteq V$ is a subspace, and $U \subseteq V^*$ is a subspace so that $u(w) = 0$ for $u\in U$ and $w\in W$. We call $(W,U)$ an orthogonal pair. Fix volume forms $\mathrm{Vol}(W) \in \wedge^\bullet V$ and $\mathrm{Vol}(U) \in \wedge^\bullet V^*$ Attached to $(W,U)$ there are two graded vector spaces, 
\[
(\wedge^\bullet W^\perp) \wedge \mathrm{Vol}(U) \cong \wedge^{\bullet}(W^\perp/U),\qquad (\wedge^\bullet U^\perp) \wedge \mathrm{Vol}(W) \cong \wedge^{\bullet} (U^\perp/W).
\]
As subspaces of $\wedge^\bullet V$ and $\wedge^\bullet V^*$, these two vector spaces are independent of choice of volume forms. Suppose $W_1 \subseteq W_2$ and that $U_1\subseteq U_2$ so that $(W_1,U_1)$ and $(W_2,U_2)$ both form orthogonal pairs. There are two induced relations:
\[
\wedge^\bullet W^\perp_2 \subseteq \wedge^\bullet W^\perp_1,\qquad \mathrm{Vol}(U_2) = v_{1,2} \wedge \mathrm{Vol}(U_1)
\]
for some choice $v_{1,2} \in \wedge^{\dim U_2 - \dim U_1} U_2$ which is uniquely defined modulo $\wedge^{\dim U_2 - \dim U_1} U_1$. Therefore, we obtain inclusions
\begin{align*}
\varpi_{(W_2,U_2)}^{(W_1,U_1)} : \wedge^\bullet W_2^\perp \wedge \mathrm{Vol}(U_2) &\longrightarrow \wedge^\bullet W_1^\perp \wedge \mathrm{Vol}(U_1) \\ 
(\alpha \wedge \mathrm{Vol}(U_2)) & \longmapsto (\alpha \wedge v_{1,2}) \wedge \mathrm{Vol}(U_1)
\end{align*}
when both are considered as subspaces of $\wedge^\bullet V^*$. When $(W_1,U_1)$ and $(W_2,U_2)$ are clear from context we simply write $\varpi$. Of course the roles of $U$ and $W$ can easily be reversed, and we get similar inclusions $\bigwedge^\bullet U_2^\perp \wedge \mathrm{Vol}(W_2) \hookrightarrow \wedge^\bullet U_1^\perp \wedge \mathrm{Vol}(W_1)$. Choosing a volume form $\mathrm{Vol}(V)$ in $\wedge^{\dim V}V$, we obtain an isomorphism $\wedge^i V^* \xrightarrow{\sim} \wedge^{\dim V -i} V$ obtained by contraction with $\mathrm{Vol}(V)$. The following statement is an exercise in linear algebra.
\begin{lemma}\label{l:linalg}
    Let $U_1\subseteq U_2$ and $W_1\subseteq W_2$ be inclusions of vector spaces so that both $(W_1,U_1)$ and $(W_2,U_2)$ form orthogonal pairs. There is a commutative diagram of vector spaces
    \[
    \begin{tikzcd}
        \wedge^\bullet W^\perp_2 \wedge \mathrm{Vol}(U_2) \ar[r,"\mathrm{Vol}(V)"] \ar[d,"\varpi"] & \wedge^\bullet U^\perp_2 \wedge \mathrm{Vol}(W_2)\ar[d,"\varpi "] \\ 
        \wedge^\bullet W^\perp_1 \wedge \mathrm{Vol}(U_1) \ar[r,"\mathrm{Vol}(V)"] & \wedge^\bullet U^\perp_1 \wedge \mathrm{Vol}(W_1)
    \end{tikzcd}
    \]
    where $\mathrm{Vol}(V)$ denotes contraction with $\mathrm{Vol}(V)$. Here we have omitted the inversion on grading from our notation.
\end{lemma}

\section{Twisted cohomology of Landau--Ginzburg models}\label{s:tw}
We introduce main characters of this paper: Landau--Ginzburg models and their cohomology. The cohomology of a Landau--Ginzburg model admits so-called \emph{irregular Hodge filtration} that we want to compute for toric cases. As these materials are not extremely well-known, we provide a brief review of relevant notions in this section. 

\subsection{Landau--Ginzburg models and their cohomology}\label{s:lgcoh} 

We introduce the notion of a Landau--Ginzburg model discussed in this paper. We warn the reader that this is not necessary the most general definition. 
\begin{defn}
    Let $X$ be a $d$-dimensional projective orbifold and $D$ be a divisor. A pair $(X,D)$ has orbifold normal crossings if, locally, $(X,D) = (\mathbb{D}^d,E)/G$ where $(\mathbb{D}^d,E)$ is a normal crossings pair, $E$ is a $G$-invariant normal crossing divisor, and $\mathbb{D}^d/G$ is an orbifold. 
\end{defn} 

\begin{defn}\label{d:lgmodel}
    Let $(X,D)$ be an orbifold normal crossings pair and $w$ be a rational function on $X$ whose pole divisor $P$ is contained in $D$. Let $Z$ be its zero divisor. We say $(X,D,w)$ is a {\em nondegenerate Landau--Ginzburg model} if $Z$ is irreducible and $Z \cup D$ is orbifold normal crossings. We say that $(X,D,w)$ is {\em tame} if the divisor $P$ is equal to $P_\mathrm{red}$ or equivalently, if $f$ has multiplicity $>-2$ on each component of $P$\footnote{There are a number of places in the literature where the word ``tame'' is used to describe Landau--Ginzburg models. Our definition here slightly generalizes that of Katzarkov, Kontsevich, and Pantev \cite[Definition 2.4]{katzarkov2017bogomolov}. Note that the word tame has a completely different meaning in, e.g., \cite[pp. 199,  Examples]{sabbah2006hypergeometric}.}.
\end{defn}

More generally, for an orbifold $U$, we say $(U,w:U \to \mathbb{C})$ is a nondegenerate Landau--Ginzburg model if $w$ is a regular function on $U$ and there exists a nondegenerate Landau--Ginzburg model $(X,D,w)$ in the sense of Definition \ref{d:lgmodel} such that $(U,w)=(X\setminus D, w|_{X\setminus D})$. 

\begin{example}\label{ex:nondeg}
Suppose $w$ is a Laurent polynomial which is nondegenerate. Let $\Delta$ be the Newton polytope of $w$, and assume that $0_M \in \Delta$ and $\Delta$ is full-dimensional. The function $w$ induces a rational map on $T_\Delta$ and, since $0_M$ is contained in $\Delta$, the zero divisor of $w$ does not contain any toric boundary strata in $T_\Delta$. Therefore, if $\widetilde{T}$ is any toric resolution of $T(\Sigma_w)$, the triple $(\widetilde{T},D,w)$, where $D$ is any union of toric boundary divisors of $\widetilde{T}$ for which $\mathrm{mult}_w(D) \leq 0$, is a nondegenerate Landau--Ginzburg model. Here we abuse notation and use the letter $w$ to denote the rational function on $\widetilde{T}$ induced by $T$.
\end{example}

Next, we provide the de Rham-theoretic description of the \emph{twisted cohomology} of a nondegenerate Landau--Ginzburg model $(X,D,w)$ for a toroidal orbifold $X$, following \cite{yu2014irregular}. Note that in {\em op. cit.}, the case where $X$ is smooth and $D$ is a normal crossing divisor was studied in detail. However, all the results reviewed in this section extend to the case of toroidal orbifolds without any difficulty, as noted in Remark \ref{r:derhamorb}, so we omit the proofs.

Define the twisted cohomology of $(X,D,w)$ by letting $\nabla = d + dw\wedge (-)$, from which we obtain 
\[
\cOmega_{X}^0(\log D)(*P) \xrightarrow{\nabla}  \cOmega_X^1(\log D)(*P) \xrightarrow{\nabla} \dots    
\]
which is called the {\em twisted de Rham} complex of $w$. The hypercohomology of this complex is called the {\em twisted cohomology} of the Landau--Ginzburg model $(X,D,w)$. We use the notation $H^p(X\setminus D,w)$ to denote the $p$th twisted cohomology group. Yu defines the irregular Hodge filtration on the twisted de Rham complex by $F_{\mathrm{irr}}^\lambda:=F_w(\lambda)^{\geq \left\lceil \lambda 
\right\rceil }$ where $F_w(\lambda)$ is the complex 
\[
F_w(\lambda) = \left[  \cOmega_{X}^0(\log D)(\lfloor -\lambda P \rfloor) \xrightarrow{\nabla}  \cOmega_X^1(\log D)(\lfloor (1-\lambda) P\rfloor) \xrightarrow{\nabla} \dots \xrightarrow{\nabla} \cOmega_X^p(\log D)(\lfloor (p-\lambda)P \rfloor)\xrightarrow{\nabla}\dots  \right]    
\]
for $\lambda \in \mathbb{Q}_{\geq 0}$. It is clear that $F^\lambda_{\mathrm{irr}}$ is trivial if $\lambda \geq d$. When $w$ is constant, one sees that $F^\lambda_\mathrm{irr} = F^\lambda$, and the irregular Hodge filtration recovers the usual Hodge filtration on the mixed Hodge structure $H^*(X\setminus D)$.

Since the twisted de Rham complex is not a complex of coherent sheaves on $X$ it is convenient to show that there is a quasi-isomorphic coherent complex. We use the specialized notation $\cOmega^\bullet_X(\log D,\nabla) := F_w(0)$, equipped with its induced irregular Hodge filtration. Note that $\cOmega_X^\bullet(\log D,\nabla)$ is in fact a coherent complex.

\begin{proposition}\label{p:Yusub}{\cite[Corollary 1.3]{yu2014irregular}}
    The following inclusion is a quasi-isomorphism of filtered complexes
    \[ 
    (\cOmega_X^\bullet(\log D, \nabla), F_\mathrm{irr}^\bullet) \hookrightarrow (\cOmega_X^\bullet (\log D)(*P), F_\mathrm{irr}^\bullet).
    \]
\end{proposition}

There is clearly an induced filtration on $H^n(X\setminus D,w):=\mathbb{H}^n(X,(\cOmega_X^\bullet(\log D)(*P),\nabla))$ which we also denote by $F^\lambda_\mathrm{irr}H^n(X\setminus D,w)$. Let 
\[
F^{>\lambda}_\mathrm{irr}H^n(X\setminus D,w) = \bigcup_{\lambda'> \lambda} F^{\lambda'}_\mathrm{irr}H^n(X\setminus D,w),\qquad \gr_{F_\mathrm{irr}}^\lambda H^n(X\setminus D,w) = F^\lambda_\mathrm{irr}/F^{>\lambda}_\mathrm{irr}.    
\]
Finally, denote $\dim \gr_{F_\mathrm{irr}}^\lambda H^n(X\setminus D,w) = f^{\lambda,n-\lambda}(X\setminus D,w)$. These are called the {\em irregular Hodge numbers} of $(X,D,w)$. 

\begin{remark}
    In \cite{yu2014irregular}, it is shown that the graded filtered vector space $H^*(X \setminus D, w), F^\bullet_\mathrm{irr})$ does not depend on the choice of compactification $(X,D,w)$. However, since this choice is essential for computations, we continue to fix a specific compactification.  
\end{remark}

\begin{remark}[Irregular Hodge theory]
While we focus only on the irregular Hodge filtration and its associated graded pieces, there is an irregular mixed Hodge theory on $H^*(X \setminus D, w)$, developed by Sabbah and Yu \cite{sabbah2018irregular}, where the irregular Hodge filtration forms a part of an irregular mixed Hodge structure. An interesting question is how to extend the discussion in this article to the level of irregular mixed Hodge structures. In our companion work \cite{descent1}, we address this question for the tame case.
\end{remark}

There is, furthermore, an isomorphism between the twisted de Rham complex $\Omega^\bullet_X(\log D, \nabla)$ and the Kontsevich complex, whose degree $p$-term is defined to be 
\[
    \cOmega_X^p(\log D,w) := \ker\left((-)\wedge dw : \cOmega_X^p(\log D) \rightarrow \cOmega_X^{p+1}(2D)/\cOmega^{p+1}_X(\log D)\right)
    \]
and the differentials are the usual twisted differential \cite{esnault20171}. This isomorphism is induced by inclusion. If $w$ is tame, then equipping  the complex $(\cOmega_X^\bullet(\log D,w),d + dw)$ with the stupid filtration makes this a filtered isomorphism. 

According to \cite{katzarkov2017bogomolov}, the spectral sequence on the Kontsevich--de Rham complex $(\Omega_X^\bullet(\log D, w),d)$ coming from the stupid filtration also degenerates at the $E_1$ term. Following \cite{katzarkov2017bogomolov, shamoto2018hodge}, there is a limit mixed Hodge structure whose complex part is isomorphic to the hypercohomology of $(\cOmega_X^\bullet(\log D,w),d)$ and the Hodge filtration is induced by the stupid filtration. Thus, the irregular Hodge numbers 
\[
f^{p,q}(U,w) = \dim H^q(X,\cOmega_X^p(\log D,w)) 
\]
are integrally graded ($p,q \in \mathbb{Z}$). Furthermore, the first named author shows that these Hodge numbers are the same as the dimension of the relative Hodge numbers of a pair $(X, w^{-1}(t))$ for a generic quasi-smooth fibre $w^{-1}(t)$ \cite{harder}. Thus we have the following result.
\begin{proposition}\label{p:har}
    If $w$ is tame, $f^{p,q}(U,w) = \dim \gr_F^pH^{p+q}(U,w^{-1}(t))$.
\end{proposition}

Beyond the tame case, the monodromy on $H^n(U,w^{-1}(t))$ (around $t=\infty$) is quasi-unipotent so that the generalized eigenvalues of the induced operators are rational in general.  By \cite[Theorem E.1]{esnault20171}, it is shown that the irregular Hodge numbers of a pair $(U,w)$ can be identified with 
\begin{equation}\label{e:irrtop}
    \dim \gr_{F_\mathrm{irr}}^{\alpha + p} H^n(U,w) = \dim \gr_{F_{\lim}}^pH^n(U,w^{-1}(t))_\alpha   
\end{equation}
where $H^n(U,f^{-1}(t))_\alpha$ denotes the generalized eigenspace of eigenvalue $\exp(2\pi i \alpha)$ with respect to the monodromy action and $F^\bullet_{\lim}$ is the induced limiting Hodge filtration. One may view the equation \eqref{e:irrtop} as a topological realization of the irregular Hodge numbers.

% the there is a similar topological description beyond the tame case. Recall that a limit of mixed Hodge structures decomposes over $\overline{\mathbb{Q}}$ into a direct sum of mixed Hodge structures,
% \[
% \mathcal{H}_\mathrm{lim} = \bigoplus_{\lambda \in [0,1)\cap \mathbb{Q}} \mathcal{H}_{\exp(2\pi {\tt i}\lambda)}     
% \]
% based on the generalized eigenspace decomposition of the limit mixed Hodge structures. By \cite{esnault20171}[Theorem E.1], the irregular Hodge numbers of a pair $(U,f)$ can be identified with 
% \[
% \dim \gr_{F_\mathrm{irr}}^{\alpha + p} H^n(U,f) = \dim \gr_F^pH^n(U,f^{-1}(t))_\alpha    
% \]
% where $H^n(U,f^{-1}(t))$ denotes the limit mixed Hodge structure of the pair as $t$ approaches $\infty$. 

We note that the irregular Hodge filtration is compatible with the residue map in the following sense. Suppose we choose an irreducible component $D_1$ of $D$ which is not contained in $P$. Let $P_1 = D_1\cap P$, $w_1 = w|_{D_1}, \nabla_1 = \nabla|_{D_1}$, and let $B_1 = D_1 \cap D^1$ where $D^1$ is the union of all components of $D$ except $D_1$. The usual residue morphism induces a morphism of complexes
\[
\mathrm{Res}_{D_1}^X: (\cOmega_X^\bullet(\log D) (*P),\nabla,F_\mathrm{irr}^\bullet)\longrightarrow    a_{1*}(\cOmega_{D_1}^{\bullet-1}(\log B_1)(*P_1),\nabla_1,F_{\mathrm{irr}}^{\bullet-1})
\]
where $a_1 : D_1\hookrightarrow X$ denotes the inclusion map. Let us explain this point briefly. Choosing local orbifold coordinates $(\mathbb{D}^d,G_a)$ on $X$ at a point $a$ near $D_1$ with variables $(x_1,\dots,x_d)$, with $D_1 = V(x_1)$, we represent an element in the stalk $\cOmega_X^p(\log D) (*P)_a$ as a $G_a$-invariant section and define
\[
\alpha = d\log x_1 \wedge \omega + \beta,\qquad \mathrm{Res}_{D_1}(\alpha) = \omega|_{D_1}.
\]
We note that the isotopy group of $a$ in $D_1$, denoted $G_{a}^1$, is a quotient of $G_a$ and that $\omega|_{D_1}$ is clearly $G_a^1$-invariant. Observe that $\mathrm{Res}_D$ satisfies $\nabla_1\cdot \mathrm{Res}_{D_1}^X = \mathrm{Res}_{D_1}^X \cdot \nabla$. To see this, let $w' : \mathbb{D}^d\rightarrow \mathbb{C}$ be the local representation of $w$ in the orbifold coordinates above. Because $dw'$ does not have $D_1$ as a pole,
\[
\mathrm{Res}_{D_1}(\nabla \alpha) = \mathrm{Res}_{D_1}(d\alpha + dw' \wedge \alpha) = d \omega|_{D_1} + dw|_{D_1}\wedge \omega|_{D_1} = \nabla_1 \mathrm{Res}_{D_1}(\alpha).
\]
Therefore the residue map is a morphism of complexes. Furthermore, $\mathrm{Res}_D$ maps $\Omega_X^p(\log D)(\lfloor(p-\lambda )\rfloor P)$ to $\Omega_{D_1}^{p-1}(\log B_1)(\lfloor(p-\lambda) P_1\rfloor) =\Omega_{D_1}^{p-1}(\log B_1)(\lfloor((p-1)-(\lambda-1))P_1 \rfloor)$, so the residue morphism takes $F^\lambda_\mathrm{irr}$ to $F^{\lambda -1}_\mathrm{irr}$. 

Given a filtered vector space $(V,F^\bullet)$, we use the notation $V(-1)$  to denote the filtered  vector space for which $F^\lambda V(-1) = F^{\lambda-1}  V$. Without much trouble, one proves the following result.
\begin{proposition}\label{p:resles}
    There is a long exact sequence of $F_\mathrm{irr}$-filtered vector spaces
    \[
    \dots \longrightarrow H^p(X\setminus D^1, w) \longrightarrow H^p(X\setminus D,w) \longrightarrow H^{p-1}(D_1\setminus B_1,w|_{D_1})(-1) \longrightarrow \cdots.
    \]
\end{proposition}

Next, we discuss an orbifold version of the twisted cohomology and the irregular Hodge filtrations which may take the stacky structure into account. Suppose, as above, that $(X,D,w)$ form a nondegenerate Landau--Ginzburg model. For each twisted sector $X_{(g)}$, let $D_{(g)}= X_{(g)} \cap D$. Since $D_{(g)}$ has orbifold normal crossings, the same is true for the pair $(X_{(g)},D_{(g)})$. Since $X_{(g)}$ is an orbifold stratum of $X$, we get a similar description of $w$ restricted to $X_{(g)}$ and consequently $(X_{(g)},D_{(g)},w_{(g)}:=w|_{X_{(g)}})$ is also a nondegenerate Landau--Ginzburg model.

\begin{defn}
The {\em orbifold twisted cohomology} of a nondegenerate Landau--Ginzburg model is 
\[
H^n_{\mathrm{orb}}(X\setminus D,w) := \bigoplus_{g\in I_X}    H^{n-2\iota(g)}(X_{(g)}\setminus D_{(g)},w_{(g)}).
\]
There is an induced irregular Hodge filtration on $H^*_\mathrm{orb}(X \setminus D, w)$ 
\[
F^\lambda_\mathrm{irr} H^n_{\mathrm{orb}}(X\setminus D,w) := \bigoplus_{g\in I_X}    F^{\lambda-\iota(g)}_\mathrm{irr}H^{n-2\iota(g)}(X_{(g)}\setminus D_{(g)},w_{(g)}).
\]
Let $f_{\mathrm{orb}}^{\lambda, n-\lambda}(X \setminus D, w):=\dim \gr^\lambda_{F_\mathrm{irr}}H_{\mathrm{orb}}^n(X \setminus D, w)$. These are called the {\em irregular orbifold Hodge numbers of $(X, D, w)$}. In other words, we have
\[
f_\mathrm{orb}^{\lambda,\mu}(X \setminus D,w) = \sum_{g \in I_X} f^{\lambda-\iota(g), \mu-\iota(g)}(X_{(g)} \setminus D_{(g)}, w|_{X_{(g)}}).
\]
\end{defn}

\begin{remark}
    We emphasize that the irregular Hodge numbers of a stacky Landau--Ginzburg model are rational and that both  $\lambda$ and $\iota(g)$  can be distinct rational numbers. In general, $\lambda,\mu,$ and $\lambda + \mu$ need not be integers.
\end{remark}

\subsection{Twisted cohomology and irregular Hodge filtrations for simplicial Laurent polynomials}

Both as an example of the theory developed above and for later use, let us look at the situation where $X\setminus D \cong (\mathbb{C}^{*})^d$ and $w$ is a nondegenerate Laurent polynomial with support $A \subset M$. Assume that $0 \in A$. Let $\Delta(w)$ denote the Newton polytope of $w$. In this case, the twisted cohomology and its irregular Hodge filtration are well known. Let us fix an orbifold compactification $T=T_{\Delta(w)}$ of $(\mathbb{C}^*)^d$ as in Example \ref{ex:nondeg} and let $B$ be the toric boundary divisor. Then we have the following statements.
\begin{theorem}[Adolphson--Sperber \cite{as}, Yu \cite{yu2014irregular}]\label{t:as}
    Let notation be as above. 
    \begin{enumerate}
        \item There is an isomorphism between $\mathbb{H}^*(T,(\cOmega_T^\bullet(\log B)(*P),\nabla))$ and $H^*(\Gamma(\cOmega^\bullet(\log B)(*P)),\nabla)$.
        \item The cohomology $\mathbb{H}^l(T,\cOmega_T^\bullet(\log B)(*P),\nabla)$ has dimension
        \[
        { d-\dim L(\Delta(w)) \choose l-\dim L(\Delta(w))} (\dim L(\Delta(w)))! \mathrm{Volume}(\Delta(w))
        \]
        where $\mathrm{Volume}(\Delta(w))$ is the volume of $\Delta(w)$.
    \end{enumerate}
\end{theorem}

In the particular case where $A$ is the set of vertices of a simplex, we can give an explicit representation of the irregular Hodge filtration on the twisted cohomology of $w$. Before doing this, we introduce some notation.  For a simplex $\tau$ whose vertex set contains 0 there is an associated stacky cone $c_\tau$. Consider the set as in \eqref{e:box},
\[
\mathrm{Box}(c_\tau) = \left\{ \left. \,\sum_{\rho \in c_\tau[1]} a_\rho \rho \,\, \right| \,\, a_\rho \in [0,1)\right\} \cap M.
\]
For $g \in \mathrm{Box}(c_\tau)$ define its degree $\iota(g) = \sum_{\rho\in c_\tau[1]} a_\rho$. Let ${\bf B}^i_\tau$ be the $\mathbb{C}$-vector space with graded basis given by elements of $\mathrm{Box}(\tau)$ of degree $i$ and let ${\bf B}_\tau= \oplus_i {\bf B}_\tau^i$. Define a decreasing filtration  
\[
\mathsf{F}^p {\bf B}_\tau = \bigoplus_{i \leq \dim\tau -p } {\bf B}^i_\tau.
\]
Note that if $\tau_2$ is a face of $\tau_1$ there is a canonical projection map $p_{(\tau_1,\tau_2)} : {\bf B}_{\tau_1} \rightarrow {\bf B}_{\tau_2}$ which preserves the grading on ${\bf B}$, hence it also preserves $\mathsf{F}^\bullet$. Let $
T_M:=\mathrm{Spec}\,\mathbb{C}[M]$.

\begin{proposition}\label{p:adolphson-sperber}
    Suppose $w$ is a nondegenerate Laurent polynomial with support given by the vertices of $\tau$. Choose a volume form $\mathrm{Vol}(L(\tau))$. Then there is a filtered isomorphism 
    \begin{equation}\label{e:adolphson-sperber}
         \xi_{\tau} : ({\bf B}_\tau \otimes \wedge^\bullet M \wedge \mathrm{Vol}(L(\tau)), \mathsf{F}^\bullet) \longrightarrow (H^*(T_M,w),F^\bullet_{\mathrm{irr}}).
    \end{equation}
\end{proposition}
\begin{proof}

   We first identify $\Omega_T^\bullet(\log B)(*P)$ with $\mathcal{O}_T(*P)\otimes \wedge^\bullet M$. One may check that  $\Gamma(\mathcal{O}_T(*P)) = \mathbb{C}[c_\tau]$. Therefore, there is a natural morphism $\xi'_\tau: \mathbb{C}[c_\tau] \otimes \wedge^\bullet M \rightarrow \Gamma(T,\cOmega^\bullet_{T}(\log B)(*P))$ sending 
    \begin{equation*}
    \left(\sum_{m\in \mathrm{Box}(c_\tau)} s_m [m] \right) \otimes \alpha \longmapsto    \left(\sum_{m\in \mathrm{Box}(c_\tau)} s_m \underline{x}^m \right) \alpha.
    \end{equation*}
    Here $[m]$ represents a basis in $\mathbf{B}_\tau$ corresponding to $m \in \mathrm{Box}(c_\tau)$ and $\underline{x}^m$ represents the  corresponding monomial. If we fix a decomposition $M \cong L(\tau) \oplus M/L(\tau)$ we see that, for any element $\alpha' \in M/L(\tau)$ we get a subcomplex of $\Gamma(T,\cOmega_T^\bullet(\log B)(*P))$ isomorphic to $(\mathbb{C}[\mathrm{cone}(\tau)]\otimes \wedge^\bullet L(\tau) \wedge \alpha',\nabla)$. Therefore the problem reduces to the case where $L(\tau) = M$.

    Viewing ${\bf B}_\tau$ as a subset of $\mathbb{C}[\mathrm{cone}(\tau)]\otimes \wedge \mathrm{Vol}(M)$ we obtain the desired map, which we denote $\xi_\tau$ by restriction of $\xi'_\tau$. It is enough to show that this map is surjective because the image consists of $\nabla$-closed forms and because 
    \[
    |\mathrm{Box}(c_\tau)| = |M/\mathrm{span}_\mathbb{Z}(\rho \in c_\tau[1])| = \dim(L(\tau))!\mathrm{Volume}(\Delta(c_\tau)).    
    \] 
    Fix some basis $x_1,\dots, x_d$ corresponding to a basis of $M$ and let $\alpha_i = \wedge_{j\neq i}d\log x_j$, and note that 
    \begin{equation}\label{e:polord}
    \nabla (f \alpha_i) = (-1)^i \left(x_i\dfrac{\partial f}{\partial x_i} + x_i f \dfrac{\partial w}{\partial x_i}\right) (d\log x_1 \wedge \dots \wedge d\log x_d). 
    \end{equation}
    Suppose $g = x_1^{a_1}\dots x_d^{a_d}$ is a monomial which is divisible by $x^\rho$ corresponding to a nonzero vertex of $\tau$. Because we have chosen the coefficients of $w$ generically, we may find values $t_1,\dots, t_d$ so that 
    \[
        g=x_1^{a_1}\dots x_d^{a_d}= \sum_{i=1}^n   (-1)^i t_i x_i \dfrac{\partial w}{\partial x_i}.
    \]
    In other words, after applying \eqref{e:polord}, 
    \[
    \nabla\left(\sum_{i=1}^d (-1)^it_i\left(\dfrac{g}{\underline{x}^\rho}\right) \alpha_i   \right) =\left(\sum_{i=1}^d a_it^i \left(\dfrac{g}{\underline{x}^\rho}\right) + g \right) (d\log x_1\wedge \dots \wedge d\log x_d).
    \]
    More precisely, the class of $g (d\log x_1\wedge \dots \wedge d\log x_d)$ in $H^n(\Gamma(\cOmega^\bullet_T(\log B)(*P)))$ is equivalent to the class of $\sum_{i=1}^d a_i t_i(g/x^\rho)$. Applying this repeatedly, we may reduce any cohomology class to a class corresponding to something in the image of $\xi_\tau$. Observe that the volume of $\tau$ times $\dim L(\tau)!$ is precisely the number of integral points in $\mathrm{Box}(c_\tau)$, therefore $\xi_\tau$ must be an isomorphism.
    
    Finally, we show that this is a filtered isomorphism. According to Yu, the irregular Hodge filtration on $H^*(T_M,w)$ may be identified with the Newton polytope filtration for which 
    \[
    \underline{x}^m(d\log x_1 \wedge \dots \wedge d\log x_d) \in F_\mathrm{NP}^\lambda \iff   m \in (d-\lambda)\tau.
    \]
    The polytope $(d-\lambda)\tau$ is precisely $m \in M_\mathbb{R}$ for which $m = \sum_{\rho \in \tau[1]} a_\rho \rho, \sum_{\rho \in \tau[1]} a_\rho \leq (d-\lambda)$ and $a_\rho \geq 0$. Thus for $m\in \mathrm{Box}(\tau)$,  $\underline{x}^m(d\log x_1\wedge \dots \wedge d\log x_d) \in F^\lambda_\mathrm{NP} \setminus F^{>\lambda}_\mathrm{NP}$ if and only if
    \[
        m = \sum_{\rho \in \tau[1]} a_\rho \rho, \qquad \sum_{\rho \in \tau[1]} a_\rho = (d-\lambda).
    \]
    This statement makes use of the  fact that $\tau[1]$ is a basis for $L(\tau)$. Therefore, $\mathrm{age}_\tau(m) = d-\lambda$.
\end{proof}
Let $\mathbf{c}$ be a stacky simplicial cone of dimension $d$ in a lattice $M$ of dimension $d$ and let $w(\mathbf{c}) = \sum_{\rho \in \mathbf{c}[1]}b_\rho \underline{x}^\rho$ where $b_\rho$ are generic in $\mathbb{C}^*$. To clarify notation here, $\mathbf{c}$ has underlying cone $c$ with primitive integral ray generators $\rho_1,\dots, \rho_d$ and weight vector $\beta = (\beta_1,\dots, \beta_d)$. Then ${\bf c}[1] =\{\beta_1\rho_1,\dots, \beta_d\rho_d\}$. We expect that the Landau--Ginzburg model $((\mathbb{C}^*)^d, w({\bf c}))$ is mirror to $T(\mathbf{c})$. We have the following statement which justifies this expectation and provides a prototype for the main result of this paper.
\begin{corollary}
     Let notation be as above. For all $\lambda, \mu \in \mathbb{Q}$,
     \[
     f^{\lambda,\mu}((\mathbb{C}^*)^d,w(\mathbf{c})) = f^{d-\lambda,\mu}_\mathrm{orb}(T(\mathbf{c}),0).
     \]
\end{corollary}

Next we discuss functoriality of the filtered isomorphism \eqref{e:adolphson-sperber} under the residue maps. Suppose $n \in N$ is primitive and choose the partial compactification $(\mathbb{C}^*)^{d-1}\times \mathbb{C}$ of $(\mathbb{C}^*)^{d}$ attached to the fan $\mathbb{R}_{\geq 0}n$. Let $D_n = (\mathbb{C}^*)^{d-1}\times 0$. Assume that $\min_{m\in \tau[1]}\{\langle n,m \rangle\} = 0$. Since $\tau$ is a simplex containing $0_M$, this is equivalent to saying that a Laurent polynomial with Newton polytope $\tau$ does not have a pole at the divisor $D_n$. Let $\tau_n = \{m \in \tau \mid \langle m,n \rangle =0\}$ and let $w_n$ be the Laurent polynomial with support $\tau_n$. We also have the wedge product map and its dual, 
\[
\wedge n : \wedge^\bullet (N/n) \longrightarrow \wedge^{\bullet +1} N ,\qquad (\wedge n)^* : \wedge^{\bullet +1}M \longrightarrow \wedge^\bullet n^\perp.  
\]
An explicit computation then allows us to represent the morphism $\mathrm{Res}_{D}$ combinatorially in terms of ${\bf B}_\tau$. 
\begin{corollary}
There is a commutative diagram of filtered vector spaces
\[
\begin{tikzcd}
    {\bf B}_\tau \otimes (\wedge^\bullet M \wedge \mathrm{Vol}(L(\tau))) \ar[r,"\xi_\tau"] \ar[d,"p_{(\tau,\tau_n)}\otimes (\wedge n)^*"] & H^*(T_M,w) \ar[d,"\mathrm{Res}_{D_n}"] \\
    {\bf B}_{\tau_n} \otimes (\wedge^\bullet n^\perp \wedge \mathrm{Vol}(L({\tau_n})))  \ar[r, "\xi_{\tau_n}"] & H^*(T_{n^\perp},w_n)(-1) 
\end{tikzcd}    
\]
\end{corollary}

We emphasize that, {\em a priori}, the irregular Hodge filtration does not have a preferred splitting, however, the isomorphism in Proposition \ref{p:adolphson-sperber} determines a choice of splitting which is identified with the splitting that induces $\mathsf{F}^\bullet$ and which commutes with the residue map.

\section{Toric and tropical Landau--Ginzburg models}\label{s:thyper}

In this section we discuss toric Landau--Ginzburg models over the field of Puissieux series $\mathbb{C}\{\!\{t\}\!\}$ and we assign to these Landau--Ginzburg models both the data of a tropicalization and, in certain cases, a filtered sheaf on the tropicalization.

\subsection{Tropical hypersurfaces in $\mathbb{R}^d$}\label{s:tropgeo}
We start reviewing basic constructions in tropical geometry. The standard references are \cite{mikhalkin2014tropical,MikhalkinRau, Payne2009trop}. 

Let $w \in \mathbb{C}\{\!\{t\}\!\}[M]$ be a Laurent polynomial with coefficients in the field $\mathbb{C}\{\!\{t\}\!\}$ of formal Puiseux series. By choosing an integral basis of $M$ and an isomorphism $M_\mathbb{R}\cong \mathbb{R}^d$, one can express $w$ as before: 
\[
w=\sum_{m \in A} u_m \underline{x}^m=\sum_{m \in A} u_m(t)x_1^{m_1}x_2^{m_2}\cdots x_d^{m_d},\qquad u_m(t) \in \mathbb{C}\{\!\{t\}\!\}^\times
\]
where $A$ is a finite collection of integral points in $M$, and $m=(m_1, \dots, m_d)$. The \emph{tropicalization of $w$} is a piecewise affine linear function on $N_\mathbb{R} \cong \mathbb{R}^d$ given by 
\[
\mathrm{Trop}(w):=\min_{m \in A}\{\nu(u_m(t)) + m\cdot \check{x} \}
\]
where we simply write $\check{x}=(\check{x}_1, \dots, \check{x}_d)$ as the coordinates on the dual $N_\mathbb{R}$ and $\nu :\mathbb{C}\{\!\{t\}\!\} \rightarrow \mathbb{Z} \cup \{\infty\}$ is the usual valuation at 0. 

\begin{remark}
    Between the two opposite conventions, \emph{min} and \emph{max}, we adopt the \emph{min} convention, following \cite{Payne2009trop}. This choice simplifies the handling of signs in many of the proofs in this section.
\end{remark}

We define the tropical hypersurface of $w$, denoted by $\mathsf{X}_w$, to be the non-differentiable locus of $\mathrm{Trop}(w)$. Namely, 
\[\mathsf{X}_w=\left\{\left. \check{x}\in \mathbb{R}^d \,\, \right| \,\, \mathrm{Trop}(w)= \nu(u_m(t))+m\cdot \check{x}=\nu(u_{m'}(t))+m'\cdot \check{x} \text{ for }m \neq m' \in A\right\}\]
Then $\mathsf{X}_w$ admits a polyhedral complex structure. Recall that a \emph{polyhedral complex} in a real vector space $\mathbb{R}^d$ is a finite non-empty set $\mathsf{P}$ of convex polyhedra, called \emph{faces}, such that for any pair of faces $\sigma, \sigma' \in \mathsf{P}$, any face of $\sigma$ belongs to $\mathsf{P}$ and the intersection $\sigma \cap \sigma'$ is either empty or is a common face of $\sigma$ and $\sigma'$. Moreover, $\mathbb{R}^d$ admits the induced polyhedral complex structure whose $(d-1)$-skeleton is $\mathsf{X}_w$. We denote it by $\mathrm{PC}(w)$. Cells in $\mathrm{PC}(w)$ are parametrized by $m\in A$. Let $\sigma_m$ be the locus where $\nu(u_m(t))+m\cdot \check{x}$ reaches the minimum. For each $\sigma \in \mathrm{PC}(w)$, we introduce some terminology:
\begin{itemize}
    \item $L(\sigma)$ is the linear span of $\sigma$ and $\dim\sigma:=\dim L(\sigma)$. In other words, this is the smallest sublattice of $N$ containing all differences of elements of $\sigma$. 
    \item $\mathrm{rec}(\sigma)$ is the recession cone of $\sigma$, that is $\{v \in \mathbb{R}^d \mid \exists \check{x} \in \sigma \text{ such that } \check{x}+sv \in \sigma,  \forall s>0\}$
\end{itemize}

\noindent On the other hand, there is a dual description of the polyhedral structure of $\mathrm{PC}(w)$. We set 
\[
\widetilde{A}:=\{(m, \nu(u_m(t))) \in \mathbb{Z}^d\times \mathbb{R}\mid m \in A, \nu(u_m(t)) \neq \infty \}
\]
and consider the convex hull $\mathrm{Conv}(\widetilde{A})$. The projection of the upper faces of $\mathrm{Conv}(\widetilde{A})$ provides a subdivision $\mathrm{SD}(w)$ of the Newton polytope $\Delta(w)$. For each cell $\sigma \in \mathrm{PC}(w)$, we define $A_\sigma=\{m \in A \mid \sigma \subseteq \sigma_m\}$, that is a set of monomials which are the minimum on $\sigma$. Then there is an inclusion--reversing bijection between $\mathrm{PC}(w)$ and $\mathrm{SD}(w)$. 

\begin{theorem}\cite[Theorem 2.3.7]{MikhalkinRau}\label{t:subdiv}
    There is an inclusion-reversing bijection between subdivisions $\mathrm{PC}(w)$ and $\mathrm{SD}(w)$ given by 
    \[
    \begin{aligned}
        \mathrm{PC}(w) & \longrightarrow \mathrm{SD}(w) \\
        \sigma & \longmapsto \tau_\sigma:=\mathrm{Conv}(A_\sigma)
    \end{aligned}
    \]
    such that $\dim \sigma + \dim \tau_\sigma =d$ and $L(\sigma)=L(\tau_\sigma)^\perp$. Furthermore, let $f_{\tau_\sigma}$ be the minimal face of $\Delta(w)$ containing $\tau_\sigma \in \mathrm{SD}(w)$. Then $\mathrm{rec}(\sigma) = \mathrm{nc}(f_{\tau_\sigma})$.
\end{theorem}

\begin{example}\label{eg:trop}
    Let $d=2$ and $A=\{(0,0), (0,1), (1,1), (1,2)\}$. Choose a Laurent polynomial 
    \[
    w=1+x_2+tx_1x_2+t^4x_1^2x_2
    \]
    The induced subdivision $\mathrm{SD}(w)$ of $\Delta(w)$ is described in Figure \ref{fig:ex}. 
    The tropicalization of $w$ is given by 
    \[
    \mathrm{Trop}(w)=\min\{0, \check{x}_2, \check{x}_1+\check{x}_2+1, 2\check{x}_1+\check{x}_2+4\}
    \]
    and the associated tropical variety $\mathsf{X}_w$ is described in Figure \ref{fig:ex}. For example, the $0$-cell $\sigma_1$ corresponds to $\tau_{\sigma_1}=\mathrm{Conv}(\{(0,0), (0,1), (1,1)\})$, and the $2$-cells $\sigma_0$ and $\sigma_2$ correspond to $\tau_{\sigma_0}=(0,0)$ and $\tau_{\sigma_2}=\{(1,1)\}$, respectively.
    
    \begin{figure}[h]
    \centering 
      \begin{tikzpicture}[scale=1]
\foreach \x in {0,1,2}
   \foreach \y in {0,1} 
      \draw[fill] (4/3*\x,4/3*\y) circle (1.5pt) coordinate (m-\x-\y);

\draw[thick] (m-0-0) -- (m-0-1);
\fill (0,-0.3) node {$(0,0)$};
\fill (0,1.6) node {$(0,1)$};
\fill (1.3,1.6) node {$(1,1)$};
\fill (2.6,1.6) node {$(2,1)$};
\draw[thick] (m-0-0) -- (m-1-1);
\draw[thick] (m-0-0) -- (m-2-1);
\draw[thick] (m-0-1) -- (m-1-1);
\draw[thick] (m-1-1) -- (m-2-1);
\end{tikzpicture} 
\qquad \qquad 
\begin{tikzpicture}[scale=0.5]

\draw[thick] (-1,0) -- (1,0);
\draw[thick] (-1,0) -- (-1,-2);
\draw[thick] (-1,0) -- (-3,2);
\draw[thick] (-3,2) -- (-4,4);
\draw[thick] (-3,2) -- (-3,-2);
\fill (-1,2) node {$\sigma_0$};
\fill (-0.7,0.3) node {$\sigma_1$};
\fill (-2, -1) node {$\sigma_2$};
\fill (-1,0) circle(.15);
\end{tikzpicture} 
\caption{\scshape $\mathrm{SD}(w)$ (Left) and $\mathrm{PC}(w)$ (Right) in Example \ref{eg:trop}. \label{fig:ex}}
\end{figure}
\end{example}

We extend the previous discussion to a more general tropical space, called a \emph{tropical toric variety}. Let $\Sigma \subset N$ be a $d$-dimensional fan. Following \cite{Payne2009trop}, we may abstractly define the tropical variety $\mathrm{Trop}(\Sigma)$ to be the union of all cells
\begin{equation}\label{e:trop}
\mathrm{Trop}(\Sigma) =  \coprod_{c \in \Sigma} \left(N_\mathbb{R}/\mathrm{span}(c)\right)
\end{equation}
equipped with the finest topology such that the inclusion maps of $N_\mathbb{R}/\mathrm{span}(c)$ are continuous and so that $\{(nv + r)\mid n \in \mathbb{N}\}$ converges to a point in $N_\mathbb{R}/\mathrm{span}(c)$ if and only if $v \in c$. Note that $\mathrm{Trop}(\Sigma)$ comes with a canonical stratification by the fan structure of $\Sigma$ and we write $\mathrm{Trop}(\Sigma)_c=N_\mathbb{R}/\mathrm{span}(c)$ for the stratum corresponding to a cone $c \in \Sigma$. The theory of tropical toric varieties is entirely analogous to the classical theory discussed above. We refer to \cite{Payne2009trop} for details.

Let $w$ be a nondegenerate Laurent polynomial and consider the tropicalization $\mathsf{X}_w$ in $N_\mathbb{R}$. By canonically embedding $\mathrm{PC}(w)$ into $\mathrm{Trop}(\Sigma)$ and taking the closure of the image, we have the induced polyhedral decomposition of $\mathrm{Trop}(\Sigma)$, which allows faces to be the closure of convex polyhedra in $N_\mathbb{R}$. We call this a \emph{tropical polyhedral complex} and often refer to it simply as a polyhedral complex if it is clear from the context. To obtain a well-behaved tropical polyhedral structure, we assume that $\Sigma$ is a subfan of a refinement of the normal fan $\mathrm{nf}(\Delta(w))$ of the Newton polytope $\Delta(w)$ and that $w$ is a regular function on $T(\Sigma)$. The induced polyhedral complex structure on $\mathrm{Trop}(\Sigma)$ is simply denoted by $\mathsf{T}(\Sigma, w)$ or $\mathsf{T}$ if the choice of $\Sigma$ and $w$ is clear from the context. Cells in $\mathsf{T}$ are parametrized by pairs $(c,\sigma) \in \Sigma \times \mathrm{PC}(w)$ with $c \subseteq \mathrm{rec}(\sigma)$. Equivalently, they are parametrized by pairs $(c,\tau) \in \Sigma \times \mathrm{SD}(w)$ with $c \subseteq \mathrm{nc}(f_{\tau})$, as stated in Theorem \ref{t:subdiv}. Therefore, we define a poset 
\begin{equation}\label{e:cellsinT}
\mathsf{S}_\mathsf{T}:=\left\{(c,\tau) \in \Sigma \times \mathrm{SD}(w) \mid c\subseteq \mathrm{nc}(f_\tau) \right\}
\end{equation}
where ${(c,\tau)} \preceq (c',\tau')$ if and only if 
$c'\subseteq c$ and $\tau' \subseteq \tau$.

\begin{example}\label{eg:trop2}
    Continuing from Example \ref{eg:trop}, let us choose $\Sigma$ as the cone generated by ${(1,0), (0,1)}$. This cone is a subfan of a refinement of $\mathrm{nf}(\Delta(w))$ (see Figure \ref{fig:ex2}). The resulting tropicalization $\mathsf{T}(\Sigma, w)$ is depicted in Figure \ref{fig:ex3}, where the bold faces should be considered as the infinity boundaries. For instance, the cell $C$ in Figure \ref{fig:ex3} corresponds to the pair $(\mathrm{cone}((0,1)), \mathrm{Conv}((0,0), (0,1)))$. 

        \begin{figure}[h]
    \centering 
\begin{tikzpicture}
\fill (-2,0) node {$\mathrm{nf}(\Delta(w))=$};
\draw[thick, ->] (0,0) -- (1,0);
\draw[thick, ->] (0,0) -- (0,-1);
\draw[thick, ->] (0,0) -- (-1,2);
\draw[dotted, ->] (0,0) -- (0,1);
\fill (3,0) node {$\Sigma=$};
\draw[thick, ->] (4,0) -- (5,0);
\draw[thick, ->] (4,0) -- (4,1);
\end{tikzpicture} 
\caption{\scshape $\mathrm{nf}(\Delta(w))$ (Left) and $\Sigma$ (Right) in Example \ref{eg:trop2}. \label{fig:ex2}}
\end{figure}

\begin{figure}
\begin{tikzpicture}[scale=0.6]
\draw[thick] (-1,0) -- (1,0);
\draw[thick] (1,0) -- (2,0);
\draw[very thick] (2,-2) -- (2, 5);
\draw[very thick] (2,5) -- (-2, 5);
\draw[dotted] (-2,5) -- (-3, 5);
\draw[thick] (-1,0) -- (-1,-2);
\draw[thick] (-1,0) -- (-3,2);
\draw[thick] (-3,2) -- (-4,4);
\draw[dotted] (-4,4) -- (-4.5,5);
\draw[thick] (-3,2) -- (-3,-2);
\fill (2,0) circle(.15);
\fill (2.5,0) node {$C$};
\end{tikzpicture} 
\caption{\scshape $\mathrm{Trop}(\Sigma, w)$ in Example \ref{eg:trop}. \label{fig:ex3}}
\end{figure}

\end{example}

%\begin{remark}\label{rem:noncpt trop}
 %   When $\Sigma$ is not a $d$-dimensional fan in $N$ (for example, $\Sigma$ contains a non-strictly convex cone), one should first introduce the refinement $\Sigma'$ of $\Sigma$, satisfying the assumption on the regularity of $f$, and apply the same construction as above. After that, we may remove cells corresponding to cones in $\Sigma'$ that are not contained in $\Sigma$. This is still a well-defined polyhedral complex whose cells are parametrized by the poset defined in the same as \eqref{e:cellsinT}.
%\end{remark}

\subsection{Tropical sheaves and their cohomology}\label{s:tropsheaves}

We do not cover the most general form of the theory of tropical sheaves and their cohomology, referring the reader to \cite{curry, itenberg2019tropical} for further details. Instead, we introduce only the minimal amount of theory necessary for our context. Consider the tropical polyhedral complex $\mathsf{X}$ introduced above.

\begin{defn}
    The {\em Alexandrov topology} of $\mathsf{X}$ is the topology whose open sets are sets of faces $\mathsf{U}$ so that if $\mathsf{f}_1 \in \mathsf{U}$ and $\mathsf{f}_1 \preceq \mathsf{f}_2$ then $\mathsf{f}_2 \in \mathsf{U}$. For any $\mathsf{f}\in \mathsf{X}$, the open star of $\mathsf{f}$ is 
        \[
        \mathrm{St}(\mathsf{f}) = \bigcup_{\mathsf{f}\preceq \mathsf{g}} \mathsf{g}.
        \]
    \end{defn}

    \begin{remark}\label{r:fanalexandrov}
        The Alexandrov topology exists for any poset, defined in precisely the same way as above. For instance, if $\Sigma$ is a fan, then $\Sigma$ has the structure of a poset where $\sigma \preceq \sigma'$ if $\sigma'$ is a face of $\sigma$. The resulting topological space is called the {\em fan space} of $\Sigma$, or sometimes a {\em fanfold}.
    \end{remark}
From now on, we consider $X$ equipped with the Alexandrov topology. Let $\mathcal{A}$ be a category; for instance, we primarily focus on the category of (complexes of) filtered vector spaces over $\mathbb{C}$.
    \begin{defn}\label{d:tropsh}
          A \emph{tropical $\mathcal{A}$-sheaf} on $\mathsf{X}$ is a contravariant functor $\mathbf{F}$ from $\mathsf{X}$ to a category $\mathcal{A}$. The cohomology of a tropical $\mathcal{A}$-sheaf is then computed as usual, with respect to the Alexandrov topology.
    \end{defn}

    \begin{remark}\label{r:alexandrov-sheaf}
        Properly speaking, the definition of a tropical sheaf in Definition \ref{d:tropsh} describes a presheaf. However, in the tropical setting, we canonically construct a sheaf on $\mathsf{X}$ using the given data as follows. For any closed cell $\mathsf{f}$ of $\mathsf{X}$ in $\mathsf{U}$ we must also have $\mathrm{St}(\mathsf{f}) \subseteq \mathsf{U}$, so $\mathsf{U} = \cup_{\mathsf{f}\in \mathsf{U}}\mathrm{St}(\mathsf{f})$. Define 
        \[
        \mathbf{F}(\mathsf{U}) := \lim_{\substack{\longleftarrow \\ \mathsf{f}\in \mathsf{U}}} \mathbf{F}(\mathrm{St}(\mathsf{f})) = \lim_{\substack{\longleftarrow \\ \mathsf{f}\in \mathsf{U}}} \mathbf{F}(\mathsf{f}).
        \]
        Curry \cite[Theorem 4.2.10]{curry} shows that this is in fact a sheaf on $\mathsf{X}$.
    \end{remark}

    We now wish to describe how one computes the cohomology of a tropical sheaf on $\mathsf{X}$. In many places (e.g., \cite{itenberg2019tropical}) the cohomology of tropical sheaves is described in terms of a cellular complex under certain conditions outlined below. 
    
    \begin{defn}\label{d:regcell}
        A {\em regular cell complex} is a topological space $\mathfrak{X}$ consisting of a locally finite decomposition $\cup_{i\in I} \sigma_i$ satisfying the following axioms. 
        \begin{enumerate}
        \item If $\sigma_i \cap \overline{\sigma}_j \neq \emptyset$ then $\sigma_i \subseteq \overline{\sigma}_j$,
        \item Each $\sigma_i$ is homeomorphic to $\mathbb{R}^k$ for some $k$,
        \item Each pair $(\sigma_i, \overline{\sigma}_i)$ is homeomorphic to $(\mathrm{int}(\mathbb{B}^k), \mathbb{B}^k)$ where $\mathbb{B}^k$ denotes the closed $k$-ball.
        \end{enumerate}
        If $\sigma$ satisfies (1), (2), and the one-point compactification ${\mathfrak{X}} \cup \{\infty\}$ of $\mathfrak{X}$ satisfies (3) with cell decomposition given by $\cup_{i\in I} (\sigma_i \cup \{\infty\})$ is a regular cell complex, we say  that $\mathfrak{X}$ is a {\em cell complex}.
    \end{defn}
   
    \begin{defn}
        Given a cell complex, a {\em signed incidence relation} is a choice of integer $[{\bm x}:{\bm y}]\in \{0,\pm 1\}$ for each pair of cells ${\bm x},{\bm y}$ of $\mathfrak{X}$  so that 
        \begin{enumerate}
        \item If ${\bm x}$ is not a facet of ${\bm y}$ then $[{\bm x}: {\bm y}] = 0$.
        \item For each pair of cells ${\bm x},{\bm y}$,  we have $\sum_{{\bm x} \preceq {\bm z} \preceq {\bm y}} [{\bm x}: {\bm z}][{\bm z}:{\bm y}] = 0 $.
        \end{enumerate}
        Signed incidence relations always exist for cell complexes and can be obtained by choosing orientations on each cell and letting $[\sigma:\gamma] = 1$ if orientations agree and $-1$ if they do not agree.
     \end{defn}
     Given any choice of signed incidence relation, the sheaf cohomology of a cellular sheaf ${\bf F}$ on $\mathfrak{X}$ is isomorphic to its cellular cohomology. Let $\mathfrak{X}^{[i]}$ denote the set of all $i$-dimensional cells of $\mathfrak{X}$. Then:
     \[
    C^i_\mathrm{cell}(\mathfrak{X},{\bf F}) =   \bigoplus_{\substack{{\bm x} \in \mathfrak{X}^{[i]} \\ {\bm x} \text{ compact}}} \mathbf{F}({\bm x}),\qquad \begin{array}{rl}
        d: C^i_\mathrm{cell}(\mathfrak{X},{\bf F}) & \longrightarrow C^{i+1}_\mathrm{cell}(\mathfrak{X},{\bf F}) \\
        (s \in {\bf F}({\bm x})) & \longmapsto \left(\sum_{{\bm y} \preceq {\bm x}}[{\bm x}:{\bm y}] {\bf F}({\bm x},{\bm y})(s)\in {\bf F}({\bm y})\right).
    \end{array}
     \]
    The cellular cohomology of $\mathfrak{X}$ is the cohomology of $(C_\mathrm{cell}^\bullet(\mathfrak{X},{\bf F}), d)$, which is independent of choice of signed incidence relation. 

    Note that {\em compact} tropical varieties are cell complexes in a strong sense, therefore, a result of Curry \cite{curry} says that we can compute their cohomology using the cellular complex. Noncompact tropical varieties frequently fail to be cell complexes. However, if they admit a well-behaved compactification, the same results can be applied. In this article, we focus on such cases, as in Proposition \ref{p:cellcomplex}, and refer to the cellular cohomology as the \emph{tropical cohomology}. The following result shows that tropical cohomology can be computed using \v{C}ech cohomology.
    
\begin{proposition}\label{p:Cech-stars}
    Suppose $\mathbf{F}$ is a tropical sheaf on a tropical polyhedral complex $\mathsf{X}$. For each face $\mathsf{x} \in \mathsf{X}$, $H^i(\mathrm{St}(\mathsf{x}),{\bf F}) = 0$  if $i \neq 0$ and $H^0(\mathrm{St}(\mathsf{x}),{\bf F}) \cong {\bf F}(\mathsf{x})$.  
\end{proposition}
\begin{proof}
    Suppose $\mathsf{x}$ is a point, then this is observed in \cite[Remark 2.15]{aksnes2023tropical}. We will explain briefly. If $\mathsf{x}$ is a point, $\mathsf{x}$ is the only compact cell in $\mathrm{St}(\mathsf{x})$ therefore $C^i_\mathrm{cell}(\mathrm{St}(\mathsf{x}),{\bf F}) = 0$ if $i\neq 0$ and $C^0_\mathrm{cell}(\mathrm{St}(\mathsf{x})) = {\bf F}(\mathsf{x})$. If $\mathsf{x}$ is not a point, then $\mathrm{St}(\mathsf{x})$ is not a cell complex in the sense of Definition \ref{d:regcell}, so we cannot apply \cite[Remark 2.15]{aksnes2023tropical} directly. Instead, we observe that if $\mathrm{St}(\mathsf{x})$ is a union of faces whose common intersection is $\mathsf{x}$. The poset structure on $\mathrm{St}(\mathsf{x})$ agrees with the the poset structure on the local fan $\Sigma_\mathsf{x}$ of $\mathsf{x}$. Let $\mathrm{lin}(\mathsf{x})=\{x \in \mathbb{R}^n \mid x + \Sigma_{\mathsf{x}} = \Sigma_\mathsf{x}\}$. The poset structure on $\Sigma_\mathsf{x}$ agrees with that of $\Sigma_\mathsf{x}/\mathrm{lin}(\mathsf{x})$ which is a pointed fan whose cones correspond to the cells in $\mathrm{St}(\mathsf{x})$. Then 
    \[
    H^i(\mathrm{St}(\mathsf{x}),\mathbf{F}) \cong H^i(\Sigma_{\mathsf{x}}/\mathrm{lin}(\mathsf{x}), \widetilde{\bf F})
    \]
    for a cellular sheaf $\widetilde{\bf F}$ so that $\widetilde{\bf F}(\mathsf{y}/\mathrm{lin}(\mathsf{x})) = {\bf F}(\mathsf{y})$. Applying the observation above, the claim follows. 
\end{proof}

\subsection{Tropical Jacobian sheaves for Laurent polynomials}\label{t:thomsheaf}

Following \cite{ruddat2014skeleta}, we will identify a particular class of triangulations.
\begin{defn}
    The collection of points $A$ is pointed if it contains the origin. We view $0_M$ as a distinguished point. A coherent triangulation $S$ of ${\Delta}$ is a star triangulation based at $0_M$ if each simplex in $S$ contains $0_M$ as a vertex, and every in  $A$ is the vertex of a simplex in $S$
\end{defn}
After translation, any set of points $A \subseteq M$ is pointed. Equivalently, a pointed set is a finite set of points $A\subseteq M$ in which one point is distinguished. Not every pointed set admits a pointed triangulation. However, for every polytope $\Delta$ there is a set of points $A$ which admits a coherent star triangulation based at $0_M$ and $\Delta = \mathrm{Conv}(A)$. Let us assume that $A$ is a pointed collection of points in $M$ and define $A'$ so that $A = A' \cup \{0_M\}$. As before, consider a Laurent polynomial 
\[
w = 1+\sum_{m\in {A'}}u_m(t) \underline{x}^m \in \mathbb{C}\{\!\{t\}\!\}[M]
\]
where $u_m(t)$ are chosen generically from the field of Puiseux series $\mathbb{C}\{\!\{t\}\!\}$. We let $\mathrm{SD}(w)$, $\Sigma$, and $\mathsf{T}(\Sigma,w)$ be as defined in the previous section. Note that $\mathrm{SD}(w)$ is a star triangulation based at $0_M$, and we define $\mathrm{SD}(w)_0$ to be those simplices which contain $0_M$. Let $\mathsf{T}(\Sigma,w)_0$ denote the subset consisting of cells indexed by $(c,\tau) \in \Sigma\times \mathrm{SD}(w)_0$ with $c\subseteq \mathrm{nc}(f_\tau)$. For instance, in Examples \ref{eg:trop} and \ref{eg:trop2}, $\mathsf{T}(\Sigma,w)_0$ is given by the closure of $\sigma_0$ with the induced tropical structure. 
 
 For each face $\tau \in \mathrm{SD}(w)_0$ we obtain the induced Laurent polynomial
\[
w_\tau = 1+\sum_{m \in A'_\tau} u_{m}(t) \underline{x}^{m}.
\]
Note that the Newton polytope of $w_\tau$ is $\tau$, and let $T_{L(\tau)}=\mathrm{Spec}(\mathbb{C}[L(\tau) \cap M])$. Since coefficients are chosen generically, $H^i(T_{L(\tau)},w_\tau) =0$ unless $i = \dim \tau$. We choose a holomorphic volume form $\mathrm{Vol}(L(\tau))$ on $T_{L(\tau)}$ for each $\tau$ so that they are compatible under the residue maps. We view $c^\perp \otimes \mathbb{C}$ as a filtered vector space by identifying $c^\perp$ with $c^\perp\otimes \mathbb{Z}(-1)$ where $\mathbb{Z}(-1)$ denotes the Tate Hodge structure. Therefore, $\wedge^p c^\perp$ is equipped with an induced filtration. 
\begin{defn}\label{d:rejac}
    For $k \in \mathbb{Z}_{\geq 0}$, the {\em tropical Jacobian sheaf of degree $k$} on $\mathsf{T}(\Sigma, w)_0$ is the filtered tropical sheaf of $\mathbb{C}$-vector spaces for which
    \[
        {\bf J}^k(c,\tau) = H^{\dim \tau}(T_{L(\tau)},w_\tau)(\dim \tau) \otimes (\wedge^{k-\dim \tau} c^\perp \wedge \mathrm{Vol}(L(\tau))).
    \]
    Observe that $L(\tau) \subseteq L(f_\tau)$ and $c\subseteq \mathrm{nc}(f_\tau) \subseteq f_\tau^\perp$, so $L(\tau) \subseteq c^\perp$. 
    Let
    \begin{align*}
    {\bf J}^k((c_1,\tau_1),(c_2,\tau_2))  : {\bf J}^k(c_1,\tau_1) &\longrightarrow {\bf J}^k(c_2,\tau_2) \\
    \alpha \otimes (\gamma \wedge \mathrm{Vol}(L(\tau_1))) &\longmapsto \mathrm{res}_{T_{L(\tau_1)}}^{T_{L(\tau_2)}}(\alpha) \otimes \varpi(\gamma \wedge \mathrm{Vol}(L(\tau_1))).
    \end{align*}
    Here $\mathrm{res}$ is the residue map and $\varpi$ is defined in Section \ref{s:linalg}. We define the trivial complex ${\bf J}^\bullet:=\bigoplus_{k \in \mathbb{Z}_{\geq 0}}{\bf J}^k[-k]$, and call it the \emph{tropical Jacobian sheaf} on $\mathsf{T}(\Sigma, w)_0$. Because the sheaf morphisms preserve the irregular Hodge filtrations, we can equip ${\bf J}^\bullet$ with the corresponding decreasing filtration denoted by $\mathsf{F}^\bullet$. 
\end{defn}
\begin{remark}
    Note that this definition does not extend to give a sheaf on $\mathsf{T}(\Sigma,w)$; for instance, the definition of ${\bf J}^k(c,\tau)$ fails when $\tau$ does not contain $0_M$. 
\end{remark}

The following statement is a direct consequence of Proposition \ref{p:adolphson-sperber} and Definition \ref{d:rejac}.
\begin{proposition}
    We have the following filtered identifications:
\begin{align*}
({\bf J}^k(c,\tau),F_\mathrm{irr}^\bullet) &\cong ({\bf B}_{\tau} \otimes (\wedge^{k-\dim \tau} c^\perp \wedge \mathrm{Vol}(L(\tau))),\mathsf{F}^\bullet) \\
{\bf J}^k((c_1,\tau_1),(c_2,\tau_2)) &= p_{(\tau_1,\tau_2)} \otimes \varpi.
\end{align*}
In other words, we can rewrite the tropical Jacobian sheaf in a purely combinatorial way.
\end{proposition}
We write the tropical cohomology of the tropical Jacobian sheaf $\mathbf{J}^\bullet$ as 
\[
H^n(\mathsf{T}(\Sigma, w)_0):=H^n(\mathsf{T}(\Sigma, w)_0, \mathbf{J}^\bullet)=\oplus_{a+b =n} H^a(\mathsf{T}(\Sigma,w)_0,{\bf J}^b). 
\]
Later on, in Theorem \ref{t:descent}, we will explain that this is isomorphic to the associated graded pieces of the irregular Hodge filtration on $H^n(T(\Sigma),w)$. This motivates the following definition:
\[
h^{\lambda,\mu}(\mathsf{T}(\Sigma,w)_0) := \sum_{a+b=\lambda+\mu} \dim \gr_{\mathsf{F}}^\lambda H^a(\mathsf{T}(\Sigma,w)_0,{\bf J}^b).    
\]
Because all $\tau$ are simplicial, the gradings on ${\bf B}_\tau$ induces a grading on ${\bf J}^\bullet$ which induces $\mathsf{F}^\bullet$. We may let 
\[
{\bf J}^{\lambda,\mu}(c,\tau) := {\bf B}_\tau^{\lambda} \otimes \wedge^{\lambda+\mu-\dim \tau} c^\perp \wedge \mathrm{Vol}(L(\tau))
\]
for $\lambda,\mu \in \mathbb{Q}_{\geq 0}$. This grading is chosen so that ${\bf J}^{\lambda,\mu}(c,\tau)=\gr_{\mathsf{F}}^\lambda{\bf J}^{\lambda+\mu}(c,\sigma)$. Thus 
\[
h^{\lambda,\mu}(\mathsf{T}(\Sigma,w)_0) =\sum_{n} \dim H^n(\mathsf{T}(\Sigma,w)_0,{\bf J}^{\lambda,\mu-n}).    
\]
We define the \emph{tropical $(\lambda,\mu)$-Hodge space of $\mathsf{T}(\Sigma, w)_0$} to be
\begin{equation}\label{e:trophodge}
    H^{\lambda,\mu}(\mathsf{T}(\Sigma,w)_0) := \bigoplus_{n} H^n(\mathsf{T}(\Sigma,w)_0,{\bf J}^{\lambda,\mu-n})
\end{equation}
and refer to the union of such spaces as the \emph{tropical Hodge spaces of $\mathsf{T}(\Sigma, w)_0$}.

\subsection{Tropical orbifold cohomology}\label{s:torbsheaf}

Let $(T(\Sigma), w)$ be defined as in the previous section. Suppose we have a simplicial stacky fan ${\bf \Sigma}$ whose underlying fan is $\Sigma$, and let $T(\mathbf{\Sigma})$ be the associated smooth toric Deligne--Mumford stack (see Section \ref{s:tstack}). Then, we obtain a stacky Landau--Ginzburg model $(T(\mathbf{\Sigma}), w)$, where we abuse notation by using $w$ to denote the composition of the canonical morphism $T(\mathbf{\Sigma}) \to T(\Sigma)$ and $w$. We will extend the previous discussion to this stacky setting.

First, we simply define the tropicalization of the pair $(T({\bf \Sigma}), w)$ to be the tropicalization of the underlying pair $(T(\Sigma),w)$. In other words, $\mathsf{T}({\bf \Sigma}, w):=\mathsf{T}(\Sigma, w)$ and $\mathsf{T}({\bf \Sigma}, w)_0:=\mathsf{T}(\Sigma, w)_0$. Since the orbifold cohomology of $(T({\bf \Sigma}), w)$ depends on the stacky structure of ${\bf \Sigma}$, we use $\mathsf{T}(\mathbf{\Sigma}, w)$ to emphasize that we consider the stacky version. For each cone $c$ of ${\bf \Sigma}$, we let $\mathsf{T}({\bf \Sigma},w)_{0,c}$ denote the closure of the tropical toric stratum attached to $c$, and we let ${\bf J}^\bullet_c$ be the restriction of ${\bf J}^\bullet$ to $\mathsf{T}({\bf \Sigma},w)_{0,c}$. By imitating the definition of orbifold cohomology, we define the \emph{orbifold tropical cohomology} of the tropical Jacobian sheaf on $\mathsf{T}({\bf \Sigma}, w)_0$ to be 
\begin{equation}\label{eq:troppq}
    H_\mathrm{orb}^n(\mathsf{T}({\bf \Sigma},w)_0) = \bigoplus_{c \in {\bf \Sigma}}\bigoplus_{g \in \mathrm{Box}^\circ(c)} H^{n-2\iota(g)}(\mathsf{T}({\bf \Sigma},w)_{0,c},{\bf J}^\bullet_c).
\end{equation}
We remark that the index $n$ need not be an integer in~\eqref{eq:troppq}. 

As an orbifold analogue of \eqref{e:trophodge}, we define the \emph{orbifold tropical $(\lambda,\mu)$-Hodge space of $\mathsf{T}({\bf \Sigma},w)$} to be 
\begin{equation*}
    H^{\lambda,\mu}_\mathrm{orb}(\mathsf{T}({\bf \Sigma},w)_0) := \bigoplus_{c \in {\bf \Sigma}}\bigoplus_{g \in \mathrm{Box}^\circ(c)} H^{\lambda-\iota(g),\mu-\iota(g)}(\mathsf{T}({\bf \Sigma},w)_{0,c},{\bf J}^\bullet_c)
\end{equation*}
and refer to the union of such spaces as the \emph{orbifold tropical Hodge spaces of $\mathsf{T}({\bf \Sigma},w)_0$}.

\begin{remark}
    Given a simplicial stacky fan $\mathbf{\Sigma}=(\Sigma, \beta)$, one can also define another space 
    \[
    H_\mathrm{orb}^n(\mathsf{T}({\Sigma},w)_0):= \bigoplus_{c \in {\Sigma}}\bigoplus_{g \in \mathrm{Box}^\circ(c)} H^{n-2\iota(g)}(\mathsf{T}({\Sigma},w)_{0,c},{\bf J}^\bullet_c). 
    \]
   As noted in Remark \ref{r:orbgr}, the orbifold structures of $T(\Sigma)$ and $T(\mathbf{\Sigma})$ differ, and thus this space is distinct from the one described in \eqref{eq:troppq}. Specifically, $H_\mathrm{orb}^n(\mathsf{T}({\Sigma},w)_0)$ computes the orbifold cohomology of the underlying Landau–Ginzburg model $(T(\Sigma), w)$.
\end{remark}

Next, we present the orbifold tropical cohomology and orbifold tropical Hodge spaces of $\mathsf{T}({\bf \Sigma},w)$ as the usual tropical cohomology of a tropical sheaf on $\mathsf{T}({\bf \Sigma},w)_0$. We may also define a bigraded sheaf, ${\bf J}^{\bullet, \bullet}_\mathrm{orb}$, on $\mathsf{T}({\bf \Sigma},w)_0$:
\begin{align*}
{\bf J}^{\lambda,\mu}_\mathrm{orb}(c,\tau) & := \bigoplus_{\substack{k-m+\ell = \lambda \\ \ell+m=\mu}} {\bf B}^\ell_c\otimes {\bf B}_{\tau}^{m}\otimes \wedge^{k-\dim\tau} c^\perp \wedge \mathrm{Vol}(L(\tau)) \\
{\bf J}^{\lambda,\mu}_\mathrm{orb}((c_1,\tau_1),(c_2,\tau_2)) & := p_{(c_1,c_2)} \otimes p_{(\tau_1,\tau_2)} \otimes \varpi.
\end{align*}
The grading above is chosen so that it agrees with the grading on ${\bf J}(c,\tau)$ when ${\bf B}_c^\ell = \mathbb{C}$. Note that there is an extra grading, indexed by $\ell$, which corresponds to the age grading. We know that the morphisms $p$ and $\varpi$ preserve gradings, so we obtain a graded sheaf from this data and we denote it by ${\bf J}_\mathrm{orb}^\bullet$. Specifically, we have ${\bf J}_\mathrm{orb}^k=\bigoplus_{\lambda+\mu=k}{\bf J}^{\lambda,\mu}_{\mathrm{orb}}$ for $k \in \mathbb{Z}_{\geq 0}$.

We define 
\[
H^{\lambda,\mu}(\mathsf{T}({\bf \Sigma},w)_0,{\bf J}_\mathrm{orb}^\bullet) := \bigoplus_{n} H^n(\mathsf{T}({\bf \Sigma},w)_0,{\bf J}_\mathrm{orb}^{\lambda,\mu-n}).  
\]

\begin{proposition}\label{p:hodgeorb}
    There are isomorphisms of $\mathbb{C}$-vector spaces
    \begin{align*}
        H^n_\mathrm{orb}(\mathsf{T}({\bf \Sigma},w)_0) & \cong H^n(\mathsf{T}({\bf \Sigma},w)_0,{\bf J}^\bullet_\mathrm{orb}) \\
        H^{\lambda,\mu}_\mathrm{orb}(\mathsf{T}({\bf \Sigma},w)_0) & \cong H^{\lambda,\mu}(\mathsf{T}({\bf \Sigma},w)_0,{\bf J}^\bullet_\mathrm{orb}).
    \end{align*}
\end{proposition}
\begin{proof}
    For $g \in \mathrm{Box}(c)$ define $\gamma(g)$ to be the dimension of the minimal stratum which contains $g$. Then  
    \[
    \coprod_{\substack{c'\subseteq c \\ \dim c' = i}} \mathrm{Box}^\circ({c'}) = \{ g\in \mathrm{Box}(c) \mid \gamma(g) = \dim c'\}.
    \]
    The grading on ${\bf B}_c$ induced by $\gamma$ is preserved by $p_{(c_1,c_2)}$ therefore, we obtain yet another grading on ${\bf J}^\bullet_\mathrm{orb}$ itself induced by $\gamma$. In particular,
    \begin{align*}
    ({\bf B}_c^\ell \otimes {\bf B}^m_\tau \otimes \wedge^{k-\dim \tau}c^\perp \wedge \mathrm{Vol}(L(\tau)))_{i} &= \bigoplus_{\substack{c'\subseteq c \\ \dim c' = i}} \bigoplus_{\substack{g \in \mathrm{Box}^\circ({c'}) \\ \iota(g) = \ell}}{\bf B}^m_\tau\otimes \wedge^{k-\dim \tau} c^\perp \wedge \mathrm{Vol}(L(\tau))[\ell] \\
    &\cong \bigoplus_{\substack{c'\subseteq c \\ \dim c' = i}} \bigoplus_{\substack{g \in \mathrm{Box}^\circ({c'}) \\ \iota(g) =\ell}} {\bf J}^{k-m,m}(c,\tau)[\ell]
    \end{align*}
    Here, the subscript $i$ refers to the grading induced by $\gamma$. Taking direct sums on both sides leads to an isomorphism of sheaves
    \[
    ({\bf J}^{\lambda,\mu}_\mathrm{orb})_i \cong \bigoplus_{\dim c'=i}\bigoplus_{g\in \mathrm{Box}^\circ({c'})} ({\bf J}^{\lambda-\iota(g),\mu-\iota(g)}|_{\mathsf{T}({\bf \Sigma}, w)_{c'}})[\iota(g)].
    \]
    Both statements in the proposition then follow by taking cohomology.
\end{proof}
\begin{remark}
    The situation simplifies enormously if one assumes that ${\bf \Sigma}=(\Sigma, \underline{1})$ is smooth and all simplices in $\mathrm{SD}(w)$ are unimodular. Then ${\bf J}^{\lambda,\mu} = 0$ unless $\mu=0$ and ${\bf J}^{\lambda,0} = \wedge^{\lambda-\dim \tau} c^\perp \wedge \mathrm{Vol}(L(\tau))$ and $H^{\lambda,\mu}(\mathsf{T}(\Sigma,w)_0) = H^\mu(\mathsf{T}(\Sigma,w)_0,{\bf J}^{\lambda,0})$. 
\end{remark}

\section{Combinatorial, tropical, and geometric duality}\label{s:clamp}

In this section, we prove two combinatorial theorems regarding the tropical orbifold cohomology of Clarke dual pairs. The geometric interpretation for the induced Clarke mirror pairs of Landau–Ginzburg models is discussed, with the proof deferred to the last two sections of this article. Two applications are presented: Hodge number duality for weak Fano toric Deligne–Mumford stacks and Berglund–Hübsch–Krawitz mirrors.

\subsection{Clarke duality}

In \cite{clarke2016dual}, Clarke introduces a very broad and flexible form of mirror symmetry of dual fans which we slightly generalize here. Let $\Sigma$ and $\check{\Sigma}$ be fans consisting of strongly convex rational polyhedral fans in rank $d$-lattices $M$ and $N$, respectively. We say that ${\Sigma}$ and $\check{{\Sigma}}$ are \emph{Clarke dual} if 
\begin{equation}\label{e:clarke}
\langle m, n \rangle \geq 0 \text{ for all } m \in \mathrm{Supp}({\Sigma}), n \in \mathrm{Supp}(\check{{\Sigma}}).
\end{equation}

\noindent This definition can be easily extended to a pair of stacky fans. However, there are several conditions that we apply to our fans that allow us to strengthen our results and, ultimately, connect our results to Hodge numbers of {\em geometric} Landau--Ginzburg models. Recall that for each $c \in {\bf \Sigma}$, $\Delta_c=\mathrm{Conv}(c[1] \cup \{0\})$ where the ray generators of $c$ are the scaled ones in ${\bf \Sigma}[1]$. 

\begin{defn}
    We say that ${\bf \Sigma}$ is \emph{convex} if the polyhedral complex 
     \[
        \Delta_{\bf\Sigma} = \bigcup_{{c}\in \bf{\Sigma}}\Delta_{{c}}
    \]
    has support $\mathrm{Supp}(\Delta_{\bf \Sigma})$ which is a convex polytope. In this case, the simplices $\Delta_c,c\in \Sigma$ form a triangulation of $\Delta_{\bf \Sigma}$.
\end{defn}

\begin{defn}\label{d:adjectives}
    Let $({\bf \Sigma}, {\bf \check{\Sigma}})$ be a pair of quasiprojective, simplicial stacky fans. We say that $({\bf \Sigma}, {\bf \check{\Sigma}})$ is Clarke dual if the following two conditions hold:
    \begin{enumerate}
        \item The relation \eqref{e:clarke} holds. Equivalently,  $\langle m, n \rangle \geq 0 \text{ for all } m \in \mathrm{Supp}({\bf{\Sigma}}), n \in \mathrm{Supp}(\bf{\check{{\Sigma}}}).$
        \item Both ${\bf \Sigma}$ and ${\bf \check{\Sigma}}$ are convex. 
    \end{enumerate}
\end{defn}
Note that we explicitly do not require $\mathrm{Supp}({\bf \Sigma})$ to be the dual cone of $\mathrm{Supp}(\check{\bf \Sigma})$.

\begin{example}
From Example \ref{eg:trop} and Example \ref{eg:trop2}, let us take $\check{\Sigma}$ to be the fan generated by ${(0,1), (1,1), (2,1)}$, and let $\Sigma$ be as introduced earlier. With trivial stacky structures, the pair $(\Sigma, \check{\Sigma})$ forms a Clarke dual pair.
\end{example}

\begin{example}
    Suppose we have a lattice $M \cong \mathbb{Z}^2$ with basis $e_1,e_2$ and a dual basis $e_1^*, e_2^*$ of $N = \Hom_\mathbb{Z}(M,\mathbb{Z})$. 
    \[
    \Sigma = \{\mathrm{cone}(e_1, e_2), \mathrm{cone}(e_2,-e_1 + 3e_2)\},\qquad \check{\Sigma} = \{\mathrm{cone}(e^*_1+e_2^* ,e^*_1+2e_2),\mathrm{cone}(e^*_1+2e^*_2, e^*_2 )\} 
    \]
    These two fans satisfy Definition \ref{d:adjectives}(1) but $\Sigma$ is not convex, so Definition \ref{d:adjectives}(2) is not satisfied. On the other hand, the fans 
    \[
    {\bf \Sigma} = \{\mathrm{cone}(e_1, 3e_2), \mathrm{cone}(3e_2,-e_1 + 3e_2)\},\qquad \check{\Sigma} = \{\mathrm{cone}(e^*_1+e_2^* ,e^*_1+2e^*_2),\mathrm{cone}(e^*_1+2e^*_2, e^*_2 )\} 
    \]
    satisfy Definition \ref{d:adjectives}(1) and (2). Note that ${\bf \Sigma} = (\Sigma, \{1,3,1\})$. So this data provides a pair of Clarke dual stacks.
\end{example}

Given a Clarke dual pair of stacky fans, \eqref{e:clarke} implies that the Laurent polynomial
\[
w({\bf \Sigma}) = 1 + \sum_{m \in {\bf \Sigma}[1]} u_m \underline{x}^m \in \mathbb{C}[M]
\]
determines a regular function on $T(\check{\Sigma})$ and therefore a regular function on $T(\check{\bf \Sigma})$ obtained by composing the canonical morphism $T(\check{\bf \Sigma})\rightarrow T(\check{\Sigma})$ with $w(\bf{\Sigma})$. The pair of stacky Landau--Ginzburg models
\[
(T({\bf \Sigma}),w(\check{\bf \Sigma})),\qquad (T(\check{\bf \Sigma}), w({\bf \Sigma}))
\]
is called a \emph{Clarke mirror pair of stacky Landau--Ginzburg models}. The justification for this terminology, particularly the term "mirror," will be provided in the subsequent sections (see Theorem \ref{t:clarke}). Clarke mirror pairs encompass numerous well-known instances of mirror symmetry between toric objects.

\begin{remark}
In this article, we always assume that all fans are simplicial. There are two reasons for this. First without this assumption, we are forced to work with properly stringy Hodge numbers instead of orbifold Hodge numbers. Our geometric results rely on constructing toroidal degenerations of pairs. It is easy to see that orbifold Hodge numbers behave well under degeneration, however, we are not aware of any literature on how stringy Hodge numbers behave under degeneration. While this is an interesting question, we do not address it here. As a consequence, without the simplicial assumption, the geometric version of our results can only be made precise in the case already studied by Batyrev--Borisov and Borisov--Mavlyutov \cite{batyrev1996mirror,borisov2003string}. Second, any non-simplicial fan admits a simplicial, quasiprojective refinement obtained without altering $\Delta_{\bf \Sigma}$, so, ultimately, nothing will be lost. 

The combinatorial Clarke duality result in Section \ref{s:combcldual} relies on nothing but the simplicial assumption while our  tropical and geometric Clarke duality results relies further on the quasiprojective and convex assumptions (Section \ref{s:tropcldual}, \ref{s:topcldual}). 
\end{remark}

\subsection{Combinatorial Clarke pairs and duality}\label{s:combcldual}

Following the notation of Borisov \cite{borisov2014stringy}, let
\[
({\bf \Sigma} \oplus \check{\bf \Sigma})_0 = \{(c_1,c_2) \in {\bf \Sigma} \times \check{\bf \Sigma} \mid \langle c_1,c_2\rangle =0 \}.
\]
The underlying fan is a subfan of $\Sigma \times \check{\Sigma}$, and it is equipped with a pair of stacky structures $(\beta, \check{\beta})$. We define a graded sheaf on the underlying polyhedral space of $({\bf \Sigma} \oplus \check{\bf \Sigma})_0$. The topological space of this rational fan given by its Alexandrov topology (Remark \ref{r:fanalexandrov}) is often called a fanfold or a fan space (e.g. \cite{aksnes2023cohomologically,barthel2002combinatorial,brion1997structure}). A basis of the topology of the open stars on $({\bf \Sigma} \oplus \check{\bf \Sigma})_0$ is given by open stars of closed strata. 

We let $(c_1,c_2) \preceq (c_1',c_2')$ if and only if $c_1'\subseteq c_1$ and $c_2' \subseteq c_2$. To the pair $(c_1,c_2)$, we assign two graded vector spaces $(\wedge^\bullet c_1^\perp)\wedge \mathrm{Vol}(L(c_2))$ and $(\wedge^\bullet c_2^\perp) \wedge \mathrm{Vol}(L(c_1))$ for a fixed choice of volume forms on $c_2$ and $c_1$. If $(c_1,c_2) \preceq (c_1', c_2')$ then there is a canonical map
\[
\varpi_{(L(c_1),L(c_2))}^{(L(c_1'),L(c_2'))}: \wedge^{m-\dim c_2} c_1^\perp \wedge \mathrm{Vol}(L(c_2)) \longrightarrow \wedge^{m-\dim c_2' } (c_1')^\perp \wedge \mathrm{Vol}(L(c_2')).
\]
defined as in Section \ref{s:linalg}. This allows us to introduce the following definition. 
\begin{defn}\label{d:combsheaves}
    We define the following bigraded cellular sheaves of abelian groups on the polyhedral complex $({\bf \Sigma} \oplus \check{\bf \Sigma})_0$. Let 
    \begin{align*}
   {\bf \Xi}(c_1,c_2)^{\lambda,\mu} & := \bigoplus_{\substack{a-b+l =\lambda\\ b+l = \mu}}(\wedge^{a-\dim c_2} c_1^{\perp})\wedge \mathrm{Vol}(L(c_2)) \otimes {\bf B}_{c_1}^b \otimes {\bf B}_{c_2}^l \\ 
    \check{\bf \Xi}(c_1,c_2)^{\lambda,\mu} & := \bigoplus_{\substack{a-b+l =\lambda\\ b+l = \mu}}(\wedge^{a-\dim c_1} c_2^{\perp})\wedge \mathrm{Vol}(L(c_1)) \otimes {\bf B}_{c_2}^b \otimes {\bf B}_{c_1}^l.
    \end{align*}
    If $(c_1,c_2)\preceq (c_1',c_2')$ we define restriction maps
    \begin{align*}
    {\bf r}_{(c_1,c_2)}^{(c_1',c_2')} & = \varpi_{(L(c_1),L(c_2))}^{(L(c_1'),L(c_2'))} \otimes p_{(c_1,c_1')}\otimes p_{(c_2,c_2')} : {\bf \Xi}(c_1,c_2)^{\lambda,\mu}\longrightarrow {\bf \Xi}(c_1',c_2')^{\lambda,\mu} \\
    \check{\bf r}_{(c_1,c_2)}^{(c_1',c_2')} & = \varpi_{(L(c_2),L(c_1))}^{(L(c_2'),L(c_1'))} \otimes p_{(c_2,c_2')}\otimes p_{(c_1,c_1')} : \check{\bf \Xi}(c_1,c_2)^{\lambda,\mu}\longrightarrow \check{\bf \Xi}(c_1',c_2')^{\lambda,\mu}.
    \end{align*}
    As in Remark \ref{r:alexandrov-sheaf}, ${\bf \Xi}^{\lambda,\mu}$ and $ \check{\bf \Xi}^{\lambda,\mu}$ do not define sheaves on $({\bf \Sigma} \oplus \check{\bf \Sigma})_0$ {\em per se}, but they specify unique sheaves on the fan space of  $({\bf \Sigma} \oplus \check{\bf \Sigma})_0$.
\end{defn}
\begin{remark}
    Note that the underlying space $({\bf \Sigma}\oplus  \check{\bf \Sigma})_0$ does not depend on the stacky structure, but the sheaves ${\bf \Xi}^{\lambda,\mu}$ and $\check{\bf \Xi}^{\lambda,\mu}$ depend on the stacky structure, because ${\bf B}_{c_1}$ and ${\bf B}_{c_2}$ are determined by $\{\beta_\rho \rho \mid \rho \in \Sigma[1]\}$ and $\{\check{\beta}_{\check{\rho}} \check{\rho} \mid \check{\rho} \in \check{\Sigma}[1]\}$, respectively.
\end{remark}

We now construct complexes computing the cohomology of ${\bf \Xi}$ and $\check{\bf \Xi}$. In Section \ref{s:tropcldual}, they will be explicitly identified with particular \v{C}ech cohomology complexes of ${\bf J}_{\mathrm{orb}}$ and $\check{\bf J}_{\mathrm{orb}}$, respectively. First, index all cells $({\bf \Sigma} \oplus \check{\bf \Sigma})_0 = \{(\tau_i,\check{\tau}_i) \mid 1 \leq i \leq |({\bf \Sigma} \oplus \check{\bf \Sigma})_0|\}$. Define
\[
\mathfrak{U}_n := \{ (c_1,c_2) \in ({\bf \Sigma} \oplus \check{\bf \Sigma})_0 \mid (c_1,c_2) = \cap_{j \in  I} (\tau_j, \check{\tau}_j), |I| = n+1\}.
\]
Let
\[
\check{\mathcal{C}}^n(\mathfrak{U},{\bf \Xi}^{\lambda,\mu}) = \bigoplus_{(c_1,c_2) \in \mathfrak{U}_n} {\bf \Xi}(c_1,c_2)^{\lambda,\mu},\qquad \check{\mathcal{C}}^n(\mathfrak{U},\check{\bf \Xi}^{\lambda,\mu}) = \bigoplus_{(c_1,c_2) \in \mathfrak{U}_n} \check{\bf \Xi}(c_1,c_2)^{\lambda,\mu}
\]
and we define $\delta_{n} : \check{\mathcal{C}}^n(\mathfrak{U},{\bf \Xi}^{\lambda,\mu}) \rightarrow \check{\mathcal{C}}^{n+1}(\mathfrak{U},{\bf \Xi}^{\lambda,\mu})$ and $\check{\delta}_n: \check{\mathcal{C}}^n(\mathfrak{U},\check{\bf \Xi}^{\lambda,\mu}) \rightarrow \check{\mathcal{C}}^{n+1}(\mathfrak{U},\check{\bf \Xi}^{\lambda,\mu})$ by using the  restriction maps ${\bf r}$ and $\check{\bf r}$ along with the usual \v{C}ech sign conventions according to the indexing on cells of $({\bf \Sigma} \oplus \check{\bf \Sigma})_0$. Finally, we  define
\begin{equation}\label{eq:pqcomb}
H^{\lambda,\mu}({\bf \Sigma} \oplus \check{\bf \Sigma})_0 := \bigoplus_n H^{n}({\bf \Xi}^{\lambda,\mu-n}),\qquad \check{H}^{\lambda,\mu}({\bf \Sigma} \oplus \check{\bf \Sigma})_0 := \bigoplus_n H^{n}(\check{\bf \Xi}^{\lambda,\mu-n}).
\end{equation}
\begin{remark}
    If we replace $(c_1,c_2)$ with their open stars in  the definition of $\mathfrak{U}_n$, then $\mathfrak{U}_n$ becomes the finest possible \v{C}ech cover of $({\bf \Sigma} \oplus \check{\bf \Sigma})_0$ and $H^n({\bf \Xi}^{\lambda,\mu})$ is just the cohomology of the complex $(\check{\mathcal{C}}^\bullet(\mathfrak{U},{\bf \Xi}^{\lambda,\mu}),\delta_\bullet)$.
\end{remark}
The following result is rather simple, but fundamental to the computations that follow.
\begin{theorem}\label{t:combdual}
    Suppose $({\bf \Sigma,\check{\Sigma}})$ is a Clarke dual pair. Contraction with $\mathrm{Vol}(M)$ and switching  the components of $({\bf \Sigma} \oplus \check{\bf \Sigma})_0$ induce an isomorphism of bigraded pieces of cellular sheaves 
    \[
    \mathrm{Vol}(M) \otimes \mathrm{sw} : {\bf{\Xi}}^{\lambda,\mu} \xlongrightarrow{\cong} \check{\bf{\Xi}}^{d-\lambda,\mu}
    \]
    for all $\lambda,\mu \in \mathbb{Q}$. Therefore, we have the induced isomorphisms of $\mathbb{C}$-vector spaces, which will be identified with orbifold tropical Hodge spaces (see Theorem \ref{t:tropidual}):
    \[
    \mathrm{Vol}(M) \otimes \mathrm{sw} : H^{\lambda,\mu}({\bf \Sigma} \oplus \check{\bf \Sigma})_0\xlongrightarrow{\cong} \check{H}^{d-\lambda,\mu}({\bf \Sigma} \oplus \check{\bf \Sigma})_0.
    \]
    
\end{theorem}
\begin{proof}
    First, we observe that $\mathrm{Vol}(M) \otimes \mathrm{sw}$ induces an isomorphism between ${\bf \Xi}(c_1,c_2)$ and $\check{\bf \Xi}(c_1,c_2)$ for each $(c_1,c_2)$ and that this map commutes with ${\bf r}$ and $\check{\bf r}$ by Lemma \ref{l:linalg}. We see that:
    \begin{align*}
    {\bf \Xi}(c_1,c_2)^{\lambda,\mu} &=\bigoplus_{\substack{a-b+l=\lambda \\b+l =\mu}} (\wedge^{a-\dim c_2} c_1^\perp) \wedge \mathrm{Vol}(L(c_2)) \otimes {\bf B}^b_{c_1} \otimes {\bf B}^l_{c_2} \\ &\cong \bigoplus_{\substack{d-(a-l+b) = \lambda \\l+b=\mu}}(\wedge^{d-a-\dim c_1} c_2^\perp) \wedge \mathrm{Vol}(L(c_1)) \otimes {\bf B}^l_{c_1} \otimes {\bf B}^b_{c_2} \\ 
    & = \check{\bf \Xi}(c_1,c_2)^{d-\lambda,\mu}.
    \end{align*}
    This provides the required isomorphism of sheaves. Consequently,
    \begin{equation}\label{eq:hdual}
    H^{\lambda,\mu}({\bf \Sigma} \oplus \check{\bf \Sigma})_0 = \bigoplus_nH^n({\bf \Xi}^{\lambda,\mu-n}) \cong \bigoplus_n H^n(\check{\bf \Xi}^{d-\lambda,\mu-n}) = \check{H}^{d-\lambda,\mu}({\bf \Sigma} \oplus \check{\bf \Sigma})_0.
    \end{equation}
    This completes the proof.
\end{proof}
\begin{remark}
    In \cite{borisov2003string}, Borisov and Mavlyutov prove a similar result for dual pairs of Gorenstein cones. Their result is more general, in the sense that it does not require the simplicial assumption, however, it requires that $\Sigma$ and $\check{\Sigma}$ both be cones. It seems likely to us that it is possible to combine their approach with ours, however, we do not pursue this here.
\end{remark}
\begin{remark}
If both $\Sigma$ and $\check{\Sigma}$ are smooth fans with trivial stacky structure then ${\bf \Xi}^{\lambda,\mu} = 0$ unless $\mu=0$. Lemma \ref{l:linalg} tells us that ${\bf \Xi}^{\lambda,0} \cong \check{\bf \Xi}^{d-\lambda,0}$. The duality statement \eqref{eq:hdual} becomes 
\[
    H^{\lambda,\mu}({\bf \Sigma} \oplus \check{\bf \Sigma})_0 = H^\mu(\check{\mathcal{C}}^\bullet(\mathfrak{U},{\bf \Xi}^{\lambda,0}), \delta_\bullet) \cong H^\mu(\check{\mathcal{C}}^\bullet(\mathfrak{U},\check{\bf \Xi}^{d-\lambda,0}), \check{\delta}_\bullet) = \check{H}^{d-\lambda,\mu}({\bf \Sigma} \oplus \check{\bf \Sigma})_0. 
\]
\end{remark}

\subsection{Tropical Clarke pairs and duality}\label{s:tropcldual}

In this section, we will show that the combinatorial duality described in the previous section can be extended to give a duality between the tropical Hodge numbers of certain Landau--Ginzburg models attached to a pair of Clarke dual fans. Assume ${\bf \Sigma}$ is quasiprojective. We may find a strictly convex ${\bf\Sigma}$-linear function $\varphi : M_\mathbb{R} \rightarrow \mathbb{R}$ which is integral on ${\bf\Sigma}[1]$. If we assume that ${\bf \Sigma}$ is simplicial, $\varphi$ is equivalent to a function $\varphi: {\bf \Sigma}[1] \rightarrow \mathbb{Z}$. Attached to $\varphi$ there is a tropical regular function,
\[
   \mathsf{w}_\varphi : N_\mathbb{R} \longrightarrow \mathbb{R},\qquad n \longmapsto \min_{m \in {A}}\{\langle  n,m \rangle + \varphi(m)\}.
    \]
    where $A={\bf \Sigma}[1] \cup \{0\}$.
Observe that $\mathsf{w}_\varphi$ is the tropicalization of the Laurent polynomial
\[
w_\varphi=w({\bf \Sigma})_\varphi := \sum_{m\in {A}} u_mt^{\varphi(m)}\underline{x}^m
\]
for generic non-vanishing coefficients $u_m \in \mathbb{C}^*$ and $u_0 =1$. Our first result describes the relationship between the tropicalization of $w_\varphi$ and ${\bf \Sigma}$. Recall Definition \ref{d:adjectives} for the terminology used in the following proposition.

\begin{proposition}\label{p:simplicial}
    Let ${\bf \Sigma}$ be a simplicial, convex, quasiprojective stacky fan. Then $\Delta_{\bf \Sigma} = \mathrm{SD}(w_\varphi)$. Therefore, $\mathrm{SD}(w_\varphi)$ is a star triangulation based at $0_M$.
\end{proposition}
\begin{proof}
    Let $A = {\bf \Sigma}[1] \cup \{0\}$ and $\widetilde{A} = \{(m,\varphi(m)) \in M_\mathbb{R} \times \mathbb{R} \mid m\in A\}\subseteq \Gamma(\varphi)$ where $\Gamma(\varphi)$ denotes the graph of $\varphi$. In this proof, the phrase {\em lower convex hull} indicates subset of $\mathrm{Conv}(\tilde{A})$
    \[
    \{(q,x) \in  \mathrm{Conv}(\tilde{A}) \mid x \leq x' \quad \forall (q,x') \in \mathrm{Conv}(\tilde{A})\}.
    \]
    The projection of faces of the lower convex hull of $\tilde{A}$ onto $M_\mathbb{R}$ is $\mathrm{SD}(w_\varphi)$.
    
    Let $(q,x) \in \mathrm{Conv}(\widetilde{A})$. Then $(q,x) = \sum_{m\in A}a_m(m,\varphi(m))$ for values $a_m \in [0,1]$ and $\sum_{m\in A} a_m = 1$. Because $\varphi$ is convex, $x=\sum_{m\in A} a_m \varphi(m) \geq \varphi(q)$. In other words, if $q \in \Delta_{\bf \Sigma}$ then  $\varphi(q)$ lies below the convex hull. If $q\in \Delta_c$ then, since $\varphi$ is linear on $\Delta_c$, along with the same calculation as above, $(q,\varphi(q)) \in \mathrm{Conv}(\widetilde{A})$. In other words, $\Gamma(\varphi|_{\Delta_c})$ is in the convex hull of $\mathrm{Conv}(\widetilde{A})$ thus it must be in the lower convex hull of $\mathrm{Conv}(\widetilde{A})$. By strict convexity, each $\Gamma(\varphi|_{\Delta_c})$ must be a distinct face of the lower convex hull. Since $\Delta_{\bf \Sigma}$ is convex, it is a triangulation of $\Delta(w_\varphi)$. This completes the proof.
\end{proof}

According to Theorem \ref{t:subdiv} and the following discussion, the cells of $\mathsf{T}(\check{{\bf \Sigma}},w_\varphi)_0$ are in bijection with
    \begin{equation*}
    \{(\check{c},\tau) \in \check{{\bf \Sigma}} \times \mathrm{SD}(w_\varphi)_0 \mid \check{c} \subseteq \mathrm{nc}(f_\tau)\}.
    \end{equation*}
    By Proposition \ref{p:simplicial}, $\mathrm{SD}(w_\varphi) =\Delta_{\bf \Sigma}$, so the cells of $\mathrm{SD}(w_\varphi)_0$ are of the form $\Delta_c = \mathrm{Conv}(c[1]\cup \{0\})$ for some cone $c$ of ${\bf \Sigma}$. 
    \begin{proposition}\label{p:cover}
    The following map is a bijection of posets.
    \[
    \mathsf{x}: (\check{\bf \Sigma} \oplus {\bf \Sigma})_0 \longrightarrow \mathsf{T}(\check{\bf \Sigma}, w_\varphi)_0,\qquad (\check{c},c) \longmapsto (\check{c}, {\Delta_c}).
    \]
    \end{proposition}
    \begin{proof}
   To see that this is well defined, suppose $\langle \check{c},c\rangle = 0$. This occurs if and only if $\langle \check{c},\Delta_c \rangle =0$. If there is some $n\in \check{c}$ and $m \in f_{\Delta_c}$ so that $\langle n, m \rangle >0$ then $\Delta_c \neq f_{\Delta_c}$. Since $\Delta_c$ is not a face of $f_{\Delta_c}$ there must also be some $m'\in f_{\Delta_c}$ so that $\langle n,m'\rangle <0$, contradicting the fact that ${\bf \Sigma}$ and ${\bf \check{\Sigma}}$ are Clarke dual. Thus $\langle \check{c}, f_{\Delta_c}\rangle = 0$ and $\check{c} \subseteq \mathrm{nc}(f_{\Delta_c})$ by definition. A similar argument proves the converse, that if $\check{c} \subseteq \mathrm{nc}(f_{\Delta_c})$ then $\langle \check{c},c\rangle = 0$.
\end{proof}
    The set
\[
\mathfrak{V} = \{\mathrm{St}(\mathsf{x}(\check{c},c)) \mid (\check{c},c) \in ({\bf \Sigma} \oplus \check{\bf \Sigma})_0\}
\]
is thus an open cover of $\mathsf{T}(\check{{\bf \Sigma}},w_\varphi)_0$. Observe that  $\mathrm{Vol}(L(\tau_c)) = \mathrm{Vol}(L(c))$. Then Proposition \ref{p:Cech-stars} allows us to identify
\begin{equation}\label{eq:jorb}
H^0(\mathrm{St}(\mathsf{x}(\check{c},{c})),{\bf J}^\bullet_\mathrm{orb})  = {\bf B}_{\check{c}} \otimes {\bf B}_{c} \otimes \wedge^\bullet \check{c}^\perp \wedge \mathrm{Vol}(L(c)).
\end{equation}
As a consequence we prove the main result of this section.

\begin{theorem}\label{t:tropidual}
    Suppose $({\bf \Sigma}, \check{\bf \Sigma})$ is a Clarke dual pair for which $\mathrm{rank}\, M = d$. Then there are canonical isomorphisms of orbifold tropical Hodge spaces
    \[
    H^{\lambda,\mu}_\mathrm{orb}(\mathsf{T}({\bf \Sigma},w_{\check{\varphi}})_0) \cong H^{d-\lambda,\mu}_\mathrm{orb}({\mathsf{T}}(\check{{\bf \Sigma}}, w_\varphi)_0)
    \]
    for $\lambda,\mu \in \mathbb{Q}$.
\end{theorem}
\begin{proof}
    It is enough to identify $H^{\lambda,\mu}_\mathrm{orb}(\mathsf{T}({\bf \Sigma}, w_{\check{\varphi}})_0)$ with $H^{\lambda,\mu}({\bf \Sigma} \oplus \check{\bf \Sigma})_0$. Then the desired result follows from Theorem \ref{t:combdual}. By Propositions \ref{p:cover} and \ref{p:Cech-stars}, $\mathfrak{V}$ forms a Leray cover so the \v{C}ech cohomology of ${\bf J}^\bullet_\mathrm{orb}$ with respect to $\mathfrak{V}$ is isomorphic to $H^*(\mathsf{T}({\bf \Sigma}, w_{\check{\varphi}})_0, {\bf J}^\bullet_\mathrm{orb})$. Index $({\bf \Sigma} \oplus \check{\bf \Sigma})_0$ as in the discussion following Definition \ref{d:combsheaves}. Applying \eqref{eq:jorb} we  find that the \v{C}ech complex is identified with the \v{C}ech complex ${\bf \Xi}^{\lambda,\mu}$ and thus there is a graded isomorphism $H^*({\bf \Xi})\cong H^*(\mathsf{T}({\bf \Sigma}, w_{\check{\varphi}})_0, {\bf J}^\bullet_\mathrm{orb})$. Finally, comparing definitions of $H^{\lambda,\mu}({\bf \Sigma} \oplus \check{\bf \Sigma})_0$ \eqref{eq:pqcomb} and $H^{\lambda,\mu}_\mathrm{orb}(\mathsf{T}({\bf \Sigma},w_{\check{\varphi}})_0)$ \eqref{eq:troppq} we obtain the desired identification.
\end{proof}

\subsection{Irregular duality for Clarke mirror pairs}\label{s:topcldual}

In this section, we extend Theorem \ref{t:tropidual} to a geometric duality for Clarke mirror pairs of Landau--Ginzburg models. Suppose that we have a Clarke dual pair of fans $({\bf \Sigma},\check{\bf \Sigma})$. We obtain a regular function
\[
w({\bf\Sigma})_\varphi =  1+\sum_{m \in{\bf \Sigma}[1]} u_mt^{{\varphi}(m)} \underline{x}^m
\]
on $T(\check{\bf \Sigma}) \times \mathbb{D}^*$ where $\mathbb{D}^*$ denotes a sufficiently small open punctured disc and $u_m \in \mathbb{C}^*$ denote generic non-vanishing coefficients. Here, {\em generic} means that for all $\varepsilon \in \mathbb{D}^*$, the fiber at $\varepsilon$ is a nondegenerate Landau--Ginzburg model. We denote by $(T({\bf\check{\Sigma}}), w({\bf \Sigma}))$ the fibre of this family over an arbitrary value $\varepsilon \neq 0$. We show in Proposition \ref{p:GM} below that the graded dimensions are independent of $\varepsilon$.

In Section \ref{s:combdeg}, we will prove that one may complete the underlying family $(T(\check{\Sigma})\times \mathbb{D}^*, w({\bf \Sigma})_{{\varphi}})$ to a family of Landau--Ginzburg models over $\mathbb{D}$. The central fibre of this family will have orbifold normal crossings singularities. In Section \ref{s:degnby} we show that there is a limiting irregular Hodge filtration which has de Rham realization in terms of the fibre over 0, whose graded pieces have constant rank. We show that a spectral sequence coming from the irreducible components of the fibre over 0 degenerates at the $E_2$ page, and the $E_2$ page is essentially the cellular complex computing ${\bf J}^\bullet$. We  have the following result.

\begin{theorem}[Corollary \ref{c:nearbyfibcoh}]\label{t:descent}
    Let $(T(\check{\Sigma}), w({\bf \Sigma})_{{\varphi}})$ be a family of Landau--Ginzburg models as above and suppose $\varepsilon \in \mathbb{D}^*$. For all $n \in \mathbb{Z}$, there is a (non-canonical) isomorphism of filtered $\mathbb{C}$-vector spaces
    \[
    (H^{n}(T({\check{\Sigma}}), w({\bf \Sigma})), F_{\mathrm{irr}}^\bullet) \cong \bigoplus_{k \in \mathbb{Z}_{\geq 0}}(H^{n-k}(\mathsf{T}({\bf {\check{\Sigma}}},w({\bf \Sigma})_\varphi)_0,{\bf J}^k), \mathsf{F}^\bullet).
    \] 
\end{theorem}

The proof of Theorem \ref{t:descent} involves a significant amount of preliminary setup that is not directly relevant to the main part of the article, so we include it in the final sections of the paper (Sections \ref{s:degnby} and \ref{s:nearbyfib}). Theorem \ref{t:descent} implies that 
\begin{align*}
\dim \gr_{F_{\mathrm{irr}}}^\lambda H^{n}(T(\check{ \Sigma}), w({\bf \Sigma})) &  = \sum_k \dim \gr_{F_{\mathrm{irr}}}^\lambda H^{n-k}(\mathsf{T}(\check{\bf \Sigma},w({\bf \Sigma})_{\varphi})_0,{\bf J}^k)
\\ & = \dim H^{\lambda,n-\lambda}(\mathsf{T}(\check{\bf\Sigma}, w({\bf \Sigma})_{\varphi})_0).
\end{align*}
Since the dimensions of irregular Hodge-graded pieces are preserved in the limit (Proposition \ref{p:GM}), we have 
\[
\dim \gr_{F_{\mathrm{irr}}}^\lambda H^n(T(\check{\Sigma}), w({\bf \Sigma})) = \dim H^{\lambda,n-\lambda}(\mathsf{T}(\check{\bf\Sigma}, w({\bf \Sigma})_{\varphi})_0).
\]
The above arguments naturally extend to the orbifold cohomology of toric Deligne--Mumford stacks, as the computation of the irregular orbifold Hodge numbers ultimately reduce to those of the underlying coarse moduli space (see also Remark \ref{r:relinertia}). Therefore Proposition \ref{p:hodgeorb} implies that for $\lambda,n \in \mathbb{Q}$ we have the following relation:
\begin{equation}\label{eq:orbhodgetrop}
\dim \gr_{F_\mathrm{irr}}^\lambda H^n_\mathrm{orb}(T({\bf \check{\Sigma}}), w({\bf \Sigma})) = \dim H^{\lambda,n-\lambda}_\mathrm{orb}(\mathsf{T}({\bf \check{\Sigma}}, w({\bf \Sigma})_{\varphi})_0).
\end{equation}
This relation \eqref{eq:orbhodgetrop} leads to the main result of this article when combined with Theorem \ref{t:tropidual}.
\begin{theorem}\label{t:clarke}
    Suppose $({\bf \Sigma},{\bf \check{\Sigma}})$ form a Clarke dual pair of stacky fans of dimension $d$ and coefficients of $w({\bf \Sigma})$ and $w(\check{\bf \Sigma})$ are chosen generically. Then for all $\lambda,\mu \in \mathbb{Q}$
    \[
    f_{\mathrm{orb}}^{\lambda,\mu}(T({\bf \Sigma}), w(\check{\bf \Sigma}))= f_{\mathrm{orb}}^{d-\lambda,\mu}(T({\bf \check{\Sigma}}), w({\bf \Sigma})).
    \]
\end{theorem}

\subsection{Weak Fano toric Deligne--Mumford stacks}\label{eg:wfano}
It is instructive to understand the situation when $\check{\Sigma} = \{0_N\}$. The convex condition forces $\Delta := \mathrm{Conv}({\bf \Sigma}[1] \cup 0_M)$ to be equal to $\Delta_{{\bf\Sigma}}$, however, it makes no  further requirements. The Clarke mirror pair of Landau--Ginzburg models arising from this situation is  then:
\[
T({ {\bf \Sigma}}) = T({{\bf \Sigma}}),\quad w(\check{\Sigma}) = 0,\quad T({\check{\Sigma}}) = (\mathbb{C}^{*})^d,\quad w({\bf \Sigma}) =1+ \sum_{m \in {\bf {\Sigma}}[1]}u_m \underline{x}^m.
\]
As usual, we choose $u_m$ generically. We have the following result.
\begin{theorem}\label{t:KKPweakFano}
For any weak Fano toric Deligne--Mumford stack $T({\bf \Sigma})$, there are identifications 
\[
 f_{\mathrm{orb}}^{\lambda,\mu}(T({\bf \Sigma}))= f_{\mathrm{orb}}^{d-\lambda,\mu}((\mathbb{C}^*)^d, w({\bf \Sigma}))
\]
for $\lambda,\mu\in \mathbb{Q}$. 
\end{theorem}

The reader may compare this to results of Douai \cite{douai2018global,douai2018note,douai2024hodge} in the case where $\Delta_{{\bf{\Sigma}}}$ is itself simplicial, and Batyrev \cite{batyrev1993variations} (also the first named author \cite{harder2016geometry}) in the case where the stacky structure on ${\Sigma}$ is the canonical stacky structure, in which case, $\Delta_{{\Sigma}}$ must be reflexive.

\subsection{Berglund--H\"ubsch--Krawitz duality}\label{s:bhk}
We follow Clarke's exposition of Berglund--H\"ubsch--Krawitz duality as described in \cite{clarke2016dual}. Suppose we have an $(d+1)\times (d+1)$ invertible matrix $B$ with non-negative integer entries. To $B$ we may attach a polynomial in $(d+1)$ variables,
\[
w_B = \sum_{j=0} \prod_i x_i^{p_{ij}}    
\]
To such a polynomial there is a finite subgroup $S_B $ which is the kernel of the map 
\[
\mathcal{F}(B) : (\mathbb{C}^*)^{d+1} \rightarrow (\mathbb{C}^*)^{d+1},\qquad y_j \mapsto \prod_{i}x_i^{p_{ij}}
\]
We then choose a subgroup $Q_B \subseteq S_B$. Then we obtain a Landau--Ginzburg model
\[
w_B : \mathbb{A}^{d+1}/Q_B \rightarrow \mathbb{A}^1.
\] 
We may factor the morphism $\mathcal{F}(B)$ as 
\[
\begin{tikzcd}
(\mathbb{C}^*)^{d+1} \ar[rr, "\mathcal{F}(B)"]\ar[rd,"\mathcal{F}(C^t)"] & & (\mathbb{C}^*)^{d+1} \\ 
& (\mathbb{C}^*)^{d+1}/Q_B \cong (\mathbb{C}^*)^{d+1} \ar[ru,"\mathcal{F}(D)"] & 
\end{tikzcd}
\]
Let $Q_B^t$  denote the kernel of the map $\mathcal{F}(C)$. One can show that $Q_B^t \subseteq S_{B^t}$ and thus we get a dual Landau--Ginzburg model 
\[
w_{B^t} : \mathbb{A}^{d+1}/Q_B^t \rightarrow \mathbb{A}^1.    
\]
Now we let $\Sigma = \mathrm{Cone}(\mathrm{cols}(C))$ and let $\check{\Sigma} = \mathrm{Cone}(\mathrm{cols}(D))$ where $\mathrm{cols}(C), \mathrm{cols}(D)$ indicate the columns of $C$ and $D$ respectively. 
\begin{theorem}[{\cite[Proposition 3.14, Theorem 3.16]{clarke2016dual}}]
	The pair of fans $(\Sigma, \check{\Sigma})$ form a Clarke mirror pair of Landau--Ginzburg models. The corresponding stacky Landau--Ginzburg models are 
	\[
	T(\Sigma) = \mathbb{A}^{d+1}/Q_B, \qquad w(\check{\Sigma}) = w_B ,\qquad     T(\check{\Sigma}) = \mathbb{A}^{d+1}/Q_B^t, \qquad w({\Sigma}) = w_{B^t}.
	\]
\end{theorem}

 Now, we have the following result, directly from Theorem \ref{t:clarke}.
\begin{corollary}\label{cor:Krawitz}
	Suppose both $C$ and $D$ are matrices.  Then for any $\lambda,\mu \in \mathbb{Q}$,
\[
    f^{\lambda,\mu}_\mathrm{orb}(\mathbb{A}^{d+1}/Q_B, w_B) = f^{d+1-\lambda,\mu}_\mathrm{orb}(\mathbb{A}^{d+1}/Q_{B^t}, w_{B^t}).
    \]
    
% Both sides will be identified with so-called orbifold Hodge numbers for Landau--Ginzburg models (see \ref{t:sha-hodge}). 
\end{corollary}
\begin{proof}
	The singularities of $\mathbb{A}^{(d+1)}/Q_B$ and $\mathbb{A}^{(d+1)}/Q_{B^t}$ are orbifold singularities, thus $\Sigma$ and $\check{\Sigma}$ satisfy the simplicial condition. The convex condition is a direct consequence of the fact that both $\Sigma$ and $\check{\Sigma}$ are cones over simplices. Similarly, since $\Sigma$ and $\check{\Sigma}$ are cones over simplices and thus any choices of linear functions $\varphi$ and $\check{\varphi}$ is strongly convex, both $\Sigma$ and $\check{\Sigma}$ are quasi-projective.
\end{proof}
In the following discussion, we assume that $w_B$ and $w_{B^t}$ have at worst isolated singularities at $0$. In this case, it is traditional in the literature for authors to study the vanishing cohomology at 0, simply denoted by $H^*(\phi_{w_B}\mathbb{C})$ and $H^*(\phi_{w_{B^t}}\mathbb{C})$, respectively. This vanishing cohomology carries a mixed Hodge structure, and a semi-simple operator, $T$, whose eigenspaces are also equipped with mixed Hodge structures. From this, one may define $\mathbb{Q}$-graded Hodge numbers similar to ours:
\begin{equation}\label{eq:lim0}
h^{\lambda,d-\lambda}_\mathrm{orb}(\phi_{w_B}\mathbb{C}) := \dim \gr^{\lfloor p \rfloor}_F H^{d}(\phi_{w_B}\mathbb{C})_{\exp(2\pi {\tt i}p)},
\end{equation}
and orbifold versions by following the usual recipe. When both matrices $C$ and $D$ have determinant $\pm 1$, the corresponding toric varieties have Gorenstein singularities and $\log T$ is nilpotent. Krawitz \cite{krawitz2010fjrw} and Borisov \cite{borisov2013berglund} prove that there is a state space isomorphism,
\begin{equation}\label{eq:krabo}
h^{\lambda,\mu}_\mathrm{orb}(\phi_{w_B}\mathbb{C}) = h^{d-\lambda,\mu}_\mathrm{orb}(\phi_{w_{B^t}}\mathbb{C}).
\end{equation}
In the case where $\det C, \det D \neq \pm 1$, Ebeling, Gusein-Zade, and Takahashi \cite{Ebeling2016Efunction} prove results similar to Corollary \ref{cor:Krawitz}. Namely, the authors of \cite{Ebeling2016Efunction} construct an $E$-function for $(\mathbb{A}^{d+1}/Q_B,w_B)$, which is a polynomial whose coefficients are signed sums of the numbers in~\eqref{eq:lim0}, and they show that these $E$-functions satisfy a relation similar to~\eqref{eq:krabo}. 

It is not obvious that the invariants studied in \cite{krawitz2010fjrw,borisov2013berglund,Ebeling2016Efunction} are the same as ours, so we will sketch a proof that they are equal. First, we mention that, by construction, $w_B$ has singularities only occurring in the fibre over $0$. Therefore, the vanishing cohomology may be viewed as a limit of the relative cohomology $H^*_\mathrm{orb}(\mathbb{A}^{d+1}/Q_B,w_B^{-1}(t))$ as $t$ approaches 0. The irregular Hodge numbers, on the other hand, can be obtained as the limit of of $H^*_\mathrm{orb}(\mathbb{A}^{d+1}/Q_B,w_B^{-1}(t))$ as $t$ approaches $\infty$, thanks to a result of Saito \cite[Theorem E.1]{esnault20171} (see also discussion following Proposition \ref{p:har}). Since the monodromy automorphisms around 0 and $\infty$ agree, we have that $h^{\lambda,\mu}_\mathrm{orb}(\phi_{w_B}\mathbb{C}) = f_\mathrm{orb}^{\lambda,\mu}(\mathbb{A}^{d+1}/Q_B,w_B)$. 

Notably, Corollary \ref{cor:Krawitz} does not require any nondegeneracy assumptions about the singularities so it is strictly more general than the work of Krawitz, Borisov, and of Ebeling, Gusein-Zade, and Takahashi.

\section{A conjecture of Katzarkov, Kontsevich, and Pantev}\label{s:smooth}

In this section, we will specialize the constructions in the previous section to the case of Gorenstein non-stacky Clarke dual pairs and show how our main result generalizes results of Batyrev and Borisov \cite{batyrev1996mirror} to {\em log} Calabi--Yau mirror pairs and proves a general form of a conjecture of Katzarkov, Kontsevich, and Pantev \cite{katzarkov2017bogomolov}. 

\begin{conjecture}[{Katzarkov--Kontsevich--Pantev \cite{katzarkov2017bogomolov}}]\label{c:kkp}
    If $X$ is a non-singular Fano variety of dimension $d$ and $(\check{U},w)$ is a homological mirror Landau--Ginzburg model to $X$, then for $p,q \in \mathbb{Z}$
\[
h^{p,q}(X)=f^{d-p,q}(\check{U},w).
\]
\end{conjecture}
In this section, we show that this relation also applies to varieties which have nef anticanonical divisors as well as varieties with orbifold singularities. The following is the orbifold generalization of Conjecture \ref{c:kkp}.

\begin{conjecture}[Orbifold KKP conjecture]\label{c:orbkkp}
    If $X$ is an orbifold which has nef anticanonical divisor and $(\check{U},w)$ is a homological mirror Landau--Ginzburg model to $X$, then for $p,q \in \mathbb{Z}$
    \[
        h_{\mathrm{orb}}^{p,q}(X)=f^{d-p,q}_{\mathrm{orb}}(\check{U},w).
   \]
\end{conjecture}

\begin{remark}
   One can also introduce a stacky generalization of the KKP conjecture. For instance, in Section \ref{eg:wfano}, we have already shown the case where $X$ is a $d$-dimensional weak Fano toric stack $T({\bf \Sigma})$ and $(\check{U},w)=((\mathbb{C}^*)^d, w({\bf \Sigma}))$ (Theorem \ref{t:KKPweakFano}). However, this does not imply the KKP conjecture for the underlying coarse moduli space $T(\Sigma)$, as the construction of a mirror Landau--Ginzburg model for $T(\Sigma)$ does not typically arise from a Clarke dual pair. We address this issue further in Section \ref{s:non-convex}.
\end{remark}

From the perspective of mirror symmetry, a choice of nef anticanonical divisor $D$ is also reflected in the mirror Landau--Ginzburg model. When $D$ is an orbifold normal crossings divisor, the complement $U := X \setminus D$ becomes log Calabi--Yau, and it is expected to be mirror to $\check{U}$. This leads to the Hodge number duality conjecture between $U$ and $\check{U}$.
\begin{conjecture}[Hodge number duality for log Calabi--Yau varieties]\label{c:logcal}
    Let $U$ and $\check{U}$ be homologically mirror dual $d$-dimensional log Calabi--Yau varieties at worst Gorenstein orbifold singularities. Then for $p,q \in \mathbb{Z}$, 
    \[
    f^{p,q}_{\mathrm{orb}}(U)=f^{d-p,q}_{\mathrm{orb}}(\check{U}).
    \]
\end{conjecture}

The remainder of this section is devoted to constructing prospective mirrors to a class of toric complete intersections with nef anticanonical divisors and to proving that Conjecture \ref{c:orbkkp} and Conjecture \ref{c:logcal} hold for the resulting pairs. We first gather some general facts about the cohomology of Landau--Ginzburg models, which play a crucial role in the proof. 

\subsection{Subvarieties and Landau--Ginzburg models}\label{s:cayley} 
We encounter cases where complete intersections and their mirrors do not fit into a Clarke mirror pair. To address this, we replace complete intersections of potentially high codimension with higher-dimensional (toric) Landau–Ginzburg models, which is often called \emph{Cayley's trick}. In this section, we establish several results that enable us to interpret the cohomology of complete intersections and complete intersection Landau–Ginzburg models in terms of Landau–Ginzburg models whose ambient spaces have higher dimensions.

 % Remark that if $s \in \Gamma(X,V)$ is a section of a vector bundle $V$ on a variety $X$ then there exists a regular morphism $g_s : \mathrm{Tot}(V^\vee) \rightarrow \mathbb{A}^1$ which extends to a rational function on $\mathbb{P}_X(\mathcal{O}_X\oplus V^\vee)$ which we also denote by $g_s$. Furthermore, 
 
 Suppose $(X,D,w)$ is a $d$-dimensional nondegenerate orbifold Landau--Ginzburg model and $Z$ is a hypersurface in $X$ so that $Z\cup D\cup w^{-1}(0)$ has orbifold normal crossings. Let $D_Z = Z\cap D$, then $(Z,D_Z,w|_Z)$ is also a nondegenerate orbifold Landau--Ginzburg model. Specifying such a hypersurface $Z$ is equivalent to taking a section $s$ of the bundle $\mathcal{O}_X(Z)$ with $s^{-1}(0)=Z$. This gives a regular morphism $g_s:\mathrm{Tot}(\mathcal{O}_X(-Z)) \to \mathbb{A}^1$ which may extend to a rational function on $\mathbb{P}_X(\mathcal{O}_X\oplus \mathcal{O}_X(-Z))$ which we also denote by $g_s$. Let $\pi:\mathrm{Tot}(\mathcal{O}_X(-Z)) \to X$ and $V_Z =\mathrm{Tot}(\mathcal{O}_X(-Z))\setminus \pi^{-1}D$.

 The following result is relatively well known and easy to prove, however, we were not able to find the precise statement that we need in the literature. See e.g., \cite{dimca2000dwork, Fresan2022, chiodo2018hybrid} for related results.

\begin{proposition}\label{p:tj2}
	Let notation be as above. Then, for each $p\in \mathbb{Z}$, there is an isomorphism of $F_\mathrm{irr}$-filtered vector spaces:
	\[
	H^{p-1}(Z\setminus D_Z, w|_Z)(-1) \cong H^{p+1}(V_Z, g_s + \pi^*w).  
	\]
\end{proposition}
\begin{proof}
	First, note that in a local orbifold chart  $U$ on $X$ with coordinates $(x_1,\dots, x_d)$ and orbifold group $G_U$, we have orbifold coordinates $(x_1,\dots, x_d,t)$ on $\mathrm{Tot}(\mathcal{O}_X(-Z))$ and orbifold group $G_U$ acting trivially on the $t$-coordinate. In these coordinates, $g_s = ts$ for a $G_U$-invariant function $s$ whose vanishing locus is $Z \cap U$. Then, $g_s + \pi^*w = ts + w$ in these coordinates. Then, on the complement of $Z$, we may define the map 
	\[
	(x_1,\dots, x_d) \mapsto (x_1,\dots, x_d, -w(x_1,\dots,x_d)/s(x_1,\dots, x_d)).
	\]
	Patching these maps together, we obtain an isomorphism between $V(g_s+\pi^*w ) \cap \pi^{-1}(X \setminus (D\cup Z))$ and $X\setminus Z$. Furthermore, $g_s$ vanishes on $\pi^{-1}(Z)$ so $(g_s +\pi^*w)|_{\pi^{-1}(Z)} = \pi^*(w|_Z)$. By Proposition \ref{p:resles} we have a long exact sequence:
	\[
	\dots  \longrightarrow H^p(V_Z, g_s + \pi^*w )  \longrightarrow H^p(V_Z \setminus \pi^{-1}(Z),g_s + \pi^*w) \longrightarrow H^{p-1}(\pi^{-1}(Z), \pi^*w)(-1) \longrightarrow \cdots 
	\]
	Since $V(g_s + \pi^*w) \cong X\setminus Z$ and $\pi^{-1}(X\setminus Z)$ is an $\mathbb{A}^1$ bundle over $X\setminus Z$, it follows from Proposition \ref{p:har} that the middle term in the sequence above vanishes. Furthermore, the fact that $\pi$ is an $\mathbb{A}^1$ bundle also implies that $H^{p-1}(\pi^{-1}(Z\setminus D_Z), \pi^*w) \cong H^{p-1}(Z\setminus D_Z,w)$. Thus we obtain an isomorphism of vector spaces
	\[
	H^{p-1}(Z\setminus D_Z,w)(-1) \cong H^{p+1}(V_Z, g_s + \pi^*w).
	\]
	This is also a morphism of filtered vector spaces, therefore it is a filtered isomorphism.
\end{proof}

Proposition \ref{p:tj2} can be extended to the orbifold cohomology as follows. By construction, $V$ is an orbifold as well, and has twisted sectors in bijection with those of $X$. In particular, $V_{(g)} = \pi^{-1}X_{(g)}$ for all $g \in I_X$. Similarly, if $Z\subseteq X$ is a quasi-smooth hypersurface, then $Z$ is also equipped with a natural orbifold structure as well, and $Z_{(g)} = X_{(g)} \cap Z$ is a union of disjoint twisted sectors whose age-grading is identical. Similarly, if $w^{-1}(1)$ is quasi-smooth and the intersection between $Z$ and $w^{-1}(1)$ is quasi-smooth, then the same thing is true for $w^{-1}(1) \cap Z$. 

\begin{remark}
	Properly speaking, distinct irreducible components of $Z_{(g)}$ represent different twisted sectors in $Z$, and $Z_{(g)}$ could be empty. In other words, the set of twisted sectors of $X$ need not be in bijection with those of $Z$. We will abuse notation, however, and treat the twisted sectors of $Z_{(g)}$ as if they are in bijection with those of $X_{(g)}$ because it simplifies our arguments.
\end{remark}

\begin{corollary}\label{c:tj2}
	Let assumptions be as in Proposition \ref{p:tj2}. Then for all $n \in \mathbb{Q}$,
	\[
	H_\mathrm{orb}^{n-2}(Z\setminus D_Z,w|_Z)(-1) \cong H_\mathrm{orb}^n(V_Z, g_s + \pi^*w).     
	\]
\end{corollary} 
\begin{proof}
	The discussion preceding the corollary implies that for each twisted sector $g\in I_X$, the triple  $(Z_{(g)},D_{Z,(g)},w|_{Z_{(g)}})$ is a nondegenerate Landau--Ginzburg model and therefore by Proposition \ref{p:tj2} we have 
	\[
	H^{k-2}(Z_{(g)}\setminus D_{Z,(g)},w|_{Z_{(g)}})(-1) \cong H^k(V_{Z,(g)}, (g_s + \pi^*w)|_{V_{Z,(g)}})
	\]
	for each twisted sector $g \in I_X$ and $k \in \mathbb{Z}$. In other words,
	\begin{align*}
	H_\mathrm{orb}^{n-2}(Z\setminus D_Z,w|_Z)(-1) &= \bigoplus_{g \in I_X} H^{n-2-2\iota{(g)}}(Z_{(g)} \setminus D_{Z,(g)}, w|_{Z_{(g)}})(-1-\iota{(g)})  \\ & \cong \bigoplus_{g \in I_X} H^{n-2\iota{(g)}}(V_{Z,(g)}, (g_s + \pi^*w)|_{V_{Z,(g)}})(-\iota{(g)})  \\ & = H^n_\mathrm{orb}(V_{Z}, (g_s + \pi^{*}w)|_{V_{Z,(g)}}).
	\end{align*}
\end{proof}

\subsection{Nef partitions and reflexive polytopes}\label{s:mirror-constr} 
In this section, we construct relevant toric mirror pairs. 

\begin{defn}
	We say that an integral polytope $\Delta$ containing the origin in $M$ is {\em reflexive} if its polar dual
	\[
	\check{\Delta} = \{ n \in N_\mathbb{R} \mid \langle n,m \rangle \geq -1\quad  \forall m \in \Delta\}
	\]
	is also an integral polytope. 
\end{defn}

 Let $\Delta$ be a reflexive polytope and let $\Sigma_\Delta$ denote its spanning fan and let $\Delta[0]$ denote the vertices of $\Delta$.
\begin{defn}\label{d:nefpar}
	A {\em nef partition} of $\Delta$ is a partition ${A}_1,\dots, {A}_{k+1}$ of the vertex set $\Delta[0]$ of $\Delta$ along with integral $\Sigma_\Delta$-linear convex functions $\varphi_j : M_\mathbb{R} \rightarrow \mathbb{R}$ so that $\varphi_j(A_i) = -\delta_{ij}$ for all $i,j$. We let $\Delta_i = \mathrm{Conv}(A_i,0)$.
\end{defn}
 To any nef partition of a reflexive polytope, there is a dual nef partition obtained by taking 
\[
\check{\Delta}_i = \{ n \in N_\mathbb{R} \mid \langle n,m\rangle \geq \varphi_i(m)\}
\]
It is known \cite{borisov1993towards} that the dual nef partition is itself a nef partition of $\check{\Delta} = \mathrm{Conv}(\check{\Delta}_1,\dots, \check{\Delta}_{k+1})$, and that $\check{\Delta}$ is reflexive itself.

A projective, simplicial, crepant refinement of $\Sigma_\Delta$ is determined by a strongly convex integral function $\varphi: \Sigma_{\Delta}[1]\rightarrow \mathbb{Z}$ such that the normal fan of the polytope
\[
\Delta_\varphi = \{ n \in N \mid \langle n,m\rangle \geq \varphi(m),\,\, \forall \,\, m \in \Delta\}
\]
is simplicial. We fix $\varphi$ and $\check{\varphi}$ with this property. The normal fan of $\Delta_\varphi$ is a projective orbifold resolution of the toric variety attached to $\Delta$. To simplify notation, we use the notation $T$ and $\check{T}$ to denote $T(\mathrm{nf}(\Delta_\varphi))$ and $T(\mathrm{nf}(\Delta_{\check{\varphi}}))$ respectively. 
\begin{remark}
	Projective crepant orbifold toric resolutions always exist in the context discussed above, and in fact one may demand that the singularities of $T$ are even terminal -- such a partial resolution is called a maximal projective crepant partial (MPCP) resolution. We do not require that $T$ and $\check{T}$ be MPCP resolutions.
\end{remark}

\subsection{Mirrors for nef anticanonical complete intersections}\label{s:mp}

Assuming that we have chosen $\Delta_1,\dots, \Delta_{k+1}$ as above, then we obtain nef toric divisors $E_1,\dots, E_{k+1}$ on $T$ and, dually, semi-ample toric divisors $\check{E}_1,\dots , \check{E}_{k+1}$ on $\check{T}$.  We obtain from this data a family of complete intersections in $T$ with semi-ample or trivial anticanonical divisors along with a family of Landau--Ginzburg models:

\begin{enumerate}
	\item Choose generic sections $s_i \in \mathcal{O}_{T}(E_i)$, we let $X_A$ denote $V(s_1,\dots, s_{k})$. The line bundle $\mathcal{O}_{T}(E_{k+1})$ remains nef on $X_A$ thus it has nef anticanonical bundle. Alternatively, one can let $D_A$ denote $X_A\cap E_{k+1}$. Then $(X_A,D_A)$ is a log Calabi--Yau pair.
	\item Choose $\check{s}_i \in \mathcal{O}_{\check{T}}(\check{E}_i)$, we let $U_{\check{A}} = X_{\check{A}} \setminus D_{\check{A}}$ as above. In homogeneous coordinates, we define a rational function
	\begin{equation}\label{eq:potential}
	w_{\check{A}} = \dfrac{\check{s}_{k+1}}{\prod_{\rho \in \check{A}_{k+1}} \check{x}_\rho } : \check{T} \dashrightarrow \mathbb{C}.
	\end{equation}
	Then $w_{\check{A}}$ has polar divisor $\check{E}_{k+1}$ hence it is a regular function when restricted to $\check{T}^\circ:=\check{T} \setminus \check{E}_{k+1}$. In particular, it is a regular function when restricted to $U_{\check{A}}.$
\end{enumerate}

\begin{proposition}
	The following statements hold.
	\begin{enumerate}
		\item The pair $(X_A,D_A)$ is an orbifold normal crossings log Calabi--Yau pair.
		\item The triple $(X_{\check{A}},D_{\check{A}},w_{\check{A}})$ is a tame, nondegenerate orbifold Landau--Ginzburg model. 
	\end{enumerate}
\end{proposition}
\begin{proof}
	Restricting our attention to the toric variety $T$, the divisors $E_i$ are semi-ample so that for generic sections $s_i \in \Gamma(T,\mathcal{O}_{T}({E}_i))$, the vanishing locus $X_A=V(s_i,\dots, s_{k})$ in $T_P$ is quasi-smooth by \cite[Theorem 2.6]{batyrev2011calabi}. Furthermore, our construction ensures that the intersection between $V(s_1,\dots, s_{k})$ and each toric stratum of $T$ is transversal. By adjunction, $D_A$ is anti-canonical in $X_A$. This proves item (1).
	
	Item (2) is similar. By the same facts as were used above, the polar locus of $w_{\check{A}}$ is $\check{D}_{k+1}$ which is orbifold normal crossings by construction and \cite[Theorem 2.6]{batyrev2011calabi}. The zero locus of $w_{\check{A}}$, $V(\check{s}_1,\dots, \check{s}_{k+1}) \subseteq V(\check{s}_1,\dots, \check{s}_{k})$ is a quasi-smooth hypersurface which meets $D_{\check{A}}$ transversally. Therefore $(X_{\check{A}}, D_{\check{A}}, w_{\check{A}})$ form a nondegenerate orbifold Landau--Ginzburg model. The fact that this is a tame Landau--Ginzburg model follows from the fact that the exponents in the denominator of~\eqref{eq:potential} are all 1.
\end{proof}

\noindent These pairs induce two expected mirror pairs: A semi--Fano mirror pair $(X_A, (U_{\check{A}}, w_{\check{A}}))$, and a log Calabi--Yau mirror pair $(U_A, U_{\check{A}})$. However, these mirror pairs are not Clarke mirror pairs in their current form, so Theorem \ref{t:clarke} cannot be applied. The idea is to use the Cayley trick, as explained in Section \ref{s:cayley}, to replace these pairs with certain Clarke mirror pairs of Landau--Ginzburg models without losing Hodge--theoretic information.

\subsection{Clarke mirror pairs from nef partitions} 
The next step is to attach pairs of toric Landau--Ginzburg models to $(X_A, (U_{\check{A}}, w_{\check{A}}))$ and  $(U_A, U_{\check{A}})$. We first deal with the semi--Fano mirror pair $(X_A, (U_{\check{A}}, w_{\check{A}}))$.
\begin{enumerate}
	\item Let $V_A = \mathrm{Tot}(\mathcal{O}_{T}(-E_1)\oplus \dots \oplus \mathcal{O}_{T}(-E_k))$ and let $g_1,\dots, g_k$ be the regular functions on $V_A$ corresponding to the regular sections $s_1,\dots, s_k$.
	\item Let $V^\circ_{\check{A}} = \mathrm{Tot}(\mathcal{O}_{\check{T}^\circ}(-\check{E}_1) \oplus \dots \oplus \mathcal{O}_{\check{T}^\circ}(-\check{E}_k))$ and let $\check{g}_1,\dots, \check{g}_k$ be regular functions attached to the regular sections $\check{s}_1,\dots, \check{s}_k$ on $U_{\check{A}}$. Let $\pi : V^\circ_{\check{A}} \rightarrow {\check{T}^\circ}$ be the canonical projection map.
\end{enumerate}
Then we have a pair of Landau--Ginzburg models
\begin{equation}\label{e:hidimLG}
    (V_A, g_1 + \dots + g_k),\qquad (V_{\check{A}}^\circ, \check{g}_1 + \dots + \check{g}_k + \pi^*w_{\check{A}})
\end{equation}
which corresponds to $(X_A, (U_{\check{A}}, w_{\check{A}}))$. 

To put the pair \eqref{e:hidimLG} into the framework of Clarke duality, we briefly recall the construction of the toric fans attached to a direct sum of line bundles on a toric variety associated to a fan $\Sigma$ (see \cite{Cox2011toricbook} for more details). Suppose we have a line bundle $\mathcal{L} = \mathcal{O}_{T(\Sigma)}(-\sum_{i=1}^n a_iE_i)$ where $E_1,\dots, E_n$ are toric boundary divisors of $T(\Sigma)$ and $a_i \in \mathbb{Z}$. Then $V_\mathcal{L} = \mathrm{Tot}(\mathcal{L})$ is a toric variety whose fan, denoted $\Sigma_\mathcal{L}$ is in $M_\mathbb{R} \times \mathbb{R}$. For each cone $c$ of $\Sigma$ there is a cone
\begin{equation*}
\widetilde{c} = \mathrm{Cone}(\{(\rho_i, a_i) \mid \rho_i \in c[1]\} \cup (0,1)) \in \Sigma_\mathcal{L}.
\end{equation*}
The fan $\Sigma_\mathcal{L}$ is the union of the cones $\widetilde{c}$ and their faces.

If we have a direct sum of line bundles $\mathcal{L}_j=\mathcal{O}_{T(\Sigma)}(-\sum_{i=1}^na_i^{(j)}E_i)$ for $j=1, \dots, k$, the total space has fan described as follows: For each cone $c$ of $\Sigma$ we construct the cone
\begin{equation}\label{eq:raygun}
\widetilde{c} = \mathrm{Cone}\{ 0\times e_j, (\rho_i,a_i^{(1)},\dots, a_i^{(k)}) \mid \forall \rho_i \in c[1] , \forall j = 1,\dots, k\}.
\end{equation}
The union of all such cones forms a fan $\Sigma_A$. This fan is simplicial if $\Sigma$ is simplicial and quasiprojective if $\Sigma$ is quasiprojective. The ray generators of this fan are simply
\begin{equation*}
\{0_M \times e_j\mid j=1,\dots, k \} \cup \{ (\rho_i,a_i^{(1)},\dots, a_i^{(k)}) \mid \rho_i \in \Sigma[1]\}.
\end{equation*}
In the particular case of interest to us, when $E_1,\dots, E_k$ are determined by a nef partition then for $\rho \in \Delta_{i,\mathbb{Z}}:= \Delta_i \cap M$ we have 
\[
(\rho, a_\rho^{(1)},\dots, a_{\rho}^{(k)}) = \begin{cases} \rho \times e_i & \text{ if } i \leq k \\
\rho\times 0 & \text{ if } i = k+1
\end{cases}
\]
To summarize, the fan $\Sigma_A$ has cones which are faces of 
\begin{equation*}
\widetilde{c} = \mathrm{Cone}( \{ (\rho,\varphi_1(\rho),\dots, \varphi_k(\rho)) \mid \rho \in c[1]\} \cup \{(0,e_1),\dots, (0,e_k)\})
\end{equation*}
The fan for $V_{\check{A}}^\circ$ can be determined similarly to be
\[
{\Sigma}^\circ_{\check{A}}[1] = \bigcup_{i=1}^{k} (\check{\Delta}_{i,\mathbb{Z}} \times \check{e}_i)   \subseteq N \times \mathbb{Z}^{k}
\]
where $\check{e}_i$ is dual to $e_i$ under any identification between $\mathbb{Z}^k$ and its dual lattice.
\begin{proposition}\label{p:npdc}
	Suppose $\Delta_1,\dots, \Delta_{k+1}$ is a nef partition of a polytope $\Delta$ and let $\check{\Delta}_1,\dots, \check{\Delta}_{k+1}$ be the corresponding dual nef partition. The fans $(\Sigma_A,\Sigma^\circ_{\check{A}})$ form a Clarke dual pair and $\Sigma_A$ and $\Sigma_{\check{A}}^\circ$ are simplicial, Gorenstein, quasiprojective, convex fans.
\end{proposition}
\begin{proof}
    That the fans are simplicial, Gorenstein, and quasiprojective follows by construction, since both $\Sigma_A$ and $\Sigma^\circ_{\check{A}}$ are the fans of total spaces of line bundles over simplicial, Gorenstein, and quasiprojective toric varieties. Also the regularity condition follows from the definition of the dual nef partition. Namely, for $m \in \Delta_i$ and $n \in \check{\Delta}_j$, we have $\langle (m,e_i), (n,\check{e}_j) \rangle \geq -\delta_{ij}+1 \geq 0$. Therefore, we only need to prove that $\Delta_{\Sigma_A}$ and $\Delta_{\Sigma^\circ_{\check{A}}}$ are convex. 

    We show that $\Delta_{\Sigma_A}$ is the intersection of $(k+1)$ convex regions, hence $\Delta_{\Sigma_A}$ is convex. Let 
    \[
    \begin{array}{llll}
    \phi_i: M_\mathbb{R} \times \mathbb{R}^k  & \longrightarrow \mathbb{R} & \qquad  \phi_{k+1} : M_\mathbb{R} \times \mathbb{R}^k & \longrightarrow \mathbb{R} \\
    (m,x_1,\dots,x_k) & \longmapsto  x_i - \varphi_i(m)&\qquad  (m,x_1,\dots, x_k)& \longmapsto  1 - \varphi_{k+1}(m)-\sum x_i
    \end{array}\]
    These functions are convex because $\varphi_i$ are convex functions, as is the linear function $(x_1,\dots, x_k) \mapsto x_i $ or $\sum x_i$. We claim that $\Delta_{\Sigma_A}$ is the subset of $M_\mathbb{R}\times \mathbb{R}^k$ where $\phi_1, \dots ,\phi_{k+1} \geq 0$. First, observe that if $\phi_1,\dots, \phi_{k+1} \geq 0$ then
    \begin{equation*}
    1 - \varphi_{k+1}(m) \geq \sum x_i \geq \sum_{i=1}^k\varphi_i(m).
    \end{equation*}
    This happens only if 
   $1 \geq \sum_{i=1}^{k+1}\varphi_i(m)$. Thus $m \in\Delta$ and there is a projection from $\Delta_{\Sigma_A}$ onto $\Delta$. We show that the preimage of $\Delta_c = \mathrm{Conv}\{c[1]\cup 0\}$ in $\Delta_{\Sigma_A}$ is $\Delta_{\widetilde{c}}$ as constructed in \eqref{eq:raygun}. 

    Suppose $(m,x_1,\dots, x_k)$ and write $m = \sum_{\rho \in c[1]} y_\rho \rho$ for $y_\rho \geq 0$ and $\sum y_\rho \leq 1$. By assumption $\varphi_i$ are all linear on $c$ and for each $\rho \in \Sigma[1]$, $\sum_{i=1}^{k+1}\varphi_i(\rho) = 1$. Therefore, $\sum_{i=1}^{k+1}\varphi_i(m) = \sum_{\rho \in c[1]} y_\rho$. Then $(m,x_1,\dots, x_k)$ is in $\widetilde{c}$ means we can write
    \begin{equation}\label{e:ineq3}
    (m,x_1,\dots, x_k) = (m, \varphi_1(m),\dots, \varphi_k(m)) + \sum_{i=1}^k z_i(0,e_i)
    \end{equation}
    for some $0 \leq z_i$ and $\sum y_\rho + \sum z_i \leq 1$. From \eqref{e:ineq3} $0\leq z_i$ is equivalent to $x_i - \varphi_i(m) \geq 0$ (i.e., $\phi_1,\dots, \phi_k \geq 0$) and $\sum y_\rho + \sum z_i \leq 1$ is equivalent to  
    \[
    \sum_{i=1}^{k+1}\varphi_i(m) + \sum_{i=1}^kx_i - \sum_{i=1}^k \varphi_i(m) \leq 1
    \]
    which is simply the statement that $\phi_{k+1} \geq 0$. In other words, the convex region bounded by $\phi_1,\dots, \phi_{k+1} \geq 0$ is the union of cells $\Delta_{\widetilde{c}}$  ranging over all cones $c$ of $\Sigma_\Delta$. Since all cones of $\Sigma_A$ are contained in a cone $\widetilde{c}$, the union of all $\Delta_{\widetilde{c}}$ is $\Delta_{\Sigma_A}$.

    We  omit the proof that $\Sigma_{\check{A}}^\circ$ is convex since it follows a similar line of reasoning.
\end{proof}
Therefore, the pair of Landau--Ginzburg models \eqref{e:hidimLG} is the Clarke mirror pair of Landau--Ginzburg model associated to $(\Sigma_A, \Sigma^\circ_{\check{A}})$. In other words, 
\[
\begin{aligned}
    T(\Sigma_A)=V_A, &\quad w(\Sigma^\circ_{\check{A}})=g_1+\cdots+g_k \\
    T(\Sigma^\circ_{\check{A}})=V^\circ_{\check{A}}, &\quad w(\Sigma_A)=\check{g}_1+\cdots \check{g}_k+\pi^*w_{\check{A}}.
\end{aligned}
\]
For the log Calabi--Yau pair $(U_A, U_{\check{A}})$, one can easily see that the corresponding pair of toric Landau--Ginzburg models can be obtained by the Clarke mirror pair of Landau--Ginzburg model associated to the Clarke dual pair $(\Sigma^\circ_A, \check{\Sigma}^\circ_{\check{A}})$:
\[
\begin{aligned}
    T(\Sigma^\circ_A)=V^\circ_A, &\quad w(\Sigma^\circ_{\check{A}})=g_1+\cdots+g_k \\
    T(\Sigma^\circ_{\check{A}})=V^\circ_{\check{A}}, &\quad w(\Sigma^\circ_A)=\check{g}_1+\cdots \check{g}_k.
\end{aligned}
\]

The following result follows directly from Propositions \ref{p:tj2} and Corollary \ref{c:tj2}. Similar results appear in work of  Chiodo and Nagel \cite{chiodo2018hybrid}.
\begin{proposition}[cf. \cite{chiodo2018hybrid}]\label{t:isohodge}
	There are isomorphisms of filtered vector spaces:
	\begin{align*}
	H^{*+2k}_\mathrm{orb}(V_A,g_1+\dots + g_k) &\cong H_\mathrm{orb}^{*}(X_A)(-k),\\ 
    H^{*+2k}_\mathrm{orb}(V^\circ_A,g_1+\dots + g_k) &\cong H_\mathrm{orb}^{*}(U_A)(-k),\\
    H_\mathrm{orb}^{*+2k}(V_{\check{A}},\check{g}_1 + \dots + \check{g}_k + \check{\pi}^*w_{\check{A}}) &\cong H_\mathrm{orb}^{*}(U_{\check{A}},w_{\check{A}})(-k).
	\end{align*}
\end{proposition}

 As a corollary, we may deduce that Conjecture \ref{c:orbkkp} and Conjecture \ref{c:logcal} holds for the toric complete intersection mirror pairs that were constructed in Section \ref{s:mp}.
\begin{theorem}\label{t:kkp}
	Let notation be as above. The following relations hold: for $p,q \in \mathbb{Z}$,
	\[
	f^{p,q}_\mathrm{orb}(X_A) = f_\mathrm{orb}^{d-p,q}(U_{\check{A}}, w_{\check{A}}),\qquad f^{p,q}_\mathrm{orb}(U_A)  = f^{d-p,q}_\mathrm{orb}(U_{\check{A}})
	\]
\end{theorem}

\begin{proof}
	We have shown (Proposition \ref{p:npdc}) that the pair $(\Sigma_A$  ${\Sigma}_{\check{A}})$ is a Clarke dual pair of fans. This also implies that $(\Sigma_A^\circ, {\Sigma}^\circ_{\check{A}})$. Now the conclusion follows from Theorem \ref{t:clarke} and Proposition \ref{t:isohodge}.
\end{proof}
\begin{corollary}[Batyrev--Borisov \cite{batyrev1996mirror}]\label{rem:BB}
    If $X_A$ and $X_{\check{A}}$ are a Batyrev--Borisov dual pair of Calabi--Yau complete intersections of dimension $d$ then 
    \[
    h^{p,q}_\mathrm{st}(X_A) = h^{d-p,q}_\mathrm{st}(X_{\check{A}}).
    \]
\end{corollary}
\begin{proof}
    Take $\Delta_{k+1} = \emptyset$. Then $X_{A}$ is a crepant orbifold resolution of singularities of the complete intersection $X'_A$ in associated to the nef partition $\Delta_1,\dots, \Delta_k$. Furthermore, $w_{\check{A}} = 0$ and $U_{\check{A}} = X_{\check{A}}$ is also an orbifold crepant resolution of singularities of the complete intersection $X'_{\check{A}}$ attached to $\check{\Delta}_1,\dots, \check{\Delta}_k$. By a result of Borisov--Mavlyutov \cite{borisov2003string}, 
	\[
	h_\mathrm{st}^{p,q} (X_A') = h_\mathrm{orb}^{p,q}(X_A),\qquad    h_\mathrm{st}^{p,q} (X_{\check{A}}') = h_\mathrm{orb}^{p,q}(X_{\check{A}}).\footnote{Here, $f_\mathrm{orb}^{p,q}(X)=h_\mathrm{orb}^{p,q}(X)$ for $X=X_A, X_{\check{A}}$ because they are both compact.}
	\]
	Thus we obtain the desired result.
\end{proof}

\begin{remark}\label{r:Rossi}
   In \cite{Rossi2023framed}, Rossi extends the duality of polytopes and introduces a new mirror construction for not necessarily nef complete intersections in toric varieties. For simplicity, consider the hypersurface case. Take an effective divisor $\mathcal{L} = \sum_{\rho \in \Sigma[1]} a_\rho E_\rho$ in $T(\Sigma)$, and let $\Delta_{\Sigma, \mathcal{L}}$ be $\mathrm{Conv}(\beta_\rho u_\rho)$ where $u_\rho$ is a primitive vector corresponding to $E_\rho$. For such data $(\Sigma, \beta)$, the author introduces the so-called \emph{calibrated $f$-process} to produce dual data $(\Sigma^\vee, \beta^\vee)$, where the convex hull of $\Sigma^\vee$ is $\mathrm{Conv}(\beta_\rho u_\rho)$, and $\beta^\vee$ is chosen so that applying the same process recovers $(\Sigma, \beta)$. By applying Cayley's trick as done in this section, we can fit such an $f$-dual pair into the Clarke dual framework, but the convexity condition does not generally hold. We discuss how to manage this type of phenomenon using stacky structures in Section \ref{s:non-convex}.
\end{remark}

\section{Hodge number duality for non-convex Clarke mirror pairs}\label{s:non-convex}

The convexity condition used in our main examples is rather stringent and we would like to relax it if possible. The goal of this section is to provide a Hodge-theoretic mirror to (1) smooth projective toric varieties and (2) {\em arbitrary} smooth hypersurfaces in smooth toric projective varieties. This construction makes use of the stacky structures that we have introduced, but at the cost of identifying Hodge numbers with a part of the irregular Hodge filtration on the mirror Landau--Ginzburg model. The mirror that we describe is a modification of the mirrors found in the literature \cite{gross2017towards,abouzaid2016lagrangian}. 
\begin{proposition}\label{p:hdst}
    Let $(\bf{\Sigma}, \bf{\check{\Sigma}})$ be a Clarke dual pair such that $\Sigma$ is unimodular and ${\bf \check{\Sigma}}=(\check{\Sigma},\underline{1})$. Assume that for any cone $c\in {\bf \Sigma}[i], i>1$, the integers in $\beta(c)$ are relatively prime. Then the following statements hold. 
    \begin{enumerate}
        \item The integer-graded orbifold Hodge numbers of $(T({\bf \Sigma}),w(\check{\Sigma}))$ are the Hodge numbers of 
    $(T({\Sigma}),w(\check{\Sigma}))$: For $p,q \in \mathbb{Z}$,
    \[
    f^{p,q}_\mathrm{orb}(T({\bf \Sigma}),w(\check{\Sigma}))=f^{p,q}(T({\Sigma}),w(\check{\Sigma})).
    \]
    \item Furthermore, Theorem \ref{t:clarke} implies that 
    \[ f^{p,q}(T(\Sigma), w(\check{\Sigma})) = f^{d-p,q}_\mathrm{orb}(T(\check{\Sigma}),w({\bf \Sigma}))_1  
    \]
    where the subscript $1$ denotes the Hodge numbers of the unipotent limit mixed Hodge structure on the cohomology of the pair.
    \end{enumerate} 
\end{proposition}

\begin{proof}
    As explained in Section \ref{s:cohdmst}, the twisted sectors are parametrized by elements in the union of $\mathrm{Box}^\circ({c})$ over all $c \in {\bf \Sigma}$. For $g=\sum_{\rho \in c[1]} a_\rho \rho \in \mathrm{Box}^\circ(c)$ with $a_\rho \in (0,1)$, the age is $\iota(g)=\sum a_\rho$. Due to our assumptions on ${\bf \Sigma}$, $\iota(g)$ becomes an integer only when $c=0$, which corresponds to the coarse moduli space $T(\Sigma)$.
    
    The second statement directly follows from the characterization of irregular Hodge numbers in terms of the limit mixed Hodge structure (see discussion following Proposition \ref{p:har}).
\end{proof}

\begin{remark}
   If one allows ${\bf \Sigma}$ to be Gorenstein simplicial, then the age grading on each twisted sector in $\mathfrak{I}_{T({\bf \Sigma})}$ needs to be further shifted to match with the age grading on the corresponding twisted sector of its coarse moduli space. As noted in Remark \ref{r:orbgr}, this is because the orbifold cohomology of $T(\Sigma)$ is not a direct summand of the orbifold cohomology of $T({\bf \Sigma})$, except after taking appropriate shift in age gradings. This indicates that we may not be able to characterize the orbifold cohomology of $T(\Sigma)$ in the orbifold cohomology of $T({\bf \Sigma})$ in terms of the integrality of the age gradings. 
\end{remark}

\subsection{KKP conjecture for smooth weak Fano toric varieties}\label{s:weakFano}

Let $\Sigma$ be a unimodular complete\footnote{This condition can be weakened slightly, but we impose it for the sake of simplicity.} fan and therefore that $T(\Sigma)$ is smooth projective toric variety. In general, the Clarke mirror pair associated to Clarke dual pair $(\Sigma, \check{\Sigma}=0)$ is not expected to be mirror \cite{abouzaid2016lagrangian, AKO2008weighted}. Our proof of Hodge number duality also cannot be applied as the convexity condition does not hold in general so that the tropicalizations of the corresponding Landau--Ginzburg models have auxiliary parts where the duality of tropical Jacobian sheaves breaks. We can resolve this issue by using stacky fans (Remark \ref{r:Rossi}).

Suppose there is a convex stacky fan ${\bf \Sigma}=(\Sigma, \beta)$ such that $\beta$ is relatively prime. Then we have a unimodular, quasiprojective, convex Clarke dual pair $({\bf \Sigma}, {\bf \check{\Sigma}}=0)$ such that the assumption in Proposition \ref{p:hdst} holds. The following theorem is a direct corollary of Proposition \ref{p:hdst}.  

\begin{theorem}\label{t:stmir1}
    Let $({\bf \Sigma}, {\bf \check{\Sigma}}=0)$ be as above. Then there are identification of Hodge numbers: For $p,q \in \mathbb{Z}$,
    \[ h^{p,q}(T(\Sigma)) = f^{d-p,q}_\mathrm{orb}(T(\check{\Sigma}),w({\bf \Sigma}))_1    
    \]
    where the subscript $1$ denotes the Hodge numbers of the unipotent limit mixed Hodge structure on the cohomology of the pair.
\end{theorem}
\begin{remark}
 We leave it as an open problem to prove the existence or non-existence of such a stacky fan in general. It is not even clear to us whether this can be carried out when $d=2$. On the other hand, there are many interesting cases (e.g., Hirzebruch surfaces, Example \ref{ex:Hirz}) where this does in fact work.
\end{remark}
\begin{example}\label{ex:Hirz}
Let us take $\Sigma$ to be the (non-stacky) fan of $\mathbb{F}_3$, the Hirzebruch surface of degree 3, which has ray generators $(1,0),(0,1), (-1,0)$, and $(3,-1)$. The polytope $\Delta_\Sigma$ is not convex, so we cannot apply Theorem \ref{t:clintro}. However, if we take the stacky fan with $\beta = (2,1,1,1)$ we obtain a toric stack $T({\bf \Sigma})$ whose corresponding inertia stack is $\mathfrak{I}_{T({\bf \Sigma})} = \mathbb{F}_3 \coprod \mathbb{P}^1$ with the age $\iota(\mathbb{P}_1) = 1/2$ (See Section \ref{s:orbcoh} for definition of orbifold cohomology). A Clarke mirror of $T({\bf \Sigma})$ is $(({\mathbb{C}^*})^2, w({\bf \Sigma})=1+x^2 + y + 1/x + x^3/y)$. Since ${\bf \Sigma}$ satisfies our convexity condition, the orbifold Hodge numbers of $T({\bf \Sigma})$ are
\[
f^{0,0} = f^{1/2,1/2} = f^{3/2,3/2} = f^{2,2} = 1,\qquad f^{1,1} = 2.
\]
One can see that the integer-graded parts are the same as the Hodge numbers of $T(\Sigma)=\mathbb{F}_3$. 

On the other hand, one can compute the vanishing cycles of the Laurent polynomial $x^2 + y + 1/x + x^3/y$. There are six vanishing cycles and the Seifert matrix of this Laurent polynomial is
\[
\chi = \left[ \begin{matrix} 1 & 2 & 1 &2 & 1& 2 \\ 0 & 1 & 2& 1& 1& 1\\ 0 & 0 & 1 & 0 & 2 & 1 \\ 0 & 0 & 0 & 1 & 0 & 1 \\ 0& 0 & 0 & 0& 1 & 2 \\ 0 & 0 & 0 & 0 & 0 & 1 \end{matrix}\right].
\]
From this, one computes the monodromy matrix on $H^2(\mathbb{C}^{\times 2}, w^{-1}(t))$ when $t$ varies in a small loop around $\infty$ by the formula $\chi^{-1}\chi^\mathrm{T}$. Its Jordan form is; 
\[
\left[ \begin{matrix}
    -1 & 1 & 0 & 0 & 0 & 0 \\ 0 & -1 & 0 & 0 & 0 & 0 \\ 0 & 0 & 1 & 1 & 0 & 0 \\
    0 & 0 & 0 & 1 & 1 & 0 \\ 
    0 & 0 & 0 & 0 & 1 & 0 \\ 
    0 & 0 & 0 & 0 & 0 & 1
\end{matrix} \right]
\]
This direct approach can be pushed further to deduce the irregular Hodge numbers from this matrix: The orbifold Hodge numbers of $((\mathbb{C}^*)^2, w({\bf \Sigma}))$ are 
\[
f^{0,2} = f^{1/2,3/2} = f^{1/2,3/2} = f^{2,0} = 1,\qquad f^{1,1} = 2.
\]
Here the integer--graded part are coming from the unipotent part of the matrix above. We will simply observe that the Jordan blocks of eigenvalue 1 and ranks 1 and 3 reflects the fact that $\mathbb{F}_3$ is a surface of Picard rank 2, while the Jordan block of eigenvalue $(-1)$ and rank 2 reflects the fact that $\mathbb{P}^1$ is a curve whose age is $1/2$.
\end{example}

\begin{remark}
   We expect that the construction of an auxiliary stacky mirror pair would provide another tool for addressing homological mirror symmetry for the smooth projective toric variety $T(\Sigma)$. On the B-side, there is a fully faithful embedding of the derived category of coherent sheaves on $T(\Sigma)$ into that of $T({\bf \Sigma})$. Mirror symmetry then predicts that there is a full subcategory of the Fukaya--Seidel category of $(T(\check{\Sigma}), w({\bf \Sigma}))$, and Theorem \ref{t:stmir1} suggests that the generating set of Lagrangians somehow implicitly reflects the shadow of unipotent monodromy around infinity. Additionally, it is interesting to compare this construction with that in the literature \cite{Hanlon2022fuctor}, where the mirror object is described as a pair consisting of a torus $(\mathbb{C}^*)^d$ and a certain FLTZ skeleton $f_\Sigma$.
\end{remark}

\subsection{Smooth hypersurfaces in smooth projective toric varieties}\label{s:gentype}
 
We discuss the Hodge number duality of smooth hypersurfaces in smooth projective toric varieties. The case of complete intersections will be analogous, so for the sake of simplicity, we will not present it here and will leave it to the reader.
 
Let $\Sigma$ be a unimodular fan and suppose $\varphi: \Sigma \rightarrow \mathbb{R}$ is a $\Sigma$-linear convex function. From this data, we obtain a line bundle 
\[
    \mathcal{L} = \mathcal{O}_{T(\Sigma)}\left(-\bigoplus_{\rho \in \Sigma[1]}\varphi(\rho)E_\rho\right)
    \]
The toric variety $\mathrm{Tot}(\mathcal{L})$ has fan $\Sigma_\mathcal{L}$ contained in $N_\mathbb{R}\times \mathbb{R}$ with ray generators $(\rho,\varphi(\rho))$ and $(0_N,1)$. On the other hand, we obtain a polytope $P_\mathcal{L} = \{m \in M_\mathbb{R} \mid m(\rho) \leq \varphi(\rho), \forall \rho \in \Sigma[1]\}$. Assume that $P_\mathcal{L}$ is an integral polytope and choose a projective triangulation of $P_\mathcal{L}$, and let $\check{\Sigma}_\mathcal{L}$ be the simplicial fan over $P_\mathcal{L}\times 1$ determined by this triangulation. The polytope $\Delta_{\check{\Sigma}_\mathcal{L}}$ is convex, essentially by assumption, however, if $\varphi(\rho) \geq 2$ for any $\rho$, the polytope $\Delta_{\Sigma_\mathcal{L}}$ need not be convex. That the fans are Clarke dual is a direct consequence of the definitions. 

\begin{defn}
    Define ${\bf \Sigma}_\mathcal{L}$ to be the stacky fan whose stacky structure is given by $\beta_{(\rho,\varphi(\rho))}=1$ and $\beta_{(0,1)}= \min_{\rho \in \Sigma[1]}\{\varphi(\rho)\}$.
\end{defn}
\begin{proposition}
    The fan ${\bf \Sigma}_\mathcal{L}$ satisfies the convexity condition.
\end{proposition}
\begin{proof}
    As in Proposition \ref{p:npdc}, we prove that $\Delta_{{\bf \Sigma}_\mathcal{L}}$ is bounded between two convex regions, hence that it is convex. Let $\varphi$ be a $\Sigma$-linear convex function whose associated nef line bundle is $\mathcal{L}$ and let $\varphi_\Sigma$ denote the convex function attached to the canonical line bundle $\omega_{T(\Sigma)}$. Then we take the region in $N_\mathbb{R}\times \mathbb{R}$ given by the inequalities 
    \[
    \{(n,x) \in N_\mathbb{R}\times \mathbb{R} \mid \varphi(n) \leq x \leq \beta_{(0,1)}(1+\varphi_\Sigma(n)) - \varphi(n)  \}.
    \]
    From this the argument is straightforward. 
\end{proof}
Applying Propositions \ref{p:tj2} and \ref{p:hdst}, we have the following theorem. 

\begin{theorem}\label{t:gkr}

    Suppose $Z$ is a generic hypersurface in a Gorenstein smooth projective toric variety $T(\Sigma)$ determined by the vanishing locus of a nef divisor $\mathcal{L}$. Then for $p,q \in \mathbb{Z}$
    \[
    h^{p-1,q-1}(Z) = f^{d-p,q}(T(\check{\Sigma}_\mathcal{L}),w({\bf \Sigma}_\mathcal{L}))_1 
    \]
    where $({\bf \Sigma_\mathcal{L}}, {\bf \check{\Sigma}_\mathcal{L}}=\check{\Sigma}_\mathcal{L})$ is the Clarke dual pair introduced as above. Here, the subscript $1$ also denotes the Hodge numbers of the unipotent limit mixed Hodge structure on the cohomology of the pair.
\end{theorem}
\begin{remark}
    We choose $\beta_{(0,1)} = \min_{\rho \in \Sigma[1]}\{\varphi(\rho)\}$ for simplicity but any value greater than and equal to $\min_{\rho \in \Sigma[1]}\{\varphi(\rho)\}$ serves the same function.
\end{remark}

    In this case, we can compute the orbifold Hodge numbers of $(T({\bf \Sigma}_\mathcal{L}),w(\check{\Sigma}_\mathcal{L}))$ more explicitly. In fact, the toric Deligne--Mumford stack determined by ${\bf \Sigma}_\mathcal{L}$ is a toric stack with orbifold structure along the zero section of the line bundle. 
By construction, $T({\bf \Sigma}_\mathcal{L})$ has inertia stack 
    \[
       \mathfrak{I}_{T({\bf \Sigma}_\mathcal{L})} \cong \mathfrak{I}_{T(\Sigma_\mathcal{L})} \coprod \underbrace{\mathfrak{I}_{T(\Sigma)} \coprod \dots \coprod \mathfrak{I}_{T(\Sigma)}}_{(\beta_{(0,1)}-1)\text{-times}}
    \]
    where the copies of $\mathfrak{I}_{T(\Sigma)}$ have age $k/\beta_{(0,1)}$ for $0 <k < \beta_{(0,1)}$ and are embedded into $T(\Sigma_\mathcal{L})$ as the zero-section when $T(\Sigma_\mathcal{L})$ is viewed as the total space of a line bundle. Since $ w(\check{\Sigma}_\mathcal{L})$ is determined by a section  $s$ of $\mathcal{L}$, it is identically 0 along the zero section of $T(\Sigma_\mathcal{L})$. Consequently, 
    \[
        H_\mathrm{orb}^*(T({\bf \Sigma}_\mathcal{L}), w(\check{\Sigma}_\mathcal{L})) \cong H^*(T({\Sigma}_\mathcal{L}), w(\check{\Sigma}_\mathcal{L}))\bigoplus_{0< k < \beta_{(0,1)}} H^*(T(\Sigma))(- k/\beta_{0,1})
        \]
    so that we have
    \[
    f^{p,q}_\mathrm{orb}   (T({\bf \Sigma}_\mathcal{L}),w(\check{\Sigma}_\mathcal{L})) = \begin{cases}
    h^{p-1,q-1}(Z) & \text{ if }p,q\in \mathbb{Z} \\
    h^{p-k/\beta_{(0,1)},p-k/\beta_{(0,1)}}(T(\Sigma)) & \text{ if } 0 < k< \beta_{(0,1)}, p,q \in \mathbb{Z} \\
    0 & \text{ otherwise }
    \end{cases} 
    \]

    The authors of \cite{gross2017towards} prove a similar result when $Z$ is a hypersurface of general type in $T(\Sigma)$. They take the non-stacky mirror pair $(T({ \Sigma}_\mathcal{L}),w(\check{\Sigma}_\mathcal{L}))$ and show that the Hodge numbers of $Z$ agree with the Hodge numbers of the sheaf of vanishing cycles of the zero-fibre. It would be very interesting to understand how these Landau--Ginzburg models differ. We expect the  following picture. The  Landau--Ginzburg model $(T({\bf \Sigma}_\mathcal{L}),w(\check{\Sigma}_{\mathcal{L}}))$ has a degenerate fibre over 0 whose singularities and vanishing cycles are identical to those of the degenerate fibre studied by Gross, Katzarkov, and Ruddat, and there should be $(\beta_{(0,1)}-1)$ clusters of $\chi_\mathrm{orb}(T(\Sigma_\mathcal{L}))$ nodal fibres over points away from the origin.

\section{Degenerations of Landau--Ginzburg models and their nearby cycles}\label{s:degnby}

The goal of this section and the next section is to prove Theorem \ref{t:descent}. In this section, we introduce and study degenerations of Landau--Ginzburg models, we will prove that if $(X,D,w,\pi)$ is a good degeneration of Landau--Ginzburg models (see Definition \ref{d:deglg} below), then $\mathbb{R}^p\pi_*(\Omega_{X/\mathbb{A}^1}^\bullet(\log D\cup X_0)(*P),d+dw)$ forms a trivial bundle over the disc centred at 0 and we will establish conditions under which the irregular Hodge numbers of a generic  fibre of $\pi$ can be recovered from a limiting irregular Hodge filtration at the fibre over 0.

\subsection{Quasi-stable degeneration of Landau--Ginzburg models}

Let us characterize the class of families of Landau--Ginzburg models to which our results apply.

\begin{defn}\label{d:deglg}

Let $X$ be an analytic orbifold and let $D$ be an orbifold normal crossings divisor. Let $(X,D,w:X \to \mathbb{A}^1)$ be a nondegenerate Landau--Ginzburg model. Suppose we have a proper morphism $\pi : X \rightarrow \mathbb{A}^1$ such that the irreducible components of $X_0 = \pi^{-1}(0)$ have multiplicity 1 with respect to $\pi$, multiplicity 0 with respect to $w$. If $X_0 \cup D$ is an orbifold normal crossings divisor, we call this data a {\em quasi-stable degeneration of Landau--Ginzburg models}. 
\end{defn}

\begin{example}
    Suppose $(X,D,w)$ is a nondegenerate orbifold Landau--Ginzburg model. For simplicity, assume that $P$ is irreducible. Construct the family of Landau--Ginzburg models $(X\times \mathbb{A}^1, D\times \mathbb{A}^1,tw)$ where $t$ is a parameter on $\mathbb{C}$ and $\pi$ is projection onto $\mathbb{A}^1$. This is not a quasistable degeneration of Landau--Ginzburg models however, after blowing up $X\times \mathbb{A}^1$ in the subvariety $0\times P$, we obtain a quasi-stable degeneration of Landau--Ginzburg models. The degenerate fibre is a normal crossings union of a copy of $X$ and $(P\times \mathbb{P}^1)$ meeting along a copy of $P$. The function $wt$ restricts to a constant function on $X$ and projection onto the second factor of $P\times \mathbb{P}^1$. For arbitrary $P$, one may construct a similar quasi-stable degeneration of Landau--Ginzburg models by blowing up irreducible components of $D$ in $X\times \mathbb{A}^1
    $ sequentially. Similar degenerations of Landau--Ginzburg models appear in the work of Sabbah in relation to the Brieskorn lattice. See \cite[\S 4]{sabbah2023singularities} and the references therein for a detailed overview.
\end{example}

If $(X,D,w,\pi)$ is a quasi-stable degeneration of Landau--Ginzburg models, then for each stratum $X_I$ of $X_0 = \pi^{-1}(0)$, the triple $(X_I,D_I = X_I \cap D, w_I = w|_{X_I})$ is an orbifold Landau--Ginzburg model. Similarly, for any $1 \gg |\varepsilon| > 0$ and $X_\varepsilon:=\pi^{-1}(\varepsilon)$, the triple $(X_\varepsilon, D_\varepsilon = D\cap X_\varepsilon, w_\varepsilon =  w|_{X_\varepsilon})$ is also a nondegenerate  Landau--Ginzburg model. The remainder of this section is devoted to constructing degenerations of Landau--Ginzburg models coming from combinatorial data, similar to Example \ref{ex:nondeg}.

\subsection{Relative twisted de Rham complex} \label{s:nbf}

In this section, we will explain how to extract irregular Hodge numbers from a limiting family of Landau--Ginzburg models. In the next section, this will be combined with the construction in Section \ref{s:combdeg} to give a tropical description of irregular Hodge numbers of a toric Landau--Ginzburg model. 

The approach that we take combines techniques of Steenbrink \cite{steenbrink1974mixed} and Yu \cite{yu2014irregular}. Suppose $(X,D,w,\pi)$ is a quasi-stable degeneration of Landau--Ginzburg models. Define the relative twisted de Rham complex of a quasi-stable degeneration of Landau--Ginzburg models to be:
\begin{align*}
    \cOmega^p_{X/\mathbb{A}^1}(\log D\cup X_0)(*P) &:= \dfrac{\cOmega_{X}^p(\log D\cup X_0)(*P)}{\cOmega_{X}^{p-1}(\log D\cup X_0)(*P)\wedge d\log \pi}\\
\cOmega^p_{X/\mathbb{A}^1}(\log D, \nabla) & := \dfrac{\cOmega_{X}^p(\log D\cup X_0, \nabla)}{\cOmega_{X}^{p-1}(D\cup X_0, \nabla)\wedge d\log \pi}.
\end{align*}
Here $X_0$ denotes $\pi^{-1}(0)$. A direct local computation shows that $\nabla$ induces differentials on the complex $\cOmega^\bullet_{X/\mathbb{A}^1}(\log D \cup X_0)(*P)$ and $\cOmega_{X/\mathbb{A}^1}^\bullet(\log D, \nabla)$ which we call the twisted relative de Rham complexes which we denote using the same notation. We also notice that the irregular Hodge filtrations extend to filtrations on the complexes above. Precisely,
\[
F_\mathrm{irr}^\lambda \cOmega^p_{X/\mathbb{A}^1}(\log D\cup X_0)(*P) = \dfrac{F^\lambda_\mathrm{irr} \Omega_{X}^p(\log D\cup X_0)(*P) + \cOmega_{X}^{p-1}(\log D\cup X_0)(*P)\wedge d\log \pi}{\cOmega_{X}^{p-1}(\log D\cup X_0)(*P)\wedge d\log \pi}.
\]
The following statement is a standard local computation.
\begin{proposition}
    Let $(X,D,w,\pi)$ be a quasi-stable degeneration of  Landau--Ginzburg models. Suppose  $t \neq 0$. Then 
    \[
    (F_\mathrm{irr}^\lambda \cOmega^p_{X/\mathbb{A}^1}(\log D\cup X_0)(*P)) \otimes \mathcal{O}_{X_t} \cong F_\mathrm{irr}^\lambda \cOmega^p_{X_t}(\log D_t)(*P_t).
    \]
\end{proposition}
For later use, we will check that $\Omega_{X/\mathbb{A}^1}^\bullet(\log D\cup X_0)(*P)$ may be replaced with a complex of coherent sheaves.

\begin{proposition}\label{p:yu}
    The following inclusion of filtered complexes is a filtered quasi-isomorphism:
    \[ 
    (\cOmega_{X/\mathbb{A}^1}^\bullet(\log D,\nabla), F_\mathrm{irr}^\bullet) \hookrightarrow (\cOmega_{X/\mathbb{A}^1}^\bullet (\log D\cup X_0)(*P), F_\mathrm{irr}^\bullet).
    \]
\end{proposition}
\begin{proof}
    One has a commutative diagram of filtered complexes 
    \[
    \begin{tikzcd}
    \cOmega_{X}^\bullet(\log D\cup X_0,\nabla) \ar[r] \ar[d,"\wedge d\log \pi"] & \cOmega_X^\bullet(\log D\cup X_0)(*P)  \ar[d, "\wedge d\log \pi"] \\
    \cOmega_{X}^{\bullet+1}(\log D\cup X_0,\nabla) \ar[r]  & \cOmega_X^{\bullet+1}(\log D\cup X_0)(*P) 
    \end{tikzcd}
    \]
    whose horizontal  arrows are filtered quasi-isomorphisms by Proposition \ref{p:Yusub}. It follows that cokernels of the vertical arrows are quasi-isomorphic. 
\end{proof}

 Suppose $D' = D \cup E$ is an orbifold normal crossings divisor and that $(X,D',w)$ and $(X,D,w)$ are both nondegenerate Landau--Ginzburg models. We let $E = \cup_{i=1}^l E_i$ where $E_i$ are irreducible. For each $I \subseteq [1,l]$, define
\begin{equation*}
E_I = \cap_{i\in I} E_i,\qquad P_I = E_I \cap P,\qquad D_I = E_I \cap D.
\end{equation*}
Let $w_I = w|_{D_I}$. Then the triple $(E_I,D_I,w_I,\pi|_{D_I})$ is a quasi-stable degeneration of Landau--Ginzburg models by construction. The induced residue maps  $\mathrm{Res}_{E_{I\cup \{j\}}}^{E_I}$ can be extended to filtered morphisms 
\[
\mathrm{Res}_{E_{I\cup \{j\}}}^{E_I}: \Omega^\bullet_{E_I/\mathbb{A}^1}(\log D_I\cup E_{I,0})(*P_{I}) \longrightarrow \Omega^{\bullet-1}_{E_{I\cup \{j\}}/\mathbb{A}^1}(\log D_{I\cup j}\cup E_{I\cup \{j\},0})(*P_{I\cup \{j\}}).
\]
Here the dual intersection complex of $E$ forms a simplicial complex. If one chooses a signed incidence rule for this simplicial complexes we may let 
\begin{align*}
\mathrm{Res}^{E_I} :  \cOmega^{\bullet}_{E_I/\mathbb{A}^1}(\log D_I\cup E_{I,0})(*P_I) &\longrightarrow \oplus_{j \notin I} \cOmega^{\bullet-1}_{E_{I\cup \{j\}}/\mathbb{A}^1}(\log D_{I\cup \{j\}}\cup E_{I \cup \{j\}, 0})(*P_{I\cup \{j\}}) \\
\omega & \longmapsto \oplus_{j \notin  I}[I\cup \{j\}:I] \cdot \mathrm{Res}_{E_I}^{E_{I\cup \{j\}}}(\omega). 
\end{align*}
and 
\begin{equation*}
    \mathrm{Res}_m : \bigoplus_{|I| = m} \mathrm{Res}^{E_I} : \bigoplus_{|I| = m} \cOmega^{\bullet}_{E_I/\mathbb{A}^1}(\log D_I \cup E_{I,0})(*P_I) \longrightarrow \bigoplus_{|I| = m+1} \cOmega^{\bullet-1}_{E_I}(\log D_I \cup E_{I,0})(*P_I)
\end{equation*}

\begin{proposition}\label{p:residue-resolution2}
    There is a filtered resolution of complexes 
    \begin{align*}
    0 \longrightarrow \cOmega^\bullet_{X/\mathbb{A}^1}(\log D'\cup X_0)(*P) &\xrightarrow{\mathrm{Res}_{0}} \bigoplus_{|I| =1}\cOmega^{\bullet-1}_{E_I/\mathbb{A}^1}(\log  D_I \cup E_{I_0})(*P_I)  \\ &  \xrightarrow{\mathrm{Res}_{1}} \bigoplus_{|I| =2}\cOmega^{\bullet-2}_{E_I/\mathbb{A}^1}(\log  D_I \cup E_{I,0})(*P_I) \xrightarrow{\mathrm{Res}_{2}} \dots 
    \end{align*}
\end{proposition}

\begin{proof}
    It is well known that such a resolution exists for the underlying relative logarithmic sheaves in the normal crossings case. The orbifold normal crossings case follows by the same argument and taking local $G$-invariants. Tensoring the  $p$th row with the projective $\mathcal{O}_X$-module $\mathcal{O}_X(\lfloor (p-\lambda)P\rfloor)$ preserves exactness in rows, therefore we obtain the desired result.
\end{proof}

\subsection{The Gauss--Manin connection for smooth families}

In this subsection, we prove that the induced irregular Hodge filtration on $\mathbb{H}^p(X,\Omega_{X/\mathbb{A}^1}^\bullet(\log D\cup X_0)(*P)\otimes \mathcal{O}_{X_0})$ has graded dimensions equal to the graded dimensions of $\mathbb{H}^p(X,\Omega_{X/\mathbb{A}^1}^\bullet(\log D\cup X_0)(*P)\otimes \mathcal{O}_{X_t})$ for $t\neq 0$. Following, e.g., Steenbrink \cite{steenbrink1974mixed}, the usual approach to this is to construct a logarithmic Gauss--Manin connection on the filtered complex $\mathbb{R}^p\pi_*\Omega_{X/\mathbb{A}^1}^\bullet(\log D\cup X_0)(*P)$. Then general results ensure that the desired result holds. However, this requires that $\mathbb{R}^p\pi_*\Omega_{X/\mathbb{A}^1}^\bullet(\log D\cup X_0)(*P)$ is a coherent sheaf. The easiest way to ensure this is to replace the (non-coherent) complex $\Omega_{X/\mathbb{A}^1}^\bullet(\log D\cup X_0)(*P)$ with the filtered quasi-isomorphic coherent complex $\Omega_{X/\mathbb{A}^1}^\bullet(\log D, \nabla)$ constructed above.

By construction, there is a short exact sequence of complexes,
\[
0 \longrightarrow \cOmega^{\bullet-1}_{X/\mathbb{A}^1}(\log D,  \nabla)\otimes \pi^{-1}\Omega_{\mathbb{A}^1}^1(\log 0)  \longrightarrow \cOmega^{\bullet}_{X}(\log D, \nabla)  \longrightarrow \cOmega^{\bullet}_{X/\mathbb{A}^1}(\log D, \nabla) \longrightarrow 0 
\]
to which we may apply the functor $\mathbb{R}^p\pi_*$ to obtain a connecting homomorphism 
\[
\nabla_{\mathrm{GM}} : \mathbb{R}^p\pi_*    \cOmega^{\bullet}_{X/\mathbb{A}^1}(\log D,\nabla) \longrightarrow (\mathbb{R}^p\pi_* \cOmega^{\bullet}_{X/\mathbb{A}^1}(\log D, \nabla)) \otimes \Omega_{\mathbb{A}^1}^1(\log 0).
\]

\begin{proposition}\label{p:GM}
    The map $\nabla_{\mathrm{GM}}$ is a logarithmic connection on $\mathbb{R}^p\pi_*(\Omega_{X/\mathbb{A}^1}^\bullet(\log D\cup X_0)(*P))$ which is  holomorphic at all points except possibly 0.  Therefore $\mathbb{R}^p\pi_*(\Omega_{X/\mathbb{A}^1}^\bullet(\log D\cup X_0)(*P))$ has constant rank except possibly at 0. 
\end{proposition}
\begin{proof}
    We check directly that $\nabla_\mathrm{GM}$ is a connection. The method is well known (e.g. \cite{katz-oda}). For instance, replace the complexes above with their Godement resolutions (denoted by $\mathcal{C}_\mathrm{Gd}$), then the calculation becomes one on individual stalks. If $\alpha \in \mathcal{C}_\mathrm{Gd}\cOmega^p_{X/\mathbb{A}^1}(\log D,\nabla)$ is $\nabla$-closed, 
    \[
        \nabla_\mathrm{GM}(\alpha\bmod d\log \pi) = d\alpha + dw\wedge \alpha = \beta \wedge d\log \pi.
        \]
If $g \in \mathcal{O}_{\mathbb{A}^1}$ a calculation shows that $\nabla_\mathrm{GM}(\pi^*g\alpha) = d\pi^*g\wedge \alpha + \pi^*g\nabla_\mathrm{GM}(\alpha)$ hence it is indeed a connection. Flatness comes directly from the fact that $\nabla^2=0$. .
    \end{proof}
The following result is, of course, true more generally but we state it in its simplest form.
\begin{corollary}\label{c:trivfib}
    Suppose $(X,D,w,\pi)$ is a 1-dimensional analytic family of nondegenerate Landau--Ginzburg models over a simply connected base. Let $(X_t, D_t,w_t)$ be the fibre of $\pi$ over $t$. Then $H^*(X\setminus D,w) \cong H^*(X_t\setminus D_t,w_t)$ for any point $t$.
\end{corollary}
\begin{proof}
    Apply the Leray spectral sequence along with Proposition \ref{p:GM}.
\end{proof}

\subsection{Constancy over 0}
We would like to show that the rank of $\mathbb{R}\pi_*\Omega_{X/\mathbb{A}^1}^\bullet(\log D\cup X_0)(*P)$ is constant over 0 as well. A coherent sheaf with a logarithmic connection need not be a vector bundle. This section addresses this question by showing that the complex $\Omega_{X/\mathbb{A}^1}^\bullet(\log D\cup X_0)(*P)\otimes \mathcal{O}_{X_0}$ has hypercohomology whose  rank agrees with that of $\Omega_{X/\mathbb{A}^1}^\bullet(\log D\cup X_0)(*P)\otimes \mathcal{O}_{X_t}$ for $t \neq 0$. The result will then follow directly from Grauert's Theorem. Our approach is an adaptation of Steenbrink's approach in \cite{steenbrink1974mixed} (see also \cite{peters2008mixed}). Without loss of generality, we may restrict $\pi$ over a small analytic disc $\mathbb{D}$ at $0$.

We have the following diagram where $\mathfrak{h}$ denotes the complex upper half plane:
\[
\begin{tikzcd}
    X_\infty \ar[r,"k"] \ar[d] & X \ar[d,"\pi"] & X_0 \ar[l,"i",swap] \ar[d] \\
    \mathfrak{h} \ar[r,"\exp(2\pi{\tt i} -)"] & \mathbb{D} & 0 \ar[l]
\end{tikzcd}
\]
Let $X^* = \pi^{-1}(\mathbb{D}\setminus 0)$ and let $D^* = X^* \cap D, P^* = P \cap X^*$ and let $D_\infty = k^{-1}D^*, P_\infty = k^{-1}P^*$. As usual, the nearby cycles complex will be defined by 
\begin{equation*}
\psi_\pi (\Omega^\bullet_{X^*}(\log  D^*)(*P^*),d+dw) := i^{*}\mathbb{R}k_*k^{*}(\Omega_{X^*}^\bullet(\log  D^*)(*P^*), d+ dw).
\end{equation*}
Since $\Omega_{X^*}^\bullet(\log D^*)(*P^*)$ is quasi-isomorphic to the coherent complex $\Omega_{X^*}^\bullet(\log D^*,w)$, the fact that $k :X_\infty\rightarrow X^*$ is Stein implies that $\mathbb{R}k_*\Omega_{X_\infty}^\bullet(\log D_\infty)(*P_\infty)$ is isomorphic to $\Omega_{X_\infty}^\bullet(\log D_\infty)(*P_\infty)$. It is a direct check to see that $k^*\Omega_X^p(\log D^*)(*P^*)$ is isomorphic to $\Omega_{X_\infty}^\bullet(\log D_\infty)(*P_\infty)$, therefore
\begin{equation*}
    \psi_\pi (\Omega^\bullet_{X^*}(\log  D^*)(*P^*),d+dw) \cong i^{*}k_*(\Omega_{X^*}^\bullet(\log  D^*)(*P^*), d+ dw).
\end{equation*}
Our goal will be to show that $\psi_\pi (\Omega_{X^*}^\bullet(\log D\cup X_0)(*P)\otimes \mathcal{O}_E)$ is quasi-isomorphic to the complex $(\Lambda_{X,D}^\bullet, d+ dw)$ defined in the previous section. The analogous fact for the logarithmic de Rham complex was proved by Steenbrink \cite{steenbrink1974mixed}, following ideas of Deligne \cite{katz1973sga}. In Steenbrink's proof, an explicit description of the stalks of the log de Rham complex is used in a crucial way. We are able to circumvent explicit computation, but we need the following fundamental fact. 

\begin{lemma}\label{l:fdstalk}
    Let $(X,D,w)$ be a nondegenerate analytic Landau--Ginzburg model and let $0$ be a point in $X$. Then $\mathcal{H}^p(\Omega_X^\bullet(\log D)(*P))_0$ is a finite dimensional $\mathbb{C}$-vector space.
\end{lemma}
\begin{proof}
    We use \cite[Lemma C.2]{esnault20171} which says that there is a triple of semi-analytic morphisms $\alpha,\beta,$ and $\varpi$ so that $\Omega_X^\bullet(\log D)(*P)$ is quasi-isomorphic to $\mathbb{R}\varpi_*\beta_! \alpha_*\underline{\mathbb{C}}_{X\setminus D}$. Therefore, it is enough to prove that the statement of the lemma holds for $\mathbb{R}\varpi_*\beta_! \alpha_*\underline{\mathbb{C}}_{X\setminus D}$, but this is essentially a topological computation whose details we omit.
\end{proof}

\begin{lemma}\label{l:redstalk}
    The natural morphism on stalks 
    \[
    \mathscr{H}^i(\Omega_{X}^\bullet(\log D\cup X_0)(*P)\otimes \mathcal{O}_{X_0})_0 \longrightarrow \mathscr{H}^i(\Omega_{X/\mathbb{A}^1}^\bullet(\log D \cup X_0)(*P)\otimes \mathcal{O}_{X_0})_0
    \]
    is surjective with kernel $\mathscr{H}^i(\Omega_{X/\mathbb{A}^1}^\bullet(\log D \cup {X_0})(*P) \otimes \mathcal{O}_{X_0})_0\wedge d\log t$.
\end{lemma}
\begin{proof}
This is a calculation on stalks. The orbifold case follows from the normal crossings case after taking orbifold coordinates and local invariants, and using the fact that $V \mapsto V^G$ is an exact functor when $G$ is a finite group and $V$ is a vector space over a field of characteristic 0. Therefore, we will only address the normal crossings case. 

We choose local coordinates $(\underline{x},\underline{y},\underline{z})$ for which $\pi = y_1\dots y_a$, $D= x_1\dots x_bz_1\dots z_c$ and $w = 1/(x_1^{e_1}\dots x_b^{e_b})$. We define 
    \[
        \begin{array}{ll}
        S_i := \bigwedge^i\left\{\dfrac{dx_1}{x_1},\dots, \dfrac{dx_b}{x_b}, \dfrac{dz_1}{z_1},\dots, \dfrac{dz_c}{z_c}\right\} &\\
        T_j := \bigwedge^j \left\{\dfrac{dy_1}{y_1},\dots, \dfrac{dy_a}{y_a}\right\}  & \tilde{T}_j = T_j \bmod d\log \pi \\
        {\Omega}_{(\mathrm{abs})}^{i,j} := \mathbb{C}\{\underline{x},\underline{y},\underline{z}\}[z^{-1}]\otimes_\mathbb{C} S_i \otimes {T}_j  & {\Omega}_{(\mathrm{rel})}^{i,j} := \mathbb{C}\{\underline{x},\underline{y},\underline{z}\}[x^{-1}]\otimes_\mathbb{C} S_i \otimes \tilde{T}_j 
        \end{array}
    \]
We have differentials $d_{x,z} : \Omega_*^{i,j} \rightarrow \Omega_*^{i+1,j}$ and $d_y : \Omega_*^{i,j}\rightarrow \Omega_*^{i,j+1}$. Here and later $* \in \{(\mathrm{abs}),(\mathrm{rel})\}$. Note that $dw \wedge (-) :\Omega_*^{i,j}\rightarrow \Omega_*^{i,j+1}$ as well. Then the stalk of $\Omega_{X}^\bullet(\log D\cup X_0)(*P) \otimes \mathcal{O}_{X_0}$ is isomorphic to the total complex of the double complex 
\begin{equation}\label{e:stalkres}
\begin{tikzcd}
\vdots & \vdots & \vdots & \\
\Omega_*^{0,2} \ar[u,"d_y"] \ar[r,"d_{x,z}+dw"] & \Omega_*^{1,2} \ar[u,"d_y"] \ar[r,"d_{x,z}+dw"]& \Omega_*^{2,2} \ar[u,"d_y"] \ar[r,"d_{x,z}+dw"] & \dots \\
\Omega_*^{0,1} \ar[u,"d_y"] \ar[r,"d_{x,z}+dw"] & \Omega_*^{1,1} \ar[u,"d_y"] \ar[r,"d_{x,z}+dw"]& \Omega_*^{2,1} \ar[u,"d_y"] \ar[r,"d_{x,z}+dw"] & \dots \\
\Omega_*^{0,0} \ar[u,"d_y"] \ar[r,"d_{x,z}+dw"] & \Omega_*^{1,0} \ar[u,"d_y"] \ar[r,"d_{x,z}+dw"]& \Omega_*^{2,0} \ar[u,"d_y"] \ar[r,"d_{x,z}+dw"] & \dots 
\end{tikzcd}
\end{equation}
Now suppose that $* = (\mathrm{rel})$. Fixing any monomial form $A_{\alpha,\beta,\omega} = \underline{x}^\alpha\underline{z}^\beta \otimes \omega$ where $\omega \in S_i$, we may apply $d_y$ to $A_{\alpha,\beta,\omega} \otimes \tilde{T}_\bullet$. We obtain a complex $\tilde{C}_{\alpha,\beta,\omega}^\bullet$:
\begin{align*}
A_{\alpha,\beta,\omega} \otimes  \mathbb{C}\{\underline{y}\}\otimes \tilde{T}_0 &\xrightarrow{d_y}A_{\alpha,\beta,\omega} \otimes \mathbb{C}\{\underline{y}\}\otimes \tilde{T}_1\xrightarrow{d_y} A_{\alpha,\beta,\omega} \otimes \mathbb{C}\{\underline{y}\}\otimes \tilde{T}_2 \xrightarrow{d_y} \dots \\
& \cong A_{\alpha,\beta,\omega} \otimes ( \mathbb{C}\{\underline{y}\}\otimes \tilde{T}_0 \xrightarrow{d_y} \mathbb{C}\{\underline{y}\}\otimes \tilde{T}_1\xrightarrow{d_y}  \mathbb{C}\{\underline{y}\}\otimes \tilde{T}_2 \xrightarrow{d_y} \dots )
\end{align*}
The argument in the proof of \cite[Theorem 11.13]{peters2008mixed} applies directly to show that the $i$th cohomology of $\tilde{C}_{\alpha,\beta,\omega}^\bullet$ is isomorphic to, and in fact, spanned by, $\mathbb{C}\{\pi\} \otimes \tilde{T}_i$. Taking the second spectral sequence of the double complex \eqref{e:stalkres} we get 
\[
    \prescript{}{(\mathrm{rel})}E_1^{p,q} = \mathbb{C}\{\underline{x},\underline{z},\pi\}[x^{-1}] \otimes_\mathbb{C} S_p \otimes_\mathbb{C}\tilde{T}_q
\]
and $d_1 = d_{x,z} + dw$. Consequently,
\[
\prescript{}{(\mathrm{rel})}E_2^{p,q} = \mathcal{H}^p(\mathbb{C}\{\underline{x},\underline{z}\}[x^{-1}]\otimes_\mathbb{C} S_\bullet,d_{x,z}+dw)\otimes_\mathbb{C} \tilde{T}_q \otimes_\mathbb{C} \mathbb{C}\{\pi\}.
\]
We now show that $d_2=0$. To compute $d_2$, choose a representative of a class $[\xi]$ in $E_2^{p,q}$ in $E_{p,q}^0 = \Omega^{p,q}$. We first assume $\xi$ is $d_y$-closed. From our argument above (or more precisely, \cite[Theorem 11.13]{peters2008mixed}), $\xi$ is $d_y$-closed if and only if the only monomials $\underline{y}^\tau$ appearing are of the form $\pi^d$ for some integer $d$, and no such form is $d_y$-exact. 

Let us choose such a representative $\xi$ for $[\xi]$. In order for $\xi \bmod \mathrm{im}(d_{x,z}+  dw)$ to represent an element of $E_2^{p,q}$, we must have that $(d_{x,z} + dw)(\xi) \equiv 0 \bmod \mathrm{im}(d_y)$. However, since we have assumed that the only $\underline{y}$ monomials appearing in $\xi$ are of the form $\pi^d$, the same is true for $(d_{x,z} + dw)\xi$. So, in order for $(d_{x,z}+ dw)\xi$ to be $d_y$-exact, it must be 0. As a consequence, $d_2=0$ and hence the second spectral sequence degenerates at the $E_2$-term, by the usual definition of $d_2$ for the second spectral sequence of a double complex. 

If we suppose $* = (\mathrm{abs})$ the argument as in the previous paragraphs can be adapted slightly to compute the cohomology of ${\bm s}\Omega^{\bullet,\bullet}_{(\mathrm{abs})}$. To do this, we construct a complex $C_{\alpha,\beta,\omega}^\bullet$ of the form
\[
    A_{\alpha,\beta,\omega} \otimes ( \mathbb{C}\{\underline{y}\}\otimes T_0 \xrightarrow{d_y} \mathbb{C}\{\underline{y}\}\otimes T_1\xrightarrow{d_y}  \mathbb{C}\{\underline{y}\}\otimes T_2 \xrightarrow{d_y} \dots )
\]
By the Poincar\'e lemma, $H^i(C^\bullet_{\alpha,\beta,\omega}) \cong T_i$. The argument above then goes through exactly as above to show that spectral sequence induced by \eqref{e:stalkres} degenerates at the $E_2$ term and 
\[
\prescript{}{(\mathrm{abs})}E_2^{p,q} = \mathcal{H}^p(\mathbb{C}\{\underline{x},\underline{z}\}[x^{-1}]\otimes_\mathbb{C} S_\bullet,d_{x,z}+dw)\otimes_\mathbb{C} {T}_q.
\]
To prove the final statement, we look at the long exact sequence on stalks in cohomology, coming from the short exact sequence of stalks of complexes,
\[
0\longrightarrow \Omega^\bullet_{X/\mathbb{A}^1}(\log D\cup X_0)(*P)[-1]_0 \xrightarrow {\wedge d\log \pi}\Omega^\bullet_X(\log D\cup X_0)(*P)_0 \longrightarrow \Omega^\bullet_{X/\mathbb{A}^1}(\log D\cup X_0)(*P)_0 \longrightarrow 0.
\]
By a simple dimension count using the calculations above, along with Lemma \ref{l:fdstalk}, we see that the connecting homomorphism in the induced long exact sequence on stalks of cohomology must be 0.  
\end{proof}

Now we note that $\log \pi$ is a global section of $\mathcal{O}_{X_\infty}$. There is a canonical morphism 
\[
  \alpha': \Omega^\bullet_{X}(\log D)(*P) \longrightarrow  k_*k^*\Omega_{X^*}^\bullet(\log D^*)(*P^*) \cong k_*\Omega_{X_\infty}^\bullet(\log D_\infty)(*P_\infty)
\]
and we may define the complex of sheaves $\Omega_X^\bullet(\log D)(*P)[\log \pi]$ whose local sections are, formally, polynomials $\sum_{i=1}^k \omega_i(\log \pi)^i$. The differential takes $d\log \pi = d\pi/\pi$. There is a canonical injective morphism of complexes
\[
\alpha :  \Omega_X^\bullet(\log D)(*P)[\log \pi]  \longrightarrow k_*\Omega_{X_\infty}^\bullet(\log D_\infty)(*P_\infty)
\]
which sends local sections $\sum_{i=1}^k \omega_i(\log \pi)^i$ to $\sum_{i=1}^k \alpha'(\omega_i)(\log \pi)^i$. Notice that there is also a natural morphism of complexes 
\[
    \eta: \Omega_X^\bullet(\log D\cup X_0)(*P) \longrightarrow \Omega_X^\bullet(\log D\cup X_0)(*P)[\log \pi]
\]
obtained by inclusion.
\begin{lemma}\label{l:qis2}
    There is a quasi-isomorphism between $i^*\Omega_X^\bullet(\log D \cup E)(*P)[\log \pi]$ and $\Omega_{X/\mathbb{D}}^\bullet(\log D\cup X_0)(*P)\otimes \mathcal{O}_{X_0}$.
\end{lemma}
\begin{proof}
     We filter the stalks of $i^*\Omega_X^\bullet(\log D \cup X_0)(*P)[\log \pi]$ by letting 
    \[
    \mathcal{G}_a = \left\{\left. \sum_{i=0}^a \omega_i (\log \pi)^i \,\, \right| \,\, \omega_i \in i^*\Omega_X^\bullet(\log D\cup X_0)(*P)_0\right\}.    
    \]
    It is clear that $\mathcal{G}_0\subseteq \mathcal{G}_1\subseteq \dots $ forms an increasing filtration, and that $\mathcal{H}^p(\gr^q_\mathcal{G})_0 \cong \mathcal{H}^{p}(i^*\Omega_X^\bullet(\log D\cup X_0)(*P))_0$. Therefore, the filtration spectral sequence has 
    \[
    E_1^{p,q}  \cong    \mathcal{H}^{p}(i^*\Omega_X^\bullet(\log D\cup X_0)(*P))_0
    \] 
    whenever $q\geq 0$. Direct calculation shows that 
    \[
    d_1: E_1^{p,q} \longrightarrow E_1^{p+1,q},\quad    \mathcal{H}^{p}(i^*\Omega_X^\bullet(\log D\cup X_0)(*P))_0 \xrightarrow{\wedge d\log \pi} \mathcal{H}^{p+1}(i^*\Omega_X^\bullet(\log D\cup X_0)(*P))_0 
    \]
    Now we have the commutative diagram,
    \[
    \begin{tikzcd}
        \mathcal{H}^p(i^*\Omega_X^\bullet(\log D\cup X_0)(*P))_0 \ar[d] \ar[rd,"\wedge d\log \pi"]& \\
        \mathcal{H}^p(i^*\Omega_{X/\mathbb{A}^1}^\bullet(\log D\cup X_0)(*P))_0 \ar[r,"\wedge d\log \pi"] & \mathcal{H}^{p+1}(i^*\Omega_X^\bullet(\log D\cup X_0)(*P))_0
    \end{tikzcd}
    \]
    By the proof of Lemma \ref{l:redstalk}, the vertical map is surjective, and the horizontal map is injective. Thus $d_1$ is exact except the map $d_1: E^{p,0}_1\rightarrow E^{p,-1}_1$, which is 0. Consequently, 
    \[
    E_2^{p,q} \cong \begin{cases} \mathcal{H}^p(\Omega_{X/\mathbb{A}^1}^\bullet(\log D\cup X_0)(*P) \otimes \mathcal{O}_{X_0})_0 & \text{ if } q =0 \\
        0 & \text{ otherwise} \end{cases}
    \]
    We immediately have $E_2$ degeneration and more precisely, that $E_2^{p,0} \cong \mathcal{H}^p(i^*\Omega_X^\bullet(\log D \cup X_0)(*P)[\log \pi])_0$. Furthermore, the natural map $i^*\eta$ described above is identified with the surjective morphism 
    \[
    \mathcal{H}^p(i^*\Omega_X^\bullet(\log D\cup X_0)(*P))_0 \longrightarrow \mathrm{coker}(\wedge d\log \pi)_0.
    \]
    Consequently, we see that both $i_*i^*\Omega_{X}^\bullet(\log D\cup X_0)(*P)$ and $\Omega_{X/\mathbb{D}}^\bullet(\log D\cup X_0)(*P)\otimes \mathcal{O}_{X_0}$ are quasi-isomorphic to the mapping cone of 
    \[
    \Omega_{X/\mathbb{D}}^{\bullet}(\log D\cup X_0)(*P)[-1]\otimes \mathcal{O}_{X_0} \xrightarrow{d\log \pi} \Omega_{X}^{\bullet}(\log D\cup X_0)(*P) \otimes \mathcal{O}_{X_0}.
    \]
    Hence they are mutually quasi-isomorphic.
\end{proof}

\begin{lemma}\label{l:qis3}
    The morphism $i^*\alpha : i^*\Omega_X^\bullet(\log D)(*P)[\log \pi]  \longrightarrow i^*k_*\Omega_{X_\infty}^\bullet(\log D_\infty)(*P_\infty)$ is a quasi-isomorphism of complexes of sheaves.
\end{lemma}
\begin{proof}
    We choose appropriate neighbourhoods $V_{\varepsilon,\gamma}$ of a point $0$ in $X_0$, which lie within a polydisc of radius $\varepsilon$ of $p$ and for which $|\pi(V)|< \gamma$. Let $W_{\varepsilon,\gamma} = k^{-1}V_{\varepsilon,\gamma}.$ Then, 
    \begin{align*}
    \mathcal{H}^p(i^*k_*\Omega_{X_\infty}^\bullet(\log D_\infty)(*P_\infty))_0 & = \lim_{\substack{ \longrightarrow \\ W_{\varepsilon,\gamma}}} H^p(\Omega_{W_{\varepsilon,\gamma}}^\bullet(\log D_\infty \cap W_{\varepsilon,\gamma})(*(P_\infty\cap W_{\varepsilon,\gamma}))) \\ 
    \mathcal{H}^p(i^*\Omega_{X}^\bullet(\log D)(*P))_0 & = \lim_{\substack{ \longrightarrow \\ V_{\varepsilon,\gamma}}} H^p(\Omega_{V_{\varepsilon,\gamma}}^\bullet(\log (D \cup X_0) \cap V_{\varepsilon,\gamma})(*(P\cap V_{\varepsilon,\gamma}))) 
    \end{align*}
    For small enough $\varepsilon$, $V_{\varepsilon,\gamma} = \mathbb{D}^a \times (\mathbb{D}^*)^{d-a+1}$ and $W_{\varepsilon,\gamma}\cong \mathbb{D}^a \times (\mathbb{D}^*)^{d-a} \times \mathfrak{h}$ and the map $k$ is simply the exponential map in the last coordinate. If we choose $t$ so that $0<|t|<\gamma$, and $\tilde{t}$ so that $\exp(2\pi {\tt i}\tilde{t}) = t$ then there is a commutative diagram of maps, 
    \[
    \begin{tikzcd}
        H^p(\Omega_{V_{\varepsilon,\gamma}}^\bullet(\log (D \cup X_0) \cap V_{\varepsilon,\gamma})(*(P\cap V_{\varepsilon,\gamma}))) \ar[r,"B"] \ar[d,"C"] & H^p(\Omega_{V_{\varepsilon,\gamma}\cap \pi^{-1}(t)}^\bullet(\log (D \cup X_0) \cap V_{\varepsilon,\gamma})(*(P\cap V_{\varepsilon,\gamma})))\ar[d,"\cong"] \\
        H^p(\Omega_{W_{\varepsilon,\gamma}}^\bullet(\log D_\infty \cap W_{\varepsilon,\gamma})(*(P_\infty\cap W_{\varepsilon,\gamma})))  \ar[r,"A"] & H^p(\Omega_{W_{\varepsilon,\gamma}\cap \tilde{\pi}^{-1}(\tilde{t})}^\bullet(\log D_\infty \cap W_{\varepsilon,\gamma})(*(P_\infty\cap W_{\varepsilon,\gamma}))) 
    \end{tikzcd}
    \]
    Since $\tilde{\pi} : W_{\varepsilon,\gamma} \rightarrow \mathfrak{h}$ is a trivial local family of Landau--Ginzburg models, the arrow $A$ is an isomorphism. A local argument shows that the top horizontal arrow $B$ has kernel spanned by the image of $\wedge d\log \pi$. Therefore, the kernel of $C$ is also spanned by the image of $\wedge d\log \pi$ and the induced morphism on stalks 
    \[
   K: \mathcal{H}^p(i^*\Omega_{X}^\bullet(\log D)(*P))_0 \longrightarrow \mathcal{H}^p(i^*k_*\Omega_{X_\infty}^\bullet(\log D_\infty)(*P_\infty))_0
    \]
    has kernel spanned by the image of $\wedge d\log \pi$. This map factors through the morphism
    \[
    \begin{tikzcd}
    \mathcal{H}^p(i^*\Omega_{X}^\bullet(\log D)(*P))_0\ar[r]\ar[rr,bend right=5,swap,"K"] &\mathcal{H}^p(i^*\Omega_{X}^\bullet(\log D)(*P)[\log \pi])_0 \ar[r,"L"] & \mathcal{H}^p(i^*k_*\Omega_{X_\infty}^\bullet(\log D_\infty)(*P_\infty))_0.
    \end{tikzcd}
    \]
    By the argument in the proof of Lemma \ref{l:qis2}, $L$ is an isomorphism.
\end{proof}

As a consequence of these three lemmas, we have a quasi-isomorphism 
\[
\psi_\pi \Omega_{X^*}^\bullet(\log D^*)(*P^*) \dashrightarrow \Omega_X^\bullet(\log D\cup X_0)(*P) \otimes \mathcal{O}_{X_0}.
\]
It remains to show that $\mathbb{H}^p(E, \psi_\pi \Omega^\bullet_{X^*}(\log D^*)(*P^*))$ has rank equal  to that of $H^p(X_t\setminus D_t,w_t)$ for a generic point $t$ in  $\mathbb{D}$, however, this is a consequence of general arguments (see e.g. \cite[II, \S 6.13]{goresky1988stratified} or \cite[Theorem 4.14]{MAXIM_2020}), along with the fact that $(\Omega^\bullet_{X^*}(\log D^*)(*P^*),d+dw)$ is quasi-isomorphic to a constructible complex \cite[Lemma C.2]{esnault20171}.  We have the following result.
\begin{theorem}\label{t:constdim}
    Let $(X,D,w,\pi)$ be a quasi-stable degeneration of Landau--Ginzburg models over a disc $\mathbb{D}$. The coherent sheaf $\mathbb{R}^p\pi_*\Omega_{X/\mathbb{D}}^\bullet(\log D\cup X_0)(*P)$ is locally free. 
\end{theorem}
\begin{proof}
    This is a direct consequence of Grauert's semicontinuity theorem (see e.g. \cite[Theorem 8.5]{barth2015compact}). The arguments above show that the rank of the stalks of $\mathbb{R}^p\pi_*\Omega_{X/\mathbb{D}}^\bullet(\log D\cup X_0)(*P)$, which are isomorphic to $\mathbb{H}^p(\Omega_{X/\mathbb{D}}^\bullet(\log D\cup X_0)(*P)\otimes \mathcal{O}_{X_t})$, have constant rank. Therefore, $\mathbb{R}^p\pi_*\Omega^\bullet_{X/\mathbb{D}}(\log D\cup X_0)(*P)$ is locally free by \cite[Theorem 8.5(iii)]{barth2015compact}.
\end{proof}
The following corollary is a direct consequence of Theorem \ref{t:constdim}. 
\begin{corollary}
Let $(X,D,w,\pi)$ be a quasi-stable degeneration of Landau--Ginzburg models and let $T_p$ be the monodromy operator attached to $\mathbb{R}^p\pi_*\Omega_{X^*/\mathbb{D}^*}^\bullet(\log D^*)(*P^*)$. Then $\log T_p = \mathrm{Res}_0\nabla_\mathrm{GM}$. 
\end{corollary}
We may also adapt usual arguments (e.g. \cite[Corollary 11.19]{peters2008mixed}) to obtain the following statement.
\begin{corollary}
    Let $(X,D,w,\pi)$ be a quasi-stable degeneration of Landau--Ginzburg models. Then $\log T_p$ is nilpotent. 
\end{corollary}

\subsection{Irregular filtration on the nearby cycles}
There is a natural irregular Hodge filtration on the nearby twisted de Rham complex $(\Omega_{X/\mathbb{D}}^\bullet(\log D\cup X_0)(*P) \otimes \mathcal{O}_{X_0},d+dw)$. Explicitly, let us define $\Lambda_{X,D}^\bullet = \Omega_{X/\mathbb{D}}^\bullet(\log D\cup X_0)$ and we define 
\begin{equation}\label{e:nbirreg}
F^\lambda_{\mathrm{irr}}(\Lambda_{X,D}^\bullet(*P) \otimes \mathcal{O}_{X_0}) := \mathrm{im}((F^\lambda_{\mathrm{irr}}\Lambda^\bullet_{X,D}(*P)) \otimes \mathcal{O}_{X_0}\longrightarrow \Lambda^\bullet_{X,D}(*P)\otimes \mathcal{O}_{X_0}).
\end{equation}
The sheaf $\mathcal{O}_{X_0} = \mathcal{O}_X/\mathcal{I}_{X_0}$ is not a flat $\mathcal{O}_X$-module, so it is not immediately clear that $F_{\mathrm{irr}}^\lambda \Lambda_{X,D}^\bullet(*P)$ is a subsheaf of $\Lambda_{X,D}^\bullet(*P)$, but this may be checked locally on stalks. 
\begin{proposition}\label{p:locslog}
    The natural morphism $(F^\lambda_{\mathrm{irr}}\Lambda^\bullet_{X,D}(*P)) \otimes \mathcal{O}_{X_0}\rightarrow \Lambda^\bullet_{X,D}(*P)\otimes \mathcal{O}_{X_0}$ is injective.
\end{proposition}
\begin{proof}
    We choose local coordinates for $\Omega_{X/\mathbb{D}}^p(\log D)(*P)$ called $x_1,\dots, x_a, y_1,\dots, y_b,z_1,\dots, z_c$ where $D = x_1\dots x_a$, $w =1/(x^{e_1}_1\dots x^{e_\ell}_\ell)$ and $\pi = y_1\dots y_b$. Then, locally, $F^\lambda_\mathrm{irr}\Omega_{X/\mathbb{D}}^p(\log D\cup X_0)(*P)$ is isomorphic to
    \[
   \left( \dfrac{1}{\underline{x}^\alpha}\right) \mathbb{C}\{\underline{x},\underline{y},\underline{z}\} \otimes_\mathbb{C}\bigwedge^p \left(\mathrm{span}_\mathbb{C} \left\{ \dfrac{dx_1}{x_1}, \dots , \dfrac{dx_a}{x_a}, \dfrac{dy_1}{y_1},\dots , \dfrac{dy_b}{y_b}, dz_1,\dots, dz_c\right\} \bmod d\log \pi \right)
    \]
    where $\underline{x}^\alpha=x_1^{\alpha_1}\dots x_\ell^{\alpha_\ell}$ and $\alpha_i = \lfloor(p-\lambda)e_i\rfloor$.  Furthermore, $\Omega_{X/\mathbb{D}}^p(\log D\cup X_0)(*P)$ is isomorphic to 
    \[
    \mathbb{C}\{\underline{x}, \underline{y}, \underline{z}\}[\underline{x}^{-1}] \otimes_\mathbb{C}\bigwedge^p \left(\mathrm{span}_\mathbb{C} \left\{ \dfrac{dx_1}{x_1}, \dots , \dfrac{dx_a}{x_a}, \dfrac{dy_1}{y_1},\dots , \dfrac{dy_b}{y_b}, dz_1,\dots, dz_c\right\} \bmod d\log \pi \right).
    \]
    We have the obvious inclusion. Taking the tensor product with $\mathcal{O}_{X_0}$ amounts to taking the quotient by the submodule generated by $\pi$. The induced  morphism of $\mathbb{C}\{\underline{x},\underline{y},\underline{z}\}$-modules is injective.
\end{proof}

There is an induced filtration on the twisted nearby cohomology, $\mathbb{H}^p({X_0},\Lambda_{X,D}^\bullet(*P)\otimes \mathcal{O}_{X_0})$. In the case where $w = c$ is constant, this simply reduces to the usual Hodge filtration on nearby cohomology, thus the spectral sequence attached to $F_{\mathrm{irr}}^\bullet$ degenerates at the initial term. We expect that the following property holds frequently. 

\begin{property}\label{c:E1deg}
    Suppose $(X,D,w,\pi)$ is a quasi-stable degeneration of Landau--Ginzburg models. The natural morphism $\mathbb{H}^p(F_{\mathrm{irr}}^\lambda\Lambda^\bullet_{X,D}(*P)\otimes \mathcal{O}_{X_0})\rightarrow \mathbb{H}^p(\Lambda^\bullet_{X,D}(*P)\otimes \mathcal{O}_{X_0})$ is injective for all $\lambda$.
\end{property}
Following standard Hodge theoretic arguments \cite[Chapters 10, 11]{peters2008mixed}, we obtain the following results.

\begin{proposition}\label{p:hnconst}
    Suppose $(X,D,w,\pi)$ is a quasi-stable degeneration of Landau--Ginzburg models for which Property \ref{c:E1deg} holds. Then for any $t \neq 0$, we have 
    \[
    \dim \gr^{F_\mathrm{irr}}_\lambda H^p(X_t,D_t,w_t) = \dim \gr^{F_\mathrm{irr}}_\lambda \mathbb{H}^p(E,\Lambda_{X,D}^\bullet(*P)\otimes \mathcal{O}_E).
    \]
\end{proposition}

Briefly, if Property \ref{c:E1deg} holds, we see that for every $t \in \mathbb{D}$,
\begin{align*}
F_{\mathrm{irr}}^\lambda\mathbb{H}^p(\Lambda_{X,D}^\bullet(*P)\otimes \mathcal{O}_{X_t}) & \cong \mathbb{H}^p(F_{\mathrm{irr}}^\lambda(\Lambda_{X,D}^\bullet(*P)\otimes \mathcal{O}_{X_t})) \\ 
& \cong \mathbb{H}^p((F_{\mathrm{irr}}^\lambda\Lambda_{X,D}^\bullet(*P))\otimes \mathcal{O}_{X_t}).
\end{align*}
By Grauert's semicontinuity theorem, the map 
\[
t \longmapsto \dim \mathbb{H}^p((F_{\mathrm{irr}}^\lambda\Lambda_{X,D}^\bullet(*P))\otimes \mathcal{O}_{X_t}) 
\]
is upper semi-continuous. Since $\sum_{\lambda}\dim \gr_{F_{\mathrm{irr}}}^\lambda \mathbb{H}^p(\Lambda_{X,D}^\bullet(*P)\otimes \mathcal{O}_{X_t}) = \dim \mathbb{H}^p(\Lambda_{X,D}^\bullet(*P)\otimes \mathcal{O}_{X_t})$ for all $t$. By Theorem \ref{t:constdim}, $\dim  \mathbb{H}^p(\Lambda_{X,D}^\bullet(*P)\otimes \mathcal{O}_{X_t})$ is constant as $t$ varies, so $\dim \gr_{F_{\mathrm{irr}}}^\lambda \mathbb{H}^p(\Lambda_{X,D}^\bullet(*P)\otimes \mathcal{O}_{X_t})$ must also be constant.

One of the results of the next section is that Property \ref{c:E1deg} holds for a certain class of quasi-stable degenerations coming from toric geometry.

\section{Toric Landau--Ginzburg models and tropical twisted cohomology}\label{s:nearbyfib}

In this section, we prove a major technical result used in this paper, Theorem \ref{t:descent}, that for a degenerating Landau--Ginzburg model $(T(\Sigma),w_\varphi)$, the irregular Hodge numbers may be calculated from a sheaf on the polyhedral complex $\mathsf{T}(\Sigma,w_\varphi)_0$. As the details of this result play only a minor role in the discussion above, we have chosen to present it in the later portions of this paper. Additionally, we work in a more general setup than is strictly required for constructing degenerations of Landau--Ginzburg models from Clarke dual pairs. Applications to Clarke dual pairs are summarized at Section \ref{s:apptoClarke}.

\subsection{Construction of a quasi-stable degeneration of toric Landau--Ginzburg models}\label{s:combdeg}

We set up the notation and recall construction of the tropicalization discussed in Section \ref{s:thyper}. Suppose $A \subset M$ is a finite collection of integral points. Let $\Delta$ be its convex hull and let $\mathrm{nf}(\Delta)$ denote the normal fan of $\Delta$. We may choose a function $\varphi : A \rightarrow \mathbb{Z}$ so that the normal fan of 
\[
\widetilde{\Delta} = \{ (n,z) \in N_\mathbb{R}  \times \mathbb{R} \mid  n(m)+\varphi(m)  \geq -z \quad \forall \,\, m \in A\}.
\]
is a simplicial refinement of the fan $\mathrm{cone}(A \times 1)$.

As before, we let
\[
\mathsf{w}_\varphi : N_\mathbb{R} \longrightarrow  \mathbb{R}, \qquad
n \in N_\mathbb{R} \longmapsto \min_{m\in A} \{ n(m)+\varphi(m)\}.
\]
The domains along which $\mathsf{w}_\varphi$ is linear determine a polyhedral decomposition of $N_\mathbb{R}$ that we denote $\mathrm{PC}(w_\varphi)$. Without loss of generality, we may scale $\varphi$ so that the vertices of $\mathrm{PC}(w_\varphi)$ are at integral points in $N_\mathbb{R}$. We may also shift $\varphi$ so that $\varphi(0_M) = 0$. We distinguish two classes of strata in $\mathrm{PC}(w_\varphi)$; we let $\mathrm{PC}(w_\varphi)_0$ denote those strata along which $\mathsf{w}_\varphi$ vanishes identically. In other words, $\mathrm{PC}(w_\varphi)_0$ is a polyhedral complex structure of a unique maximal-dimensional cell $\sigma_0$ which corresponds to $0_M \in A$ under the bijection in Theorem \ref{t:subdiv}. Let $\Sigma_0$ be any refinement of $\mathrm{nf}(\Delta)$. 

\begin{assumption}\label{a:1-4}
Let $A$ be a pointed set of points in $M$ along with a pointed triangulation determined by a function $\varphi : A\rightarrow \mathbb{Z}$. Let $\Sigma_0$ be a simplicial, quasiprojective refinement of $\mathrm{nf}(\Delta)$. We assume that $\Sigma \subseteq N \times \mathbb{Z}$ satisfies the following conditions.
\begin{enumerate}
    \item $\Sigma \cap (N\times 0) = \Sigma_0$,
    \item the polyhedral complex $\{c \cap N \times 1 \mid c \in \Sigma\}$ is a refinement of $\mathrm{PC}(w_\varphi)$,
    \item every ray $r$ of $\Sigma$ is of the form $(n,0)$ or $(n,1)$,
    \item if $r = (n,1)$ then $\mathsf{w}_\varphi(n) =0$. In other words, $n \in \mathrm{PC}(w_\varphi)_0$.
\end{enumerate}
\end{assumption}
The fan $\Sigma$ may be projected onto $\mathbb{R}_{\geq 0}$, therefore the toric variety $T(\Sigma)$ admits a projection map to  $\mathbb{A}^1$ that we denote $\pi$. The corresponding toric rational function will be denoted by $t$. The ring of characters on $T(\Sigma)$ is then identified with $\mathbb{C}[t^\pm,M]$. According to Theorem \ref{t:subdiv} each cell $\sigma \in \mathrm{PC}(w_\varphi)$ is determined by a collection of points $A_\sigma\subseteq A$ which we define to be
\[
A_\sigma := \{ m \in  A \mid n(m) + \varphi(m) = \mathsf{w}_\varphi(n), n \in \sigma \}.
\]
Following notation in Section \ref{s:thyper}, we let $\tau_\sigma:=\mathrm{Conv}(A_\sigma)$ and $f_{\tau_\sigma}$ be the minimal face of $\Delta$ containing $\tau_\sigma$. For the following discussion, we choose a simplicial refinement $\Sigma_0$ of $\mathrm{nf}(\Delta)$. We obtain a family of Laurent polynomials 
\[
w_\varphi = \sum_{m \in A} u_m t^{\varphi(m)} \underline{x}^m 
\]
where $u_m \in \mathbb{C}^*$ are viewed as parameters. Condition (4) implies that $\mathsf{w}_\varphi(n) \leq n(0_M) + \varphi(0_M) = 0$ for all $n$, and it implies that $w_\varphi$, viewed as a family of rational functions on $T(\Sigma_0) \times \mathbb{D}^*$, does not vanish along any toric divisor of $T(\Sigma_0)$. For simplicity, we will assume $a_0 = 1$. We view this both as a generic Laurent polynomial with monomial support $\{(m,\varphi(m)) \mid m \in A'\} \cup (0,0)$ and as a rational function on $T(\Sigma)$. 

Suppose we are given a maximal cone $c$ of $\Sigma$ with generators 
\[
(n_1,1),\dots, (n_a,1), (n_{a+1},0),\dots, (n_{d+1},0) \in N \times \{0,1\}.
\] 
Let $c$ be a maximal cone of $\Sigma$, there is an orbifold chart $T(\Sigma)_c$ with coordinates $y_1,\dots, y_{a},z_1,\dots, z_b$ so that $y_1,\dots, y_a$ corresponding to ray generators $(n_1,1),\dots, (n_a,1)$ and $z_{a+1},\dots, z_{d+1}$ correspond to ray generators $(n_{a+1},0),\dots, (n_{d+1},0)$ respectively. In these coordinates, $w_\varphi$ is written as:
\begin{equation*}
w_\varphi|_{T(\Sigma)_c} = \sum_{m\in A}u_m \left( \prod_{i=1}^a y_i^{n_i(m) + \varphi(m)}    \right) \left( \prod_{i=a+1}^{d+1} z_i^{n_i(m)}    \right).
\end{equation*}
The fibre $\pi^{-1}(0)$ is the union of toric divisors $\prod y_i = 0$, which we call {\em vertical divisors}. The remaining toric boundary divisors will be called {\em horizontal divisors}. Let $Z$ denote the zero divisor of $w_\varphi$ in $T(\Sigma)$. 
The following statement follows directly from the definitions. Recall the bijection in Theorem \ref{t:subdiv} and the notation used therein.
\begin{proposition}
  Suppose $\Sigma_0$ is a refinement of $\mathrm{nf}(\Delta)$ and suppose $\Sigma$ satisfies Assumption \ref{a:1-4}. Let $\sigma_\kappa$ be the minimal cell in $\mathrm{PC}(w_\varphi)$ containing $\kappa \cap( N \times 1)$. For each cone $\kappa \in \Sigma$ define $\kappa_1[1]$ to be the ray generators of $\kappa$ of the form $(n,1)$ and let $\kappa_0[1]$ be  the ray generators of the form $(n,0)$. Then
  \[
  A_{\sigma_\kappa} = \left\{m \in  A \,\, \left| \,\, \begin{array}{ll} \varphi(m) + n(m) = \mathsf{w}_\varphi(n) &\text{ for all }n \in \kappa_1[1] ,\\ n(m)= \min_{n'\in \Sigma_0[1]}\{n'(m)\} &\text{ for all } n \in \kappa_0[1] \end{array}\right.\right\}.
  \]

\end{proposition}

The following result is now a direct computation.
\begin{lemma}\label{l:poly}
  Let $A$ be a collection of points, $\Sigma_0$ a complete fan refining $\mathrm{nf}(\Delta)$, and let $\Sigma$ be a fan satisfying Assumption \ref{a:1-4} listed above. For each stratum $T_{\kappa}$ of $T_0$, $Z$ is the vanishing locus of 
       \begin{equation*}
      w_{{\sigma_\kappa}} = \sum_{m \in {A}_{\sigma_\kappa}} u_m \underline{x}^m \in \mathbb{C}[\kappa^\perp].
      \end{equation*}
      If $T_\kappa$ is not contained in $P_\mathrm{red}$ then $w_\varphi$ neither vanishes nor has pole along $T_\kappa$.
  \end{lemma} 
  
  \begin{proof}
      Assume that $T_{\kappa}$ is a stratum of $T_0$ and let $c$ be a maximal cone of $\Sigma$ containing $\kappa$. Let the coordinates $z_i, y_i$ be as in the discussion above. We may assume that $T_{\kappa}$ is the vanishing locus of a collection of variables $y_1,\dots, y_k$ and $z_1,\dots, z_\ell$ corresponding to the elements of $\kappa_1[1]$ and $\kappa_0[1]$ respectively. 

      By condition (4), $\mathsf{w}_\varphi(n_i) = 0$ for all $n_i \in \kappa_1[1]$. Therefore, $w_\varphi$ is regular along all vertical divisors. In fact, since $n_i(0_M) + \varphi(0_M) = 0$ by construction, $w_\varphi$ is also nonvanishing along all vertical divisors. Similarly, since $0_M \in A$ and $n_i(0_M) = 0$ for every $n_i$, so no $z_i$ appears to a positive power in all monomials, or in other words, $w_\varphi$ does not vanish along any horizontal divisor. The intersection locus of $Z$ and $T_c$ is the vanishing locus of the Laurent polynomial obtained by clearing all denominators involving $z_1,\dots, z_\ell$. The monomials $t^{\varphi(m)}\underline{x}^m$ which do not vanish when $y_1,\dots,y_k$ and $z_1,\dots, z_\ell$ vanish therefore are precisely those which correspond to members of $A_{\sigma_\kappa}$. 
  \end{proof}

\begin{proposition}\label{p:poly}
  Choose a collection of horizontal divisors $D$ which contains all irreducible components of $P_\mathrm{red}$. If $\Sigma$ is a fan satisfying Assumption \ref{a:1-4}, the data $(T(\Sigma), D, w_\varphi,\pi)$ defines a quasistable degeneration of Landau--Ginzburg models.
\end{proposition}
\begin{proof}
  By assumption, $T(\Sigma)$ is an orbifold toric variety, the poles of $w_\varphi$ are along toric boundary divisors, and the fibre of $\pi$ over 0 is also a collection of toric boundary divisors. Thus $X_0\cup D$ has orbifold normal crossings. Because ray generators of $T(\Sigma)$ have height 0 or 1, $\pi$ vanishes to order 0 or 1 along all toric boundary divisors. Finally, we need to check that $Z\cap \pi^{-1}\mathbb{D}$ is quasismooth and that that $Z\cup X_0 \cup D$ is orbifold normal crossings. The fact that coefficients $u_m$ are chosen generically means that $Z$ is a quasi-smooth hypersurface in $T(\Sigma)$ if it does not vanish identically along any toric stratum. The calculations in Lemma \ref{l:poly} ensure this. 
\end{proof}

\begin{remark}
    We comment on the roles played by conditions (1)--(4) in Assumption \ref{a:1-4}. Condition (1) ensures that $\pi$ has generic fibre $T(\Sigma_0)$ and (2) ensures that $w_\varphi$ is nondegenerate on $T(\Sigma)$ as long as $|\pi(x)| \ll 1$. Condition (3) ensures that $\pi$ is a quasi-stable degeneration. Condition (4) ensures that $w_\varphi$ has no poles on the fibre over $0$.
\end{remark}

\subsection{Polyhedral subdivisions}\label{s:polysub}

Let the notation be as in the previous section. We first show that for any simplicial, quasiprojective refinement $\Sigma_0$ of $\mathrm{nf}(\Delta)$, there exists a refinement of $\mathrm{PC}(w_\varphi)$ which satisfies both $(2)$ and $(4)$ in Assumption \ref{a:1-4}. 
    
\begin{proposition}\label{p:refined-complex}
 Given any simplicial refinement $\Sigma_0$ of $\mathrm{nf}(\Delta) = \mathrm{rec}(\mathrm{PC}(w_\varphi))$, there is a polyhedral complex $\mathrm{PC}(w_\varphi,\Sigma_0)$ which refines $\mathrm{PC}(w_\varphi)$ and whose recession fan is $\Sigma_0$ and whose $0$-dimensional strata are all contained in $\mathrm{PC}(w_\varphi)_0$. In particular, each cell in $\mathrm{PC}(w_\varphi,\Sigma_0)$ is the Minkowski sum of a compact simplex $\sigma$ in $N$ and a (possibly trivial) simplicial cone $c$ in $N$.
\end{proposition}

\begin{proof}

It suffices to construct a desired refinement of $\mathrm{PC}(w_\varphi)$. We take a similar approach used in \cite{Hasse2002integral}. Let $\mathrm{SD}(w_\varphi)$ denote the star coherent subdivision at $0$ of $\Delta$. There is a unique cell $\sigma_0 \in \mathrm{PC}(w_\varphi)$ corresponding to $0 \in A$. Suppose $\dim\Delta = \mathrm{rk}M$. There are two cases to consider: either $\sigma_0$ is compact or non-compact.

\begin{enumerate}
    \item ($\sigma_0$ is compact). Note that $\sigma_0$ is the unique compact maximal cell. Consider a barycentric subdivision of $\mathrm{SD}(w_\varphi)$, denoted by $\widetilde{\mathrm{SD}}(w_\varphi)$. By  Theorem \ref{t:subdiv}, $\widetilde{\mathrm{SD}}(w_\varphi)$ induces a barycentric subdivision of $\sigma \cap \sigma_0$ for each non-compact cell $\sigma \in \mathrm{PC}(w_\varphi)$. These subdivisions then induce subdivisions of both $\sigma_0$ and every non-compact cell $\sigma$ as follows: 
    \begin{itemize}
        \item For $\sigma_0$, choose an interior rational point $p$ of $\sigma_0$ and take a cone over the subdivision of $\sigma \cap \sigma_0$.
        \item For $\sigma$, add the corresponding recession cone $\mathrm{rec}(\sigma)$ at the center of $\sigma \cap \sigma_0$.
    \end{itemize}
    Let the resulting subdivision be denoted by $\widetilde{\mathrm{PC}}(w_\varphi)$. Each non-compact cell $\sigma' \in \widetilde{\mathrm{PC}}(w_\varphi)$ is the Minkowski sum of its recession cone and $\sigma' \cap \sigma_0$. Therefore, we can further coherently subdivide each non-compact cell in $\widetilde{\mathrm{PC}}(w_\varphi)$ with respect to the given refinement data $\Sigma_0$. We denote the resulting subdivision by $\mathrm{PC}(w_\varphi, \Sigma_0)$. By construction, this is a well-defined polyhedral subcomplex of $\mathrm{PC}(w_\varphi)$, and all $0$-dimensional strata are contained within $\sigma_0$.
    
    \item ($\sigma_0$ is non-compact). In this case, we first subdivide $\mathrm{PC}(w_\varphi)$ to apply the previous argument. As before, a barycentric subdivision $\widetilde{\mathrm{SD}}(w_\varphi)$ induces a subdivision of each compact cell of the form $\sigma \cap \sigma_0$, leaving non-compact cells of the form $\sigma \cap \sigma_0$ unrefined. By choosing an interior point $p$ of $\sigma_0$, we subdivide $\sigma_0$ as before. To make it a polyhedral complex, we must add the recession cone $\mathrm{rec}(\sigma_0)$ at $p$ in $\sigma_0$. Then this subdivision of $\sigma_0$ has cells of the product form. Also, at each center of compact $\sigma \cap \sigma_0$, we add the corresponding recession cone $\mathrm{rec}(\sigma)$ to subdivide $\sigma$. We denote the resulting subdivision by $\widetilde{\mathrm{PC}}(w_\varphi)$. Since all the non-compact cells in $\widetilde{\mathrm{PC}}(w_\varphi)$ are of the product form, we can coherently subdivide further with respect to $\Sigma_0$. Then by the same reasoning, we obtain the desired refinement $\mathrm{PC}(w_\varphi, \Sigma_0)$. 
\end{enumerate}
Next, consider the case where $\dim\Delta < \mathrm{rk}(M)$. Let $M_\Delta$ be the minimal lattice containing $A$, and let $N_\Delta$ be its dual lattice. In this case, we first regard $w_\varphi$ as a Laurent polynomial in $\mathbb{C}[M_\Delta]$, denoted by $w_{\varphi,\Delta}$. We then perform the same construction to obtain $\widetilde{\mathrm{PC}}(w_{\varphi,\Delta})$. This complex is canonically embedded into $N_\mathbb{R}$, where it is coherently subdivided with respect to $\Sigma_0$. Specifically, if each cell $\sigma \in \widetilde{\mathrm{PC}}(w_{\varphi,\Delta})$ can be expressed as the Minkowski sum of its recession cone $\mathrm{rec}(\sigma)$ and its compact part $\mathrm{cpt}(\sigma)$, then for every cone $c$ containing $\mathrm{rec}(\sigma)$ as a face, we add cells of the form $c + \mathrm{cpt}(\sigma)$.
\end{proof}

\begin{defn}\label{d:sigma-varphi}
    Suppose $\Sigma_0$ is a complete simplicial quasiprojective refinement of $\mathrm{nf}(\Delta)$ and let $\mathrm{PC}(w_\varphi,\Sigma_0)$ be as in Proposition \ref{p:refined-complex}. Define $\Sigma_{\varphi} \subset N_\mathbb{R} \times \mathbb{R}_{\geq 0}$ to be the following set of cones:
    \begin{enumerate}
        \item For each cone $c \in \Sigma_0$ construct a cone $c = c \times 0$. 
        \item For each cell $\sigma$ in $\mathrm{PC}(w_\varphi,\Sigma_0)\times \{1\}$ construct the cone 
        \[\kappa(\sigma) = \mathrm{cone}(\mathrm{rec}(\sigma) \times \{0\},\sigma\times \{1\}).\]
    \end{enumerate}
    \end{defn}    
    The reader may verify that $\Sigma_\varphi$ is in fact a fan. Furthermore, the final statement in Proposition \ref{p:refined-complex} implies that $\Sigma_\varphi$ is in fact simplicial.
 \begin{proposition}\label{p:schon}
     Given a pointed subset $A \subseteq M$, a coherent star triangulation of $A$ based at $0_M$ induced by $\varphi$, and a simplicial refinement of $\mathrm{nf}(\Delta)$, there is a fan $\Sigma$ satisfying Assumption \ref{a:1-4}. 
 \end{proposition}
 \begin{proof}
     The fan $\Sigma_\varphi$ constructed in Definition \ref{d:sigma-varphi} satisfies Assumption \ref{a:1-4} by Proposition \ref{p:refined-complex}.
 \end{proof}

\subsection{The tropical cell complex of a toric family}

We finish this section by discussing tropicalization. Suppose $\Sigma_0$ and $w_\varphi$ are fixed, as in the previous section, and choose a fan $\Sigma$ satisfying Assumption \ref{a:1-4}. We obtain a tropical polyhedral complex $\mathsf{T}_\Sigma(\Sigma_0, w_\varphi)$ obtained by taking the closure of $\mathrm{PC}(w_\varphi,\Sigma_0)$ in $\mathrm{Trop}(\Sigma_0)$. Since $\mathsf{T}(\Sigma_0,w_\varphi)$ is the closure of $\mathrm{PC}(w_\varphi)$ in $\mathrm{Trop}(\Sigma_0)$, there is a refinement morphism $s:\mathsf{T}_\Sigma(\Sigma_0,w_\varphi) \rightarrow \mathsf{T}(\Sigma_0, w_\varphi)$.
 
From the discussion in Section \ref{s:thyper}, the cells in $\mathsf{T}_\Sigma(\Sigma_0, w_\varphi)$ are in bijection with pairs $(c,\sigma)$ where $(c,\sigma) \in \Sigma_0 \times \mathrm{PC}(w_\varphi,\Sigma_0)$ and $c \subseteq \mathrm{rec}(\sigma)$. We write the corresponding poset as: 
\[
\mathsf{S}_{\mathsf{T}} = \{(c,\sigma)  \mid (c,\sigma) \in \Sigma_0 \times \mathrm{PC}(w_\varphi,\Sigma_0), c\subseteq \mathrm{rec}(\sigma)\} 
\]
where $(c,\sigma) \prec (c',\sigma')$ if $c' \subseteq c$ and $\sigma \subseteq \sigma'$. We use the notation $\mathsf{p}(c,\sigma)$ to denote the cell of $\mathsf{T}_\Sigma(\Sigma_0,w_\varphi)$ corresponding to $(c,\sigma)$. 
For later use, we will describe the topology of $\mathsf{p}(c,\sigma)$. In the following lemma, $\Delta_{q}$ denotes the standard $q$-simplex.

        \begin{lemma}\label{lem:simptrop}
            The cells ${\mathsf{p}}(c,\sigma)$ of  $\mathsf{T}_\Sigma(\Sigma_0,w_\varphi)$ are homeomorphic to a product $\Delta_{k-l} \times [0,\infty]^{l}$ where $k = \dim \sigma$ and $l = \dim \mathrm{rec}(\sigma)$. The intersection between $\mathsf{p}(c,\sigma)$ and each tropical boundary divisor is either empty or of the form $\Delta_{k-l} \times [0,\infty]^{l-1}\times \infty$.
         \end{lemma}
         \begin{proof}
            Assume without loss of generality that $c=0$. Then the cone $\kappa({\sigma})$ (Definition \ref{d:sigma-varphi}) has generators $\sigma_0 = \{(\rho,i) \in {\Sigma}[1] \mid i=0\}$ and  $\sigma_1 = \{(\rho,i) \in {\Sigma}[1] \mid i=1\}$. Here $\Sigma[1]$ denotes the ray generators of ${\Sigma}$. After (rational) change of basis, we may assume that $\sigma_1 = \{v_1,\dots, v_{k-l}\}$ and $\sigma_0 = \{e_{1},\dots, e_l\}$ where $e_1,\dots, e_d$  form a basis of $\mathbb{R}^d$. The affine patch in the tropical toric variety containing the closure of $\mathsf{p}(0,\sigma)$ is homeomorphic to $ \mathbb{R}^{k-l}\times (-\infty,\infty]^{l}\times \mathbb{R}^{d-k}$, in which 
             \begin{equation*}
             {\mathsf{p}(0,\sigma)} = \left\{\left. \sum_{i=1}^{k-l} a_i v_i + \sum_{j=1}^{l}b_j e_j\,\, \right|\,\, 0 \leq a_i \leq 1, \sum_{i}a_i =1, b_j \in [0,\infty]\right\}.
             \end{equation*}
             Projecting onto $(a_1,\dots, a_{k-l}, b_1,\dots, b_l)$, we obtain the desired homeomorphism. 
         \end{proof}
         \begin{remark}
         In other words, the cells of $\mathsf{T}_\Sigma(\Sigma_0,w_\varphi)$ are homeomorphic to the closed ball and its boundary. This implies that $\mathsf{T}_\Sigma(\Sigma_0,w_\varphi)$ is a regular cell complex (Definition \ref{d:regcell}). This fact will be useful when we want to describe irregular Hodge numbers in terms of tropical cohomology groups in the next section.
         \end{remark}
 \begin{defn}
Choose a union of irreducible toric boundary divisors, $D$, and let $\Sigma_{0,D}$ be a subfan of $\Sigma_0$ that corresponds to the complement of $D$. We define the tropical polyhedral subcomplex of $\mathsf{S}_\mathsf{T}$ and its corresponding poset:
\[
    \mathsf{T}_\Sigma(\Sigma_{0,D},w_\varphi) := \{\mathsf{p}(c,\sigma) \in \mathsf{T}_\Sigma(\Sigma_0,w_\varphi)  \mid c \in \Sigma_{0,D}\},\qquad \mathsf{S}_{\mathsf{T},\mathsf{D}} = \{ (c,\sigma) \mid \mathsf{p}(c,\sigma) \in \mathsf{T}_\Sigma(\Sigma_{0,D},w_\varphi)\}.  
\]
    The poset of compact cells in $\mathsf{S}_{\mathsf{T},\mathsf{D}}$ is 
    \[
        \mathsf{S}_{\mathsf{T},\mathsf{D}}^\mathrm{cpt} := \{(c,\sigma) \in \mathsf{S}_\mathsf{T} \mid  c\subseteq \mathrm{rec}(\sigma) \in \Sigma_{0,D}\} \subseteq \mathsf{S}_{\mathsf{T},\mathsf{D}}.
    \]
 \end{defn}
In the next section, we will construct a $\mathbb{Q}$-graded sheaf of $\mathbb{C}$-vector spaces on $\mathsf{S}_{\mathsf{T},\mathsf{D}}$ whose cohomology we would like to compute. To do so, we need to prove the following result.

 \begin{proposition}\label{p:cellcomplex}
    The polyhedral complex $\mathsf{T}_{\Sigma}(\Sigma_{0,D},w_\varphi)$ is a cell complex (see Definition \ref{d:regcell}).
 \end{proposition}
 \begin{proof}
    By Lemma \ref{lem:simptrop}, the tropical polyhedral complex $\mathsf{T}_\Sigma(\Sigma_0,w_\varphi)$ is a regular cell complex. The tropical toric variety $\mathrm{Trop}(\Sigma_{0,D})$ is obtained from $\mathrm{Trop}(\Sigma_{0,D})$ by removing a number of toric boundary divisors. For each cell $\mathsf{p}(c,\sigma) \cong \Delta_k \times [0,\infty]^{l-k}$, the corresponding subset of $\mathsf{T}_\Sigma(\Sigma_{0,D},w_\varphi)$ is either empty or isomorphic to 
    \begin{equation}\label{e:simplex}
    \Delta_k \times \{(x_1,\dots, x_{l-k}) \in [0,\infty]^{l-k}\mid x_i \in [0,\infty], i\in I, x_j\in [0,\infty), j \notin I \}    
    \end{equation}
    for some subset $J$ of $[1,l-k]$. The sets in \eqref{e:simplex} and their boundaries are homeomorphic to $(\mathrm{int}(\mathbb{B}^l),\mathbb{B}^l \setminus \mathbb{B}^{l-1})$. The one-point compactification of such a set satisfies axiom (3) in Definition \ref{d:regcell}.
 \end{proof}
Proposition \ref{p:cellcomplex} tells us that we may use $\mathsf{S}_{\mathsf{T},\mathsf{D}}^\mathrm{cpt}$ to compute the cohomology of any cellular sheaf on $\mathsf{T}_\Sigma(\Sigma_{0,D},w_\varphi)$. Let $\mathsf{T}_{\Sigma}(\Sigma_{0,D},w_\varphi)_0$ be the subset of $\mathsf{T}_\Sigma(\Sigma_{0,D},w_\varphi)$ consisting of cells $\mathsf{p}(c,\sigma)$ which are contained in cells of $\mathsf{T}(\Sigma_{0,D},w_\varphi)_0$. By construction, $\mathsf{T}_\Sigma(\Sigma_{0,D},w_\varphi)_0$ is a refinement of $\mathsf{T}(\Sigma_{0,D},w_\varphi)_0$ therefore it is also a cell complex. We use the notation $\mathsf{S}_{\mathsf{T},\mathsf{D},0}$ to denote the pairs $(c,\sigma)$ so that $\mathsf{p}(c,\sigma) \in \mathsf{T}_\Sigma(\Sigma_{0,D},w_\varphi)_0$. Note that $\mathsf{S}_{\mathsf{T},\mathsf{D}}^\mathrm{cpt} \subseteq \mathsf{S}_{\mathsf{T},\mathsf{D},0}$.

\subsection{The tropical twisted cohomology sheaf}\label{s:tsheaf}
In this section we assume that we are in the situation described in Section \ref{s:combdeg}. We choose a fan $\Sigma$ satisfying Assumption \ref{a:1-4} and let $T:=T({\Sigma})$ for simplicity and we choose a toric boundary divisor $D$ of $T(\Sigma_0)$. By Proposition \ref{p:poly} there is a quasi-stable degeneration of Landau--Ginzburg models, $(T,D,w_\varphi, \pi)$ attached to this data. Let $T_0 = \pi^{-1}(0)$. The toric strata of $T_0$ are in bijection with cells of $\mathrm{PC}(w_\varphi,\Sigma_0)$. By Lemma \ref{l:poly} the strata along which $w_\varphi$ is regular are exactly those in $\mathrm{PC}(w_\varphi,\Sigma_0)_0$. In this section, we use this bijection and the nearby twisted cohomology complex to build a tropical sheaf $\tilde{\bf J}^\bullet$ on a simplicial complex $\mathsf{T}_\Sigma(\Sigma_{0,D},w_\varphi)_0$. In the next section, we will show that the (tropical) cohomology of $\tilde{\bf J}^\bullet$ is isomorphic to the (usual) cohomology of the nearby twisted de Rham complex.

For each cone $c$ of $\Sigma_0$ we let $\tilde{c} = c\times 0 \in {\Sigma}$ and let $T(c)$ denote the closed torus orbit attached to $\tilde{c}$, $B(c)$ is its toric boundary, and $P_c := T(c) \cap P$. Let $T(c)_0$ denote $T(c) \cap T_0$. We use similar notation for $\sigma \in \Sigma_1$: Let $T({\sigma})$ denote the closed torus orbit attached to $\kappa({\sigma}) = \mathrm{cone}(\mathrm{rec}(\sigma)\times 0,\sigma\times 1)$ and $P_\sigma:=T({\sigma}) \cap P$. For each $(c,\sigma)$ in $\mathsf{S}_{\mathsf{T},\mathsf{D},0}$, we define 
\begin{align*}
\Lambda_{(c,\sigma)}^\bullet &:= \cOmega^\bullet_{T(c)/\mathbb{A}^1}(\log B(c))(*P_c) \otimes \mathcal{O}_{T({{\sigma}})} \cong \cOmega^\bullet_{T(c)/\mathbb{A}^1}(\log B(c)) \otimes \mathcal{O}_{T({\sigma})}(*P_\sigma) \\
\Lambda^\bullet_c &:= \cOmega^\bullet_{T(c)/\mathbb{A}^1}(\log B(c))(*P_c)\otimes \mathcal{O}_{T(c)_0}.
\end{align*}
As  before, $\Lambda_{(c,\sigma)}^\bullet$ is equipped with an irregular Hodge filtration (see~\eqref{e:nbirreg}). Similarly, we  define a filtration on $\Lambda^\bullet_c$;
\[
F_\mathrm{irr}^\lambda\Lambda^\bullet_{(c,\sigma)} = \mathrm{im}\left( F^\lambda_\mathrm{irr}\Omega^\bullet_{T(c)/\mathbb{A}^1}(\log B(c))(*P_c) \otimes  \mathcal{O}_{T({\sigma})} \longrightarrow \Omega^\bullet_{T(c)/\mathbb{A}^1}(\log B(c))(*P_c)\otimes \mathcal{O}_{T({\sigma})} \right).
\]
A local computation like the proof of Proposition \ref{p:locslog} identifies 
\begin{equation*}
F_\mathrm{irr}^\lambda\Lambda_{(c,\sigma)}^p = \begin{cases} \Omega_{T(c)/\mathbb{A}^1}^p(\log B(c))\otimes \mathcal{O}_{T(\tilde{\sigma})}(\lfloor (p-\lambda)P_\sigma\rfloor) & \text{ if } p \geq \lambda \\ 0 & \text{ otherwise} \end{cases}
\end{equation*}

In this section, we use the fact that $\Omega_{T(c)}^\bullet(\log B(c)) \cong \mathcal{O}_{T(c)} \otimes \wedge^\bullet \tilde{c}^\perp$ frequently to give very concrete descriptions of $\Lambda^\bullet_{(c,\sigma)}$ and $\Lambda^\bullet_c$. If $m_\pi = (0,1)$ denotes the element of $\tilde{M}$ which projects $\tilde{N} = N \times \mathbb{Z}$ onto the second coordinate, the corresponding rational function on $T$ is simply $\pi$. Thus, $\Omega_{T/\mathbb{A}^1}^\bullet(\log B) \cong \mathcal{O}_T \otimes \wedge^\bullet M/(\wedge^{\bullet-1}M \wedge m_\pi)$ and, since $\langle \tilde{c}, m_\pi \rangle = 0$ for any $c \subseteq \Sigma_{0,D}$, 
\begin{equation}\label{eq:trivsheaf} 
   \Lambda_c:= \Omega_{T(c)/\mathbb{A}^1}^\bullet(\log B(c))(*P_c) \otimes \mathcal{O}_{T(c)} \cong \mathcal{O}_{T(c)}(*P_c)\otimes \left[ \wedge^\bullet \tilde{c}^\perp/(\wedge^{\bullet-1}\tilde{c}^\perp \wedge m_\pi)\right].
\end{equation}
Therefore,
\[
\Lambda^\bullet_{(c,\sigma)} \cong \mathcal{O}_{T({\sigma})}(*P_\sigma)\otimes    \left[ \wedge^\bullet \tilde{c}^\perp/(\wedge^{\bullet-1}\tilde{c}^\perp \wedge m_\pi)\right] \cong \mathcal{O}_{T({\sigma})}(*P_\sigma) \otimes \left[  \wedge^\bullet (\tilde{c}^\perp/m_\pi)\right].
\] 
Furthermore, $\tilde{c}^\perp/m_\pi \cong c^\perp$ by projection. Suppose $(c,\sigma)\preceq (c',\sigma')$ with both $(c,\sigma), (c',\sigma') \in \mathsf{S}_{\mathsf{T},\mathsf{D},0}$. We define a morphism of sheaves,
\begin{align}
\alpha((c,\sigma),(c', \sigma')) : \mathcal{O}_{T({\sigma})}  \otimes \wedge^\bullet c^\perp & \longrightarrow \label{eq:resmap}\mathcal{O}_{T({\sigma}')}  \otimes \wedge^\bullet (c')^\perp  \\
g \otimes \omega & \longmapsto g|_{T({\sigma}')} \otimes i(\omega)
\end{align}
where $i : \wedge^\bullet c^\perp \hookrightarrow \wedge^\bullet (c')^\perp$ is the natural inclusion. Suppose $\tilde{c}\subseteq \kappa({\sigma})$, or equivalently, that $c\subseteq \mathrm{rec}(\sigma)$. Then there is an inclusion $\kappa({\sigma})^\perp \hookrightarrow \tilde{c}^\perp$ and since $m_\pi \notin \kappa({\sigma})^\perp$, there is also a natural injection $\kappa({\sigma})^\perp \hookrightarrow c^\perp$. Choosing a splitting $c^\perp \cong \kappa({\sigma})^\perp \oplus (c^\perp/\kappa({\sigma})^\perp)$ we obtain a decomposition 
\begin{align}
\Lambda_{(c,\sigma)}^\bullet &\cong \mathcal{O}_{T({\sigma})}(*P_\sigma)\otimes (\wedge^\bullet \sigma^\perp \otimes \wedge^\bullet ({c}^\perp/\kappa({\sigma})^\perp)) \\ &\cong \cOmega^\bullet_{T({\sigma})}(\log B(\sigma))(*P_\sigma) \otimes \wedge^\bullet ({c}^\perp/\kappa({\sigma})^\perp). \label{e:trivial}
\end{align}
This isomorphism commutes with $\nabla$ and preserves irregular filtrations. In particular, under this isomorphism,
\begin{align}\label{e:irrigation}
    F^\lambda_\mathrm{irr} \Lambda_{(c,\sigma)}^p = \bigoplus_q F_\mathrm{irr}^{\lambda - q} (\Omega_{T({\sigma})}^{p-q}(\log B(\sigma))(*P_\sigma) )\otimes \wedge^q(c^\perp/\kappa({\sigma})^\perp).
\end{align}
It is shown in \cite[Corollary 4.2]{yu2014irregular} that the natural map
    \[
    \mathbb{H}^p(T({\sigma}), F_\mathrm{irr}^{\lambda} \Omega_{T({\sigma})}^{\bullet}(\log B(\sigma))(*P_\sigma) ) \longrightarrow \mathbb{H}^p(T({\sigma}), \Omega_{T({\sigma})}^{\bullet}(\log B(\sigma))(*P_\sigma) )
    \]
    is injective for all $p$ and $\lambda$. The following claim then follows from~\eqref{e:irrigation}.
\begin{proposition}\label{p:propertyholds}
    The natural morphism 
    \[
    \mathbb{H}^p(T(\Sigma), F^\lambda_\mathrm{irr}\Lambda^\bullet_{(c,\sigma)}) \longrightarrow \mathbb{H}^p(T(\Sigma), \Lambda^\bullet_{(c,\sigma)})
    \]
    is injective for all $p \in \mathbb{Z}$ and $\lambda \in \mathbb{Q}$.
\end{proposition}

Let $w_\sigma := w|_{T({\sigma})}$ and assume that the monomial support of $w_\sigma$ is the vertex set of a simplex. Then by Proposition \ref{p:adolphson-sperber} we see that there is an isomorphism of bigraded vector spaces,  $\mathbb{H}^*(T({\sigma}),\cOmega^\bullet_{T({\sigma})}(\log B(\sigma))(*P_\sigma)) \cong {\bf B}_{\tau_\sigma}\otimes \wedge^\bullet \kappa({\sigma})^\perp \wedge \mathrm{Vol}(L(\Delta(w_\sigma)))$. Combining this, we obtain a filtered quasi-isomorphism 
\begin{equation}\label{e:adolph-sperb}
\xi_{(c,\sigma)}:{\bf B}_{\tau_\sigma} \otimes \wedge^\bullet c^\perp \wedge \mathrm{Vol}(L(\Delta(w_\sigma))) \longrightarrow \Gamma(T({\sigma}),\Lambda_{(c,\sigma)}^\bullet).
\end{equation}
We  note that this expression no longer depends on the choice of splitting that we made above. Whenever $(c,\sigma) \prec (c',\sigma')$ there is a commutative diagram 
\begin{equation}\label{e:cd}
\begin{tikzcd}
{\bf B}_{\tau_\sigma}\otimes \wedge^\bullet c^\perp \wedge \mathrm{Vol}(L(\Delta(w_\sigma))) \ar[r, "\xi_{({c},\sigma)}"] \ar[d,"p_{(c,\sigma)}^{(c',\sigma')}\otimes \varpi"] & \Gamma(T({\sigma}),\Lambda_{(c,\sigma)}^\bullet) \ar[d,"\alpha"] \\
{\bf B}_{\tau_\sigma}\otimes \wedge^\bullet( c')^\perp \wedge \mathrm{Vol}(L(\Delta(w_{\sigma'}))) \ar[r, "\xi_{({c'},\sigma')}"]  & \Gamma(T({\sigma}'),\Lambda_{(c',\sigma')}^\bullet) 
\end{tikzcd}
\end{equation}
\begin{defn}
    The {\em refined tropical Jacobian sheaf} of the quasistable degeneration of Landau--Ginzburg models $(T(\Sigma), D, \pi, w_\varphi)$ is the sheaf $\tilde{\bf J}^i$ of $\mathbb{Q}$-filtered vector spaces  on $\mathsf{T}_\Sigma(\Sigma_{0,D}, w_\varphi)_0$ so that 
    \[
    \tilde{\bf J}^i(c,\sigma) = ({\bf B}_{\tau_\sigma} \otimes \wedge^{i-\dim \tau_\sigma} {c}^\perp \otimes \mathrm{Vol}(L(\tau_\sigma)), \mathsf{F}^\bullet).
    \]
\end{defn}
Note that the sheaf $\tilde{\bf J}^i$ is not identical to the sheaf ${\bf J}^i$, because ${\bf J}^i$ is a sheaf on $\mathsf{T}({\Sigma}_{0,D}, w_\varphi)_0$, while $\tilde{\bf J}^i$ is a sheaf on $\mathsf{T}_\Sigma({\Sigma}_{0,D}, w_\varphi)_0$. However, we have the following result.
\begin{proposition}\label{p:curry}
    There are filtered isomorphisms in cohomology: 
    \[
    (H^p(\mathsf{T}({\Sigma}_{0,D}, w_\varphi)_0,{\bf J}^i), \mathsf{F}^\bullet) \cong (H^p(\mathsf{T}_\Sigma(\Sigma_{0,D}, w_\varphi)_0,\tilde{\bf J}^i),\mathsf{F}^\bullet).    
    \]
\end{proposition}
\begin{proof}
    In \cite[Theorem 7.3.9]{curry}, Curry shows that if $s : {\mathfrak{X}}'\rightarrow \mathfrak{X}$ is a subdivision of posets and $\mathbf{F}$ is a sheaf on $\mathfrak{X}$ and ${\bf F}' = s^*\mathbf{F}$ is its pullback then $H^*(\mathfrak{X}',{\bf F}') \cong H^*(\mathfrak{X},{\bf F})$. Since $s:\mathsf{T}_\Sigma(\Sigma_{0,D}, w_\varphi)_0 \rightarrow \mathsf{T}({\Sigma}_{0,D}, w_\varphi)_0$ is a subdivision of tropical polyhedral complexes, it induces a corresponding subdivision of posets. We need to show that $s^*{\bf J}^i = \tilde{\bf J}^i$.

    In \cite[Definition 5.1.3]{curry} and the following discussion, Curry explains that the naive formula for the pullback of sheaves on posets is in fact correct: $(s^*{\bf F})({\bm x}) = {\bf F}(s({\bm  x}))$ and $({s^*{\bf F}})({\bm{x}_1, \bm{x}_2}) = {\bf F}({s(\bm{x}_1),s(\bm{x}_2)})$. The result then follows by comparing the definitions of ${\bf J}^i$ and $\tilde{\bf J}^i$.
\end{proof}

\subsection{The tropical resolution of the twisted nearby cycles sheaf}\label{s:twistednearby}
Our next step is to construct a tropical resolution of the twisted relative de Rham complex $\Lambda^\bullet$ by using the tropical stratification. 

Following Proposition \ref{p:curry}, we start with a quasistable degeneration of Landau--Ginzburg models $(T(\Sigma), D, w_\varphi, \pi)$. We fix an orientation on $\mathrm{Trop}(\Sigma_{D,0})$. This orientation is obtained from a pair of orientations; an orientation on the cones of $\Sigma_0$ which we denote $[c:c']$, and an orientation on $\Sigma_{1}$ which we denote $[\sigma:\sigma']$. The total orientation is then obtained by $[(c,\sigma):(c',\sigma')] = [c:c'][\sigma:\sigma']$ where $[c:c], [\sigma:\sigma]$ are taken to be 1 and $(-1)$ respectively.
\begin{lemma}\label{p:tropresn-cpt}
    There is an $F_\mathrm{irr}$-filtered isomorphism of complexes;
    \[
    \Lambda^\bullet_{c} \cong {\bf s\Lambda}_c^\bullet := {\bf s}\left[\bigoplus_{\substack{\dim {\sigma} - \dim c = 0 \\ \mathrm{rec}(\sigma)=c}} \Lambda^\bullet_{(c,\sigma)}  \longrightarrow \bigoplus_{\substack{\dim {\sigma} - \dim c = 1 \\ \mathrm{rec}(\sigma)=c}} \Lambda^\bullet_{(c,\sigma)} \longrightarrow  \dots \right]
    \]
    where each map is a sum of the maps $[\sigma:\sigma']\cdot\alpha((c,\sigma),(c,\sigma'))$.
\end{lemma}
\begin{proof}
    For any orbifold normal crossings variety $X = \cup_{i \in  I} X_i$ there is a  resolution of sheaves 
    \[
    \mathcal{O}_X \longrightarrow \bigoplus_{|I| = 1} \mathcal{O}_{X_I} \longrightarrow \bigoplus_{|I| = 2} \mathcal{O}_{X_I}      \longrightarrow \dots 
    \]
    where the morphisms making up the complex are signed-sums of pullbacks, where the signs obey a signed incidence relation. The variety $T(c)_0$ has orbifold normal crossings with irreducible components given by $T({\sigma})$ where $\sigma$ are cells in $\mathrm{PC}(w_\varphi,\Sigma_0)$ whose recession cone is $c$. Thus we have a resolution of complexes 
    \[
    \mathcal{O}_{T(c)_0} \longrightarrow \bigoplus_{\substack{\dim {\sigma} - \dim c = 0 \\ \mathrm{rec}(\sigma)= c}} \iota_{T(\sigma)*}\mathcal{O}_{T({\sigma})}    \longrightarrow \bigoplus_{\substack{\dim {\sigma} - \dim c = 1 \\ \mathrm{rec}(\sigma)=c}} \iota_{T(\sigma)*}\mathcal{O}_{T({\sigma})}    \longrightarrow  \dots 
    \]
    Taking the tensor product with the vector bundle  $\mathcal{O}_{T(c)}(kP_c) \otimes ( \wedge^\bullet {c}^\perp)$ preserves exactness and $(\iota_{T(\sigma)*} \mathcal{O}_{T(\sigma)}) \otimes \mathcal{O}_{T(c)}(kP_c) \cong  \iota_{T(\sigma)*}\mathcal{O}_{T({\sigma})}(kP_\sigma)$. The required result follows from these facts and \eqref{eq:trivsheaf} after taking the limit as $k\rightarrow\infty$.
\end{proof}

The tensor product $\Lambda^\bullet_{c} \otimes \mathcal{O}_{T(c')}$ where $c\subseteq c'$ commutes with the resolution in Proposition \ref{p:tropresn-cpt}. Note that $\Lambda_{(c,\sigma)}^\bullet \otimes \mathcal{O}_{T(c')} \cong \Lambda_{(c,\sigma+c')}$ where $\sigma+c'$ denotes the Minkowski sum of $\sigma$ and $c'$, if that Minkowski sum is a cell of $\mathrm{PC}(w_\varphi,\Sigma_0)$, and zero otherwise. The recession cone of $\sigma + c'$ is $c'$ if the recession cone of $\sigma$ is $c$. Thus we get a resolution,
\[
    \Lambda^\bullet_{c}\otimes \mathcal{O}_{T(c')} \longrightarrow \bigoplus_{\substack{\dim {\sigma} - \dim c' = 0 \\ \mathrm{rec}(\sigma) = c' \\ c\subseteq c'}} \Lambda^\bullet_{(c,\sigma)}  \longrightarrow \bigoplus_{\substack{\dim {\sigma} - \dim c' = 1 \\ \mathrm{rec}(\sigma) = c'\\c\subseteq c'}} \Lambda^\bullet_{(c,\sigma)} \longrightarrow  \dots 
\]
From the~\eqref{eq:trivsheaf}  we may deduce that 
\begin{equation}\label{e:nearbyf}
    \Lambda_{c}^\bullet \otimes \mathcal{O}_{T(c')}  \cong \mathcal{O}_{T(c')}(*P_{c'}) \otimes \wedge^\bullet c^\perp.
\end{equation}
So for $c_1 \subseteq c_2\subseteq  c'$ there are natural morphisms obtained from the inclusion $c_2^\perp \subseteq c_1^\perp$;
\[
\alpha(c_2,c_1):\Lambda_{c_2}^\bullet\otimes \mathcal{O}_{T(c')} \longrightarrow \Lambda_{c_1}^\bullet\otimes \mathcal{O}_{T(c')}.
\]
As noted above, there are also morphisms;
\[
\alpha((c_2,\sigma),(c_1,\sigma)) : \Lambda^\bullet_{(c_2,\sigma)} \longrightarrow \Lambda^\bullet_{(c_1,\sigma)}.
\]
Choosing compatible signed incidence relations we obtain commutative diagrams of complexes;
\[
\begin{tikzcd}
\Lambda_{c_2}^\bullet \otimes \mathcal{O}_{T(c')} \ar[r] \ar[d,"\alpha({c_1,c_2})"] & \bigoplus_{\dim\sigma - \dim c'=0}\Lambda^\bullet_{(c_2,\sigma)} \ar[r] \ar[d,"\bigoplus_{\dim\sigma -\dim c' = 1}\alpha"] & \dots \\
\Lambda_{c_1}^\bullet \otimes \mathcal{O}_{T(c')} \ar[r]  & \bigoplus_{\dim\sigma - \dim c'=0}\Lambda^\bullet_{(c_1,\sigma)} \ar[r]  & \dots
\end{tikzcd}
\]
 
\begin{lemma}\label{l:lesresn}
    For any choice of signed incidence relation between the subcones of $c' \in  \Sigma_0$ there is a $F_\mathrm{irr}$-filtered long exact sequence of complexes of sheaves
\begin{equation}\label{e:res-resn}
   \dots \longrightarrow \bigoplus_{\substack{c\subseteq c' \\ \dim c=1}} \Lambda_{c}^\bullet\otimes \mathcal{O}_{T(c')}\longrightarrow  \Lambda^\bullet_{0_N}\otimes \mathcal{O}_{T(c')} \longrightarrow \Lambda^{\bullet-d+\dim c'}_{c'} \longrightarrow 0 
\end{equation}
where the filtration $F_\mathrm{irr}$ on the rightmost term is shifted by $-(d-\dim c')$.
\end{lemma}
\begin{proof}
    Applying \eqref{e:nearbyf}, this reduces to the fact that there is a long exact sequence induced by signed inclusions
    \[
        \dots \longrightarrow \bigoplus_{\substack{c\subseteq c' \\ \dim c = 2}} \wedge^\bullet c^\perp \longrightarrow \bigoplus_{\substack{c\subseteq c' \\ \dim c=1}} \wedge^\bullet c^\perp  \longrightarrow \wedge^\bullet M \longrightarrow \wedge^{\bullet-d+\dim c'} (c')^\perp \longrightarrow 0.
    \]
    In particular, for $m\in c^\perp$, $d(m) = \oplus_{c''\subseteq c} [c:c'']i_{c''}(m)$ where $i_{c''}$ denotes inclusion $c^\perp \subseteq (c'')^\perp$. The fact that this sequence is a complex is then a consequence of the signed incidence relation. Exactness follows by a dimension count.
\end{proof}
Now suppose we have $c_1 \subseteq c_2$ of codimension $1$, then there is a well-defined residue morphism described in Section \ref{s:nbf}:
\[
\mathrm{Res}_{c_1,c_2}: \Lambda_{c_1}^\bullet \longrightarrow \Lambda_{c_2}^\bullet.
\]
Under the identifications in \eqref{e:nearbyf}, and letting $c_2 = c_1 + \mathbb{R}_{\geq 0}\rho$ for some primitive $\rho$, we may identify the residue map with a combination of restriction and contraction:
\begin{align*}
\mathrm{Res}_{c_1,c_2} : \mathcal{O}_{T(c_1)_0}(*P_{c_1}) \otimes\wedge^\bullet c_1^\perp & \longrightarrow \mathcal{O}_{T(c_2)_0}(*P_{c_2})\otimes\wedge ^\bullet c_2^\perp   \\
g \otimes \omega & \longmapsto g|_{T(c_2)_0} \otimes (\wedge \rho)^*(\omega).
\end{align*}
This map extends to a morphism between the complexes in \eqref{e:res-resn}. 
\[
\begin{tikzcd}
\dots \ar[r] &  \oplus_{\substack{c'\subseteq c_1\\ \dim c' =1}} \Lambda_{c'}^\bullet\otimes \mathcal{O}_{T(c_1)}\ar[d] \ar[r] & \Lambda_{0_N}^\bullet\otimes \mathcal{O}_{T(c_1)} \ar[r]\ar[d] & \Lambda_{c_1}^{\bullet- d+\dim c_1} \ar[r]\ar[d,"\mathrm{Res}_{c_1,c_2}"] & 0 \\
\dots \ar[r] &  \oplus_{\substack{c'\subseteq c_2\\ \dim c' =1}} \Lambda_{c'}^\bullet\otimes \mathcal{O}_{T(c_2)} \ar[r] & \Lambda_{0_N}^\bullet\otimes \mathcal{O}_{T(c_2)} \ar[r] & \Lambda_{c_2}^{\bullet-d+\dim c_2} \ar[r] & 0 
\end{tikzcd}
\]
Here, the vertical maps are direct sums of morphisms 
\begin{equation}\label{eq:resmaps2}
\bigoplus_{c'\subseteq c_1\subseteq c_2} R_{c_1,c_2}^{c'} : \bigoplus_{\substack{c'\subseteq c_1 \\ \dim c' = i}}\Lambda_{c'}^\bullet\otimes \mathcal{O}_{T(c_1)} \longrightarrow \bigoplus_{\substack{c'\subseteq c_2 \\ \dim c' = i}}\Lambda_{c'}^\bullet\otimes \mathcal{O}_{T(c_2)}
\end{equation}
is a direct sum of morphisms 
\[
R_{c_1,c_2}^{c'} : \Lambda_{c'}^\bullet\otimes \mathcal{O}_{T(c_1)} \longrightarrow \Lambda_{c'}^\bullet\otimes \mathcal{O}_{T(c_2)}
\]
which are simply restriction of functions in the second coordinate. We obtain the following result by applying the residue resolution (Proposition \ref{p:residue-resolution2}) after taking the tensor product with $\mathcal{O}_T$.
\begin{lemma}\label{l:bigres}
    There is a triple complex of sheaves which is $F_\mathrm{irr}$-filtered quasi-isomorphic to $\Lambda^\bullet$ which is depicted in Figure \ref{eq:resol}. The horizontal morphisms are direct sums of morphisms described in \eqref{e:res-resn} and the vertical morphisms are induced by a direct sum of the residue morphisms \eqref{eq:resmaps2} taken with respect to a signed incidence relation.
\end{lemma}
\begin{figure}

    \begin{tikzcd}
          & & & \Lambda_{0_N}^\bullet \ar[d] \\
         & & \bigoplus_{\substack{\dim c'=1}} \Lambda^\bullet_{c'} \ar[r]\ar[d] & \bigoplus_{\dim c'=1} \Lambda_{0_N}^\bullet \otimes \mathcal{O}_{T(c')}  \ar[d]\\ 
          & \bigoplus_{\substack{\dim c' = 2}} \Lambda^\bullet_{c'} \ar[r] \ar[d] & \bigoplus_{\substack{\dim c' = 2 \\ \dim c=1 \\ c\subseteq c'}} \Lambda^\bullet_{c} \otimes \mathcal{O}_{T(c')}  \ar[r]\ar[d] & \bigoplus_{\dim c' =2} \Lambda_{0_N}^\bullet \otimes \mathcal{O}_{T(c')} \ar[d]  \\
          \dots \ar[r] & \bigoplus_{\substack{\dim c'  =3 \\ \dim c = 2\\ c \subseteq c'}} \Lambda^\bullet_{c} \otimes \mathcal{O}_{T(c')} \ar[r] \ar[d] & \bigoplus_{{\substack{\dim c'=3 \\ \dim c=1 \\ c\subseteq c'}}} \Lambda_{c}^\bullet \otimes \mathcal{O}_{T(c')} \ar[r] \ar[d] & \bigoplus_{\dim c' =3} \Lambda_{0_N}^\bullet \otimes \mathcal{O}_{T(c')}  \ar[d] \\
           & \vdots & \vdots & \vdots & 
    \end{tikzcd}
    \caption{}\label{eq:resol}
    \end{figure}

    In particular, after a shift, the resolution described in  Lemma \ref{l:bigres} is the total complex of a double complex of the form
    \[
    \bigoplus_{\substack{\dim c' = \dim c \\ c\subseteq c'}} \Lambda_{c}^\bullet \otimes \mathcal{O}_{T(c')} \longrightarrow \bigoplus_{\substack{\dim c' = \dim c +1 \\ c\subseteq c'}} \Lambda_{c}^\bullet \otimes \mathcal{O}_{T(c')} \longrightarrow\bigoplus_{\substack{\dim c' = \dim c+2 \\ c\subseteq c'}} \Lambda_{c}^\bullet \otimes \mathcal{O}_{T(c')} \longrightarrow \dots
    \]
    Finally, each filtered complex object appearing in Figure \ref{eq:resol} may be replaced by the filtered complex ${\bf s\Lambda}^\bullet_{c}\otimes \mathcal{O}_{T(c')}$. This then becomes a complex whose $p$-th term is of the form
    \[
        \bigoplus_{a+b=p}\bigoplus_{\substack{\dim c' = \dim c +a \\\dim {\sigma} - \dim c' = b \\ \mathrm{rec}(\sigma) = c'\\c\subseteq c'}} \Lambda^\bullet_{(c,\sigma)} 
    \]
    The sum is taken over all possible $c,c'$ and $\sigma$ satisfying the given numerical conditions and so that $\mathrm{rec}(\sigma) \in \Sigma_{D,0}$. The expression above simplifies to 
    \[
    \bigoplus_{\substack{  (c,\sigma) \in  \mathsf{S}_{\mathsf{T},\mathsf{D}}^\mathrm{cpt}\\ \dim \sigma - \dim c = p }}\Lambda^\bullet_{(c,\sigma)}
    \]
    Taking care to keep track of the morphisms and signs appearing in this total complex, this proves the following result.
\begin{proposition}\label{lem:res-irregular}
    There is a $F_\mathrm{irr}^\bullet$-filtered resolution of $\Lambda^\bullet$ given by the total complex of the following double complex;
    \begin{equation*}
        \bigoplus_{\substack{(c,\sigma) \in \mathsf{S}_{\mathsf{T},\mathsf{D}}^\mathrm{cpt} \\ \dim \sigma - \dim c = 0}} \Lambda^\bullet_{(c,\sigma)} \longrightarrow \bigoplus_{\substack{(c,\sigma) \in \mathsf{S}_{\mathsf{T},\mathsf{D}}^\mathrm{cpt} \\ \dim \sigma - \dim c = 1}} \Lambda^\bullet_{(c,\sigma)} \longrightarrow \bigoplus_{\substack{(c,\sigma) \in \mathsf{S}_{\mathsf{T},\mathsf{D}}^\mathrm{cpt} \\ \dim \sigma - \dim c = 2}} \Lambda^\bullet_{(c,\sigma)} \longrightarrow \dots 
        \end{equation*}
        There is a choice of signed incidence relation on $\mathsf{S}_{\mathsf{T},\mathsf{D},0}$ so that the $i$th differential of this complex is the matrix whose entries are 
        \[
        [(c,\sigma): (c',\sigma')]\cdot \alpha((c,\sigma),(c',\sigma')).
        \]
        for all pairs $(c,\sigma)$ so that $\dim \sigma - \dim c =i$ and $\dim\sigma' -\dim c' = i+1$.
\end{proposition}
\begin{remark}
    All of the resolutions in this section can be viewed as termwise resolutions of the sheaves $\Lambda^p$ or even as resolutions of $\Omega^p_{T/\mathbb{A}^1}(\log D\cup T_0)\otimes \mathcal{O}_{T_0}$. These resolutions commute with $d + dw_\sigma$ and with $F^\bullet_\mathrm{irr}$ to produce resolutions of filtered complexes.
\end{remark}

Finally, we prove the main theorem of this section. 

\begin{theorem}\label{t:nearbyfibcoh}
    Suppose $(T(\Sigma),D,w_\varphi,\pi)$ is a quasi-stable degeneration of toric Landau--Ginzburg models as constructed in Proposition \ref{p:poly}, and that the monomial support of $w_\sigma$ is the vertex set of a simplex for all $\sigma$. Then:
    \begin{enumerate}
    \item Property \ref{c:E1deg} holds for $(T(\Sigma),D,w_\varphi,\pi)$.
    \item There is a filtered isomorphism 
    \[
        (\mathbb{H}^n(T,\Lambda^\bullet),F^\bullet_\mathrm{irr}) \cong \bigoplus_{p+q=n} (H^p(\mathsf{T}_\Sigma(\Sigma_{0,D},w_\varphi)_0,\tilde{\bf J}^q),\mathsf{F}^\bullet).
    \]
    \end{enumerate}
\end{theorem}
\begin{proof}
    According to Proposition \ref{lem:res-irregular}, $\mathbb{H}^n(T,F^\lambda_\mathrm{irr}\Lambda^\bullet)$ is the hypercohomology of 
    \begin{equation*}
        \bigoplus_{\substack{(c,\sigma) \in \mathsf{S}_{\mathsf{T},\mathsf{D}}^\mathrm{cpt} \\ \dim \sigma - \dim c = 0}} F^\lambda_\mathrm{irr}\Lambda^\bullet_{(c,\sigma)} \longrightarrow \bigoplus_{\substack{(c,\sigma) \in \mathsf{S}_{\mathsf{T},\mathsf{D}}^\mathrm{cpt} \\ \dim \sigma - \dim c = 1}} F^\lambda_\mathrm{irr}\Lambda^\bullet_{(c,\sigma)} \longrightarrow \bigoplus_{\substack{(c,\sigma) \in \mathsf{S}_{\mathsf{T},\mathsf{D}}^\mathrm{cpt} \\ \dim \sigma - \dim c = 2}} F^\lambda_\mathrm{irr}\Lambda^\bullet_{(c,\sigma)} \longrightarrow \dots.
        \end{equation*} 
        By Theorem \ref{t:as} and \eqref{e:trivial}, there is a functorial quasi-isomorphism $\mathbb{R}\Gamma(F^\lambda_\mathrm{irr}\Lambda^\bullet_{(c,\sigma)}) \cong \Gamma(F^\lambda_\mathrm{irr}\Lambda_{(c,\sigma)}^\bullet)$ and by \eqref{e:adolph-sperb} and Proposition~\ref{p:propertyholds} there is a functorial quasi-isomorphism between $\Gamma(F^\lambda_\mathrm{irr}\Lambda^\bullet_{(c,\sigma)})$ and the trivial complex with trivial differential whose $p$-th component is $\mathsf{F}^{\lambda-p}{\bf B}_{\tau_\sigma} \otimes \wedge^{p - \dim \tau_\sigma} c^\perp \otimes  \mathrm{Vol}(L(\tau_\sigma)) = \mathsf{F}^{\lambda}\tilde{\bf J}^p{(c,\sigma)}$. The rows of this double complex are of the form:  
    \begin{align}
        C_\mathrm{cell}^p&(\mathsf{F}^\lambda\tilde{\bf J}^\bullet):= \nonumber \\ & \left[ \bigoplus_{\substack{(c,\sigma) \in \mathsf{S}_{\mathsf{T},\mathsf{D}}^\mathrm{cpt} \\ \dim \sigma - \dim c = 0}} \mathsf{F}^\lambda\tilde{\bf J}^p{(c,\sigma)} \longrightarrow \bigoplus_{\substack{(c,\sigma) \in \mathsf{S}_{\mathsf{T},\mathsf{D}}^\mathrm{cpt} \\ \dim \sigma - \dim c = 1}} \mathsf{F}^\lambda\tilde{\bf J}^p{(c,\sigma)} \longrightarrow \bigoplus_{\substack{(c,\sigma) \in \mathsf{S}_{\mathsf{T},\mathsf{D}}^\mathrm{cpt} \\ \dim \sigma - \dim c = 2}} \mathsf{F}^\lambda\tilde{\bf J}^p{(c,\sigma)} \longrightarrow \dots   \right]\label{e:descent5}
        \end{align}
        where morphisms are matrices of maps whose entries are identified with the morphisms induced by $[(c,\sigma):(c',\sigma')]\alpha((c,\sigma),(c',\sigma'))$. By the discussion in Section \ref{s:tsheaf}, specifically \eqref{e:cd}, these morphisms are simply $\tilde{\bf J}^p((c,\sigma),(c',\sigma'))$ when restricted to $\mathrm{im}(\xi_{(c,\sigma)})$ (see \eqref{e:adolph-sperb} for the notation). Consequently, they are graded morphisms and in particular, $\mathsf{F}^\lambda$-strict. Therefore, the inclusion $\mathsf{F}^\lambda\tilde{\bf J}^\bullet \hookrightarrow \tilde{\bf J}^\bullet$ induces an inclusion in cohomology. By the argument above, this inclusion is identified with the morphism $\mathbb{H}^n(T,F^\lambda_\mathrm{irr}\Lambda^\bullet) \rightarrow \mathbb{H}^n(T,\Lambda^\bullet)$. This proves the first statement.
        
        To prove the second statement, we identify \eqref{e:descent5} with $C_\mathrm{cell}^\bullet(\mathsf{T}(\Sigma_{0,D},w_\varphi)_0,\tilde{\bf J}^p)$. Therefore, $\mathbb{H}^n(T,\Lambda^\bullet)$ is isomorphic to the cohomology of 
        \[
        \bigoplus_{p} C^\bullet_\mathrm{cell}(\mathsf{T}(\Sigma_{0,D},w_\varphi)_0,\tilde{\bf J}^p)[p]
        \]
        which implies that $\mathbb{H}^n(T,\Lambda^\bullet) \cong \oplus_{p+q = n} H^q(\mathsf{T}(\Sigma_{0,D},w_\varphi)_0,\tilde{\bf  J}^p)$. The second statement then follows by application of Proposition~\ref{p:curry}. 
\end{proof}

Finally, we arrive at the main result of this section by combining Theorem \ref{t:nearbyfibcoh} and Proposition \ref{p:curry} with Proposition \ref{c:E1deg}.

\begin{corollary}\label{c:nearbyfibcoh}
Suppose $(T(\Sigma),D,w_\varphi,\pi)$ is a simplicial quasi-stable degeneration of toric Landau--Ginzburg models and choose some $t$ so that $1\gg |t| > 0$. Then 
\[
\dim \gr_{F_\mathrm{irr}}^\lambda H^n(T(\Sigma)_t\setminus D_t,w_{\varphi,t}) = \sum_{p+q=n} \gr_{\mathsf{F}}^{\lambda}H^q(\mathsf{T}(\Sigma_{0,D},w_\varphi)_0,{\bf J}^p).
\]
\end{corollary}

\begin{remark}[Orbifold cohomology]\label{r:relinertia}
	 It is straightforward to generalize Theorem \ref{t:nearbyfibcoh} and Corollary \ref{c:nearbyfibcoh} to the orbifold cohomology. For a quasistable degeneration of Landau--Ginzburg models $(X, D, w, \pi)$, the relevant twisted sectors are parametrized by a relative inertia stack $\mathfrak{I}_{\pi,X}$ that is a substack of $\mathfrak{I}_{X}$ defined by the union of components over which the restriction of $\pi$ is flat. Geometrically, each component in $\mathfrak{I}_{\pi, X}$ corresponds to a quasi-stable degeneration of the corresponding twisted sector, hence we apply the same argument above to compute the irregular Hodge numbers of the nearby fibre in each twisted sector. 
\end{remark}

\subsection{Application to Clarke dual pairs}\label{s:apptoClarke}
    The very last observation that we need to make is that the constructions in this section may be applied in the case of Clarke dual pairs of Landau--Ginzburg models. Recall that we are given a pair $({\bf \Sigma},\check{\bf \Sigma})$ satisfying Definition \ref{d:adjectives}. We let $A = {\bf \check{\Sigma}}[1]\cup \{0\}$ and note that $\Delta_{\bf \check{\Sigma}}=\mathrm{Conv}(A)$. Since ${\bf  \check{\Sigma}}$ is quasiprojective and convex we are given a star triangulation based at $0_N$ by the function $\check{\varphi}$ (Proposition \ref{p:simplicial}). The condition that $\langle n, m\rangle \geq 0$ for $n \in \mathrm{Supp}(\check{\Sigma})$ and $m \in \mathrm{Supp}(\Sigma)$ implies that the simplicial fan ${\Sigma}$ is contained in $\mathrm{nc}(0_N)$. We would like ${\Sigma}$ to play the role of $\Sigma_{0,D}$ in this situation, which requires that we extend ${\Sigma}$ to a subfan of a refinement of $\mathrm{nf}(\Delta_{\bf \check{\Sigma}})$. 

\begin{proposition}
    Suppose $({\bf \Sigma},\check{\bf \Sigma})$ form a Clarke dual pair of fans. Then there is a simplicial fan $\Sigma_0$ so that 
    \begin{enumerate}
        \item $\Sigma_0$ is a refinement of $\mathrm{nf}(\Delta_{\check{\bf \Sigma}})$,
        \item $\Sigma$ is a subfan of $\Sigma_0$,
        \item $T(\Sigma_0) \setminus T(\Sigma)$ is the support of a divisor.
    \end{enumerate}
\end{proposition}

 \begin{proof}

 This construction is perhaps easiest to explain geometrically. Since $\Sigma$ is quasiprojective, the toric variety $T(\Sigma)$ may be compactified to a toric subvariety $\overline{T}$ of $\mathbb{P}^n$ for some integer $n$, and whose image is the complement of a collection of divisors. By blowing up repeatedly in singular strata, we may assume that $\overline{T}$ has at worst orbifold singularities. This blow-up avoids $T({\Sigma}) \subseteq \overline{T}$ because $T(\Sigma)$ itself has at worst orbifold singularities. The zero-locus of $w(\check{\bf \Sigma})$, denoted  $Z$, is not necessarily a nondegenerate hypersurface in $\overline{T}$. However, for each open toric stratum $B$ of $\overline{T}$, either $Z\cap B$ is smooth or $B\subseteq Z$ because coefficients for $w(\check{\bf \Sigma})$ are chosen generically. After iterated blow-up at those strata $B \subseteq Z$, we may assume $Z\cap B$ is smooth for all $B$. Let ${T}_0$ be such a variety. Since $\check{\Sigma} \subseteq \mathrm{nc}(0_M)$, the intersection of $Z$ and all open toric strata in $T({\Sigma})$ is smooth, so this iterated blow up again does not affect $T({\Sigma}) \subseteq \overline{T}$. Let ${\Sigma}_0$ denote the fan attached to ${T}_0$. The fact that $Z\cap B$ is smooth for all open toric strata is  equivalent to item (1). Item (2) holds because $T(\Sigma)$ is an  open subvariety of $T_0$.

    By construction, $T\setminus T(\Sigma)$ is a divisor. The blow up process does not  change this. Therefore, $T'\setminus T(\Sigma)$ is also a collection of divisors. Therefore, item (3) holds.
\end{proof}

\bibliographystyle{plain}
\bibliography{hom.bib}{}

\end{document}